\documentclass[a4paper,11pt]{amsart}

\pdfoutput=1

\usepackage[T1]{fontenc}
\usepackage{amssymb,amsfonts,amsmath,amsthm,mathabx}
\usepackage{url}
\usepackage{color}
\usepackage{setspace}
\usepackage{enumerate}
\usepackage{epsfig,graphicx,psfrag}
\usepackage{tikz}
\usepackage{asymptote}
\usepackage{pinlabel}
\usepackage[normalem]{ulem}

\input{xy}
\xyoption{all}
\objectmargin={3mm}

\newcounter{notes}%[page]   %Le 2eme argument fait reinitialiser les numeros de notes a chaque page
\newcommand{\ignore}[1]{}

\newtheorem{theorem}{Theorem}
\newtheorem{proposition}[theorem]{Proposition}
\newtheorem{corollary}[theorem]{Corollary}
\newtheorem{lemma}[theorem]{Lemma}
\newtheorem{sublemma}[theorem]{Sublemma}

\newtheorem{problem}[theorem]{Probem}

\newtheorem{question}[theorem]{Question}

\newtheorem{observation}[theorem]{Observation}

\theoremstyle{definition}
\newtheorem{definition}[theorem]{Definition}
\newtheorem{remark}[theorem]{Remark}

\newtheorem{notation}[theorem]{Notation}

\newtheorem{convention}[theorem]{Convention}

\newtheorem{assumption}[theorem]{Assumption}
\newtheorem*{claim*}{Claim}
\newtheorem*{acknowledge}{Acknowledgements}

\newtheoremstyle{theoremwithref}{}{}{\itshape}{}{\bfseries}{.}{.5em}{#1 #2 #3}
\theoremstyle{theoremwithref}

\newcommand{\C}{\mathbb{C}}
\newcommand{\R}{\mathbb{R}}
\newcommand{\Q}{\mathbb{Q}}
\newcommand{\Z}{\mathbb{Z}}
\newcommand{\N}{\mathbb{N}}

\newcommand{\PSL}{\mathrm{PSL}(2,\mathbb{R})}
\newcommand{\PSLk}{\mathrm{PSL}^k(2,\mathbb{R})}
\newcommand{\SO}{\mathrm{SO}}

\newcommand{\Hom}{\mathrm{Hom}}

\newcommand{\HOS}{\mathrm{Homeo}^+(S^1)}
\newcommand{\HOZ}{\mathrm{Homeo}_\Z^+(\R)}
\newcommand{\rot}{\mathrm{rot}}
\newcommand{\rotild}{\widetilde{\mathrm{ro}}\mathrm{t}}
\newcommand{\eu}{\mathrm{eu}}
\newcommand{\id}{\mathrm{id}}
\newcommand{\Fix}{\mathrm{Fix}}
\newcommand{\Per}{\mathrm{Per}}
\newcommand{\AF}{\mathit{A}}
\newcommand{\bbrac}[1]{
  \left(\mkern-5mu\left(#1\right)\mkern-5mu\right)}

\DeclareMathOperator{\Homeo}{Homeo}

\title[Rigidity and geometricity]{Rigidity and geometricity for surface group actions on the circle}

\author{Kathryn Mann}
\address{Department of Mathematics, Brown University, 151 Thayer Street, Providence, RI 02912, USA
}
\email{mann@math.brown.edu}

\author{Maxime Wolff}
\address{Sorbonne Universit\'es, UPMC Univ.\ Paris 06, Institut de Math\'ematiques
de Jussieu-Paris Rive Gauche, UMR 7586, CNRS, Univ. Paris Diderot, Sorbonne
Paris Cit\'e, 75005 Paris, France}
\email{maxime.wolff@imj-prg.fr}

\begin{document}

\maketitle
\numberwithin{theorem}{section}
\numberwithin{equation}{section}

\begin{abstract}

We prove that rigid representations of $\pi_1\Sigma_g$ in $\HOS$
are geometric, thereby establishing a converse statement of a theorem
by the first author in~\cite{KatieInvent}. 

\vspace{0.2cm}

\end{abstract}

%%%%%%%%%%%%%%%%%%%%%%%%%%%%%%%%%%%%%%%%%%%%%%%%%%%%%%%%%%%%%%%%%%%%%%
%%%%%%%%%%%%%%%%%%%%%%%%%%%%%%%%%%%%%%%%%%%%%%%%%%%%%%%%%%%%%%%%%%%%%%
%%%%%%%%%%%%%%%%%%%%%%                    %%%%%%%%%%%%%%%%%%%%%%%%%%%%
%%%%%%%%%%%%%%%%%%%%%%  I. INTRODUCTION   %%%%%%%%%%%%%%%%%%%%%%%%%%%%
%%%%%%%%%%%%%%%%%%%%%%                    %%%%%%%%%%%%%%%%%%%%%%%%%%%%
%%%%%%%%%%%%%%%%%%%%%%%%%%%%%%%%%%%%%%%%%%%%%%%%%%%%%%%%%%%%%%%%%%%%%%
%%%%%%%%%%%%%%%%%%%%%%%%%%%%%%%%%%%%%%%%%%%%%%%%%%%%%%%%%%%%%%%%%%%%%%

\section{Introduction}

%%%%%%%%%%%%%%%%%%%%%%%%%%%%%%%%%%%%%%%%%%%%%%%%%%%%%%%%%%%%%%%%%%%%%%
%%%%%%%%%%%%%%%%%%%%%%%%%%%%%%%%%%%%%%%%%%%%%%%%%%%%%%%%%%%%%%%%%%%%%%
%%%%%%%%%%%%%%%%%%%%%%%%%%%%%%%%%%%%%%%%%%%%%%%%%%%%%%%%%%%%%%%%%%%%%%

\subsection{Character spaces and rigidity}
Let $\Gamma$ be a discrete group and $G$ a topological group. If
$G \subset \Homeo(X)$ for some space $X$, then the {\em representation space}
$\Hom(\Gamma,G)$, equipped with the compact-open topology, parameterizes
actions of $\Gamma$ on $X$ with image in $G$.
Typically, $G$ is used to specify the regularity of the action -- for instance,
taking $G = \mathrm{Diff}(X)$ parametrizes smooth actions, and if $G$ is a 
Lie group acting transitively on $M$ these are \emph{geometric} actions in the sense of Ehresmann.
Since conjugate actions are dynamically equivalent, 
the appropriate moduli space of actions is
the quotient $\Hom(\Gamma,G)/G$ under the natural conjugation action of $G$.
However, this quotient space is typically non-Hausdorff.   

When $G$ is a Lie group and $\Hom(\Gamma,G)$ is an affine variety,
algebraic geometers solve this problem by considering the 
quotient $\Hom(\Gamma,G)/\!/G$ from geometric invariant theory.
In the special case where $G$ is a semi-simple complex reductive Lie group, 
this GIT quotient is simply the quotient of $\Hom(\Gamma,G)$ by the equivalence relation
$\rho_1 \sim \rho_2$ whenever {\em the closures} of their conjugacy classes
intersect \cite{Luna1, Luna2}.  
In particular, this relation makes the quotient space Hausdorff.
In the well-studied case of
$G=\mathrm{SL}(n,\C)$,
the GIT quotient agrees with the space of \emph{characters} of
$G$-representations, motivating the terminology in the following definition.
\begin{definition}
  For any discrete group $\Gamma$ and topological group $G$,
  the {\em character space} $X(\Gamma, G)$ is the largest Hausdorff
  quotient\footnote{Recall the largest Hausdorff quotient $X_H$ of a
  topological space $X$ is a space with the universal property that any
  continuous map $f\colon X\to Y$ from $X$ to a Hausdorff topological space
  factors canonically through the projection $X \to X_H$.
  One construction of $X_H$ is as the quotient of $X$ by the intersection
  of all equivalence relations $\sim$ such that $X/\sim$ is Hausdorff.}
  of $\Hom(\Gamma, G)/G$.
  We say that two representations are {\em $\chi$-equivalent}
  if they give the same point in $X(\Gamma,G)$.
\end{definition}
A representation $\rho: \Gamma \to G$ is \emph{rigid}, loosely speaking,
if all deformations of $\rho(\Gamma)$ in $G$ are trivial.   This notion can be made
precise in the setting of character spaces, as follows. 
\begin{definition}
  A representation $\rho \in \Hom(\Gamma, G)$ is \emph{rigid} if the image of $\rho$ is an
  isolated point in the character space $X(\Gamma, G)$.
\end{definition}
This is a strong condition on $\rho$, and can be loosened to more explicit conditions. In particular, we will say that $\rho$ is
\emph{path-rigid} if the path component of $\rho$ in $\Hom(\Gamma, G)$ is
contained in a single $\chi$-equivalence class.

The case of interest in this article is when $G=\HOS$, the group of
orientation-preserving homeomorphisms of the circle, and
$\Gamma = \Gamma_g = \pi_1(\Sigma_g)$ is the fundamental group of an orientable
surface of genus $g \geq 2$.
In this case $\Hom(\Gamma, G)$ has an important interpretation as the space of
\emph{flat} or \emph{foliated} topological circle bundles over $\Sigma$.
As will be explained in Section \ref{ssec:CharSpace}, in this setting the
character space $X(\Gamma, G)$ is the space of
\emph{semi-conjugacy} classes
of actions of $\Gamma$ on $S^1$, and path-rigid representations are those $\rho$ such that every path through $\rho$
can be obtained by a continuous family of semi-conjugacies.   Both rigid and path-rigid representations
can be thought of as corresponding to foliated bundles that admit only trivial types of deformations.  

A second motivation for the study of $X(\Gamma_g, \HOS)$ comes from Goldman's
seminal work on $X(\Gamma_g, \PSL)$ and its relations with Teichm\"uller
spaces.
In \cite{Goldman88}, Goldman showed that the connected components of
$X(\Gamma_g,\PSL)$ are classified by the \emph{Euler number}.  The Euler number is a characteristic integer, which   
Milnor \cite{Milnor} showed takes values in $[-2g+2, 2g-2] \cap \Z$ on
(equivalence classes of) representations in $X(\Gamma_g, \PSL)$.  
This is the famous \emph{Milnor--Wood} inequality; to which Wood's
contribution was an extension of Milnor's result to representations into
$\HOS \supset \PSL$ \cite{Wood}.    However, as was shown in \cite{KatieInvent}, the 
Euler number does not classify connected components of $X(\Gamma_g, \PSL)$, and 
extending Goldman's work to representations into $\HOS$ appears to be a difficult task.  
We will comment further on this in the next section. 

%%%%%%%%%%%%%%%%%%%%%%%%%%%%%%%%%%%%%%%%%%%%%%%%%%%%%%%%%%%%%%%%%%%%%%
%%%%%%%%%%%%%%%%%%%%%%%%%%%%%%%%%%%%%%%%%%%%%%%%%%%%%%%%%%%%%%%%%%%%%%
%%%%%%%%%%%%%%%%%%%%%%%%%%%%%%%%%%%%%%%%%%%%%%%%%%%%%%%%%%%%%%%%%%%%%%

\subsection{Geometric representations}

The first known example of a rigid representation of a surface group into
$\HOS$ comes from a celebrated theorem of Matsumoto \cite{Matsumoto87}.
He showed that the set of representations with maximal Euler number, i.e. Euler number equal to $2g-2$, 
in $X(\Gamma_g,G)$ consists of a single
point. As the Euler number is a continuous function on $\Hom(\Gamma_g, G)$,
this implies that representations of maximal Euler number are rigid.
(The same holds for representations with Euler number $-2g+2$.)

Phrased otherwise, Matsumoto's result says that all 
representations with Euler number $\pm(2g-2)$ are
$\chi$-equivalent to discrete, faithful representations into $\PSL$.
This hints at an underlying phenomenon for rigidity, namely that these 
are representations coming from a geometric structure.

\begin{definition}[\cite{KatieSurvey}]
  Let $M$ be a manifold, and $\Gamma$ a countable group. A representation
  $\rho\colon\Gamma\to\Homeo(M)$ is called {\em geometric} if
  it is $\chi$-equivalent to a 
  faithful representation with image a cocompact lattice in a transitive,
  connected Lie group $G\subset\Homeo(M)$.
\end{definition}
It is not difficult to classify the geometric representations of surface
groups in $\HOS$: up to $\chi$-equivalence,
all are either discrete, faithful representations into $\PSL$, or obtained by lifting such
a representation to a finite cyclic extension of $\PSL$. See \cite{KatieSurvey} for details.

The main result of \cite{KatieInvent} is the following.
\begin{theorem}[Mann \cite{KatieInvent}]\label{theo:KatieGeomRigid}
  In the space $\Hom(\Gamma_g,\HOS)$, all geometric
  representations are rigid.
\end{theorem}

In fact, the main theorem of \cite{KatieInvent} is stated in a weaker form; 
it says that the connected component of $\Hom(\Gamma_g, \HOS)$ is a single semi-conjugacy, or $\chi$-equivalence, class (we will see soon that these two notions coincide).  
However, the proof of the theorem is carried out on the level of semi-conjugacy invariants of representations, so actually 
proves the stronger result that geometric representations descend to isolated points in $X(\Gamma_g, \HOS)$.   

%%%%%%%%%%%%%%%%%%%%%%%%%%%%%%%%%%%%%%%%%%%%%%%%%%%%%%%%%%%%%%%%%%%%%%
%%%%%%%%%%%%%%%%%%%%%%%%%%%%%%%%%%%%%%%%%%%%%%%%%%%%%%%%%%%%%%%%%%%%%%
%%%%%%%%%%%%%%%%%%%%%%%%%%%%%%%%%%%%%%%%%%%%%%%%%%%%%%%%%%%%%%%%%%%%%%

\subsection{Results}

This article is devoted to proving the converse of
Theorem~\ref{theo:KatieGeomRigid}. We show the following.
\begin{theorem}\label{theo:RigidGeom}
  Every rigid representation in $\Hom(\Gamma_g,\HOS)$ is geometric.
\end{theorem}
In other words, the \emph{only} source of rigidity for actions of $\Gamma_g$
on $S^1$ is the existence of an underlying geometric structure.
Our main technical result in the course of proving
Theorem~\ref{theo:RigidGeom} is stronger for representations of non-zero Euler class, 
as we need to assume only path-rigidity in this case.

\begin{theorem}\label{theo:Main}
  Let $\rho\colon\pi_1\Sigma_g\to\HOS$ be a path-rigid representation.
  If $\rho$ is not geometric, then its Euler class is zero, and there exists
  a one-holed, genus $g-1$ subsurface $\Sigma' \subset \Sigma_g$ such that
  $\rho|_{\pi_1 \Sigma'}$ has a finite orbit.
\end{theorem}
The condition of having a large subsurface with a finite orbit makes it
very unlikely that such a representation cannot be deformed along a path.
This gives strong evidence for the fact that all path-rigid representations
should in fact be geometric.
However, at the time of writing we are unable to prove this.

While the proof of Theorem~\ref{theo:Main} is quite long, a much simpler argument can be carried out
under the additional assumption that the relative Euler number on some 
genus $1$ subsurface is equal to $1$ (this is the case in particular for
representations of Euler class $\geq g$). This simpler, though much weaker, proof is presented in the companion
article~\cite{AuMoinsG}.
Although the present article is self-contained, the reader may prefer to
take \cite{AuMoinsG} as a starting point.  

We conclude this introduction by putting our result in the perspective of
the following ambitious problem (which remains wide open) in the natural
continuation of Goldman's work on $X(\Gamma_g, \PSL)$.

\begin{question}\label{q:CC}
  What are the connected components of $X(\Gamma_g,\HOS)$?
  What are its path components?
\end{question}
One of the implications of Theorem~\ref{theo:KatieGeomRigid} was that the space
$X(\pi_1\Sigma_g,\HOS)$ has strictly more connected components than
$X(\pi_1\Sigma_g,\mathrm{PSL}(2,\R))$.  At this time, we do not even know if
the former has finitely many connected components. The present article,
more than simply proving Theorem~\ref{theo:RigidGeom}, aims at providing
some technical tools towards this question; we hope to address it in 
future work.

%%%%%%%%%%%%%%%%%%%%%%%%%%%%%%%%%%%%%%%%%%%%%%%%%%%%%%%%%%%%%%%%%%%%%%
%%%%%%%%%%%%%%%%%%%%%%%%%%%%%%%%%%%%%%%%%%%%%%%%%%%%%%%%%%%%%%%%%%%%%%
%%%%%%%%%%%%%%%%%%%%%%%%%%%%%%%%%%%%%%%%%%%%%%%%%%%%%%%%%%%%%%%%%%%%%%

\subsection{Strategy of the proof and organization of the article}

The main ingredient in the
proof of Theorem \ref{theo:Main} is the effect of {\em bending deformations}
on the periodic sets of simple closed curves.
Bending deformations are classical in (higher) Teichm\"uller theory, 
(see paragraph \ref{sss:bending} for a reminder); 
and we extend their study to representations to $\HOS$.  We now outline the major steps. 

\noindent \textit{Step 1: local-to-global.}
Our proof starts by making a strong additional technical hypothesis on representations that
forces them to look ``locally'' (i.e. on the level of some pairs of curves) like representations 
into $\PSLk$.  Specifically, we say that the action of two
elements $a,b \in \Gamma_g$ representing standard generators of a one-holed torus subsurface of $\Sigma_g$
satisfies $S_k(a,b)$ if $\rho(a)$ and $\rho(b)$ are separately conjugate to hyperbolic elements of $\PSLk$, and 
their fixed points alternate around the circle.  We show the following.
\begin{theorem}\label{theo:SkRigid}
  Let $\rho$ be a path-rigid, minimal representation, and suppose furthermore
  that there exists $k\geq 1$ such that $S_k(a,b)$ holds for all standard generators of one-holed torus subsurfaces. 
  Then $\rho$ is geometric.
 \end{theorem}

The proof of Theorem \ref{theo:SkRigid} uses
bending deformations of $\rho$
to move the periodic points of generators of $\pi_1\Sigma_g$; 
provided $\rho$ is path-rigid, we are able to conclude 
the periodic points of many simple closed curves are in the same cyclic order as if $\rho$ were geometric.
In the companion article \cite{AuMoinsG}, (whose additional hypothesis
guarantees that $k=1$) this same process was sufficient to demonstrate that $\rho$ has maximal
Euler number, hence is geometric.
Here in the general case, we need to use a more sophisticated tool, and
use Matsumoto's theory of \emph{Basic Partitions}
(see Section~\ref{subsec:Matsumoto}). \smallskip

\noindent \textit{Step 2: good and bad tori.} 
We next make extensive use of bending deformations to prove the following.
\begin{proposition}\label{prop:RatRotIntro}
  If a representation $\pi_1\Sigma_g\rightarrow G$ is path-rigid then all
  nonseparating simple closed curves have rational rotation number.
\end{proposition}
\begin{theorem}\label{theo:PerDisjointsSkIntro}
  Suppose $\rho$ is path-rigid and minimal. Then, for all standard generators
  $a,b$ of one-holed subsurfaces,
  we have the implication
  \[ \Per(\rho(a))\cap\Per(\rho(b))=\emptyset \Rightarrow S_k(a,b)\text{ for some }k. \]
\end{theorem}

The upshot of these results is that, if a path-rigid and minimal representation
\emph{fails} to be geometric, then many curves are forced to have common
periodic points.
Common periodic points hint at the existence of a finite orbit for $\rho$, so we next look for a finite orbit in order to derive a
contradiction (indeed, representations with a finite orbit are easily seen
to be non-path-rigid).
This idea proves difficult to implement, 
so we search first for curves
with rotation number zero, as the dynamics of these are easier to control. 
This search can be performed separately in every one-holed torus in the surface, where
the action of the mapping class group is simple to work with. Accordingly,
a one-holed torus in $\Sigma_g$ is called a {\em good torus} if it
contains a nonseparating simple loop with rotation number zero; otherwise
we say it is a {\em bad torus}. Further, a one-holed torus is called {\em very good} if its fundamental
group has a finite orbit in $S^1$. We  prove:
\begin{proposition}\label{prop:BadToriIntro}
  Let $\rho$ be path-rigid.
  Suppose that $\Sigma_g$ contains a bad torus $\Sigma'$.
  Then its complement $\Sigma''$ contains only very good tori.
\end{proposition}

\begin{proposition}\label{prop:NoTwoJustGoodIntro}
  Let $\rho$ be path-rigid, and non-geometric.
  Then there cannot exist two disjoint
  good tori that are not very good.
\end{proposition}
\begin{theorem}\label{theo:VeryGoodFiniteOrbitIntro}
  Let $\rho$ be a path-rigid representation.
  Let $\Sigma_{g',1}$ be a subsurface in which all tori are very good.
  Then $\pi_1\Sigma_{g',1}$ has a finite orbit.
\end{theorem}
These three last statements show that if $\rho$ is a path-rigid and
non-geometric representation, then it has a subsurface of genus $g-1$
with a finite orbit; the statement about the Euler class in
Theorem~\ref{theo:Main} then follows easily.
\smallskip

\noindent \textit{Conclusion.}
Provided $g\geq 3$, Theorem~\ref{theo:VeryGoodFiniteOrbitIntro} implies 
that if $\rho$ is path-rigid and non-geometric, then there exist 
curves $a,b$, generating a torus subsurface of $\Sigma_g$,
such that $\rho(a)$ and $\rho(b)$ have a common fixed point. 
It then follows from a recent theorem of Alonso, Brum and Rivas \cite{ABR} 
that $\rho$ cannot be rigid.  However, path-rigidity and the genus $g = 2$ case do not follow, so we prove an independent, 
simple lemma on rigid representations that shows all torus subsurfaces have only finitely many finite orbits. This
applies to all genera, and allows us conclude the proof of Theorem~\ref{theo:RigidGeom}.  

The article is organized as follows. Section~\ref{sec:Prelim}
introduces tools and frameworks that will be frequently used in the proof.  
We review background and prove new results on complexes of based curves; then prove a series of results on the movement of periodic sets under
specific bending deformations; and finally discuss character spaces, semi-conjugacy, and the Euler class.  
In Section~\ref{sec:Exercices} we prove Theorem~\ref{theo:SkRigid}.
In Section~\ref{sec:Per} we prove Proposition~\ref{prop:RatRotIntro}
and Theorem~\ref{theo:PerDisjointsSkIntro}.
The proof of Theorem~\ref{theo:Main} is then completed in
Section~\ref{sec:Section5}.
Finally, in Section~\ref{sec:Fin} we complete the proof of Theorem~\ref{theo:RigidGeom}
and state some open questions and directions for further work.

\begin{acknowledge}
This work was started at MSRI during spring 2015 at a program supported by NSF grant 0932078.
Both authors also acknowledge support of the U.S. National Science Foundation grants
DMS 1107452, 1107263, 1107367 ``RNMS: Geometric Structures and Representation
Varieties'' (the GEAR network).
K. Mann was partially supported by NSF grant DMS-1606254, and thanks the Inst. Math. Jussieu and Foundation Sciences Math\'ematiques de Paris.   Parts of this work were written as M. Wolff was visiting the Institute for Mathematical Sciences, NUS, Singapore, and the Universidad de la Rep\'ublica, Montevideo, Uruguay; he wants to thank them for
their great hospitality.
\end{acknowledge}

%%%%%%%%%%%%%%%%%%%%%%%%%%%%%%%%%%%%%%%%%%%%%%%%%%%%%%%%%%%%%%%%%%%%%%
%%%%%%%%%%%%%%%%%%%%%%%%%%%%%%%%%%%%%%%%%%%%%%%%%%%%%%%%%%%%%%%%%%%%%%
%%%%%%%%%%%%%%%%%%%%%%                     %%%%%%%%%%%%%%%%%%%%%%%%%%%
%%%%%%%%%%%%%%%%%%%%%%  II. PRELIMINAIRES  %%%%%%%%%%%%%%%%%%%%%%%%%%%
%%%%%%%%%%%%%%%%%%%%%%                     %%%%%%%%%%%%%%%%%%%%%%%%%%%
%%%%%%%%%%%%%%%%%%%%%%%%%%%%%%%%%%%%%%%%%%%%%%%%%%%%%%%%%%%%%%%%%%%%%%
%%%%%%%%%%%%%%%%%%%%%%%%%%%%%%%%%%%%%%%%%%%%%%%%%%%%%%%%%%%%%%%%%%%%%%

\section{Preliminaries}\label{sec:Prelim}

%\textcolor{blue}{
%This section sets notation and develops a toolkit for use in the main proofs.  
%The first part treats curves on surfaces.  This subsection will feel familiar to low
%dimensional topologists, except that we will give much more attention to
%{\em based} curves than is usually present in the literature.
%The second subsection deals with actions of surface groups on $S^1$ and
%their deformations.   We introduce new material on the behavior of periodic sets under deformations, and the topology of sets of persistent (and non-persistent) periodic points; this will be crucial in later sections of the work.  The third subsection covers character spaces, semi-conjugacy, and the Euler number in more detail, including the proof that $\chi$-equivalence coincides with semi-conjugacy in $\Hom(\Gamma, \HOS)$.}\note{M: je pense qu'on peut enlever ce
%paragraphe, quitte \`a grossir un peu le dernier paragraphe de la section 1.  K: D'accord !}

%%%%%%%%%%%%%%%%%%%%%%%%%%%%%%%%%%%%%%%%%%%%%%%%%%%%%%%%%%%%%%%%%%%%%%
%%%%%%%%%%%%%%%                                   %%%%%%%%%%%%%%%%%%%%
%%%%%%%%%%%%%%%            SURFACES               %%%%%%%%%%%%%%%%%%%%
%%%%%%%%%%%%%%%                                   %%%%%%%%%%%%%%%%%%%%
%%%%%%%%%%%%%%%%%%%%%%%%%%%%%%%%%%%%%%%%%%%%%%%%%%%%%%%%%%%%%%%%%%%%%%

\subsection{Based curves on surfaces}\label{ssec:Marquage}

This subsection should seem familiar to low
dimensional topologists, except that we will give much more attention to
{\em based} curves than is usually present in the literature.
As in the introduction, we use the notation $\Gamma_g = \pi_1\Sigma_g$.
While this notation omits mention of a basepoint, all elements of $\Gamma_g$ are always assumed based.
This is crucial -- for example, we recall, as a warning, that the set of {\em based} simple closed
loops, in $\Gamma_g$, is not closed under conjugation.   We now set some conventions.  

Since we are interested in actions of $\Gamma_g$ by homeomorphisms on
the circle, we will write words in $\Gamma_g$ (ie, products of loops by
concatenation) from right to left, in the same order as composition of
homeomorphisms. We also fix the commutator notation to be
$[a,b]:=b^{-1}a^{-1}ba$.

The based curves $(a_1,b_1,\ldots,a_g,b_g)$, depicted in Figure~\ref{fig:BaseStandard}, are called a {\em standard system of loops}, and give the following standard presentation
of~$\Gamma_g$:

\[
\Gamma_g=\langle a_1,b_1,\ldots,a_g,b_g\,|\,[a_g,b_g]\cdots[a_1,b_1]=1\rangle.
\]
\begin{figure}[hbt]
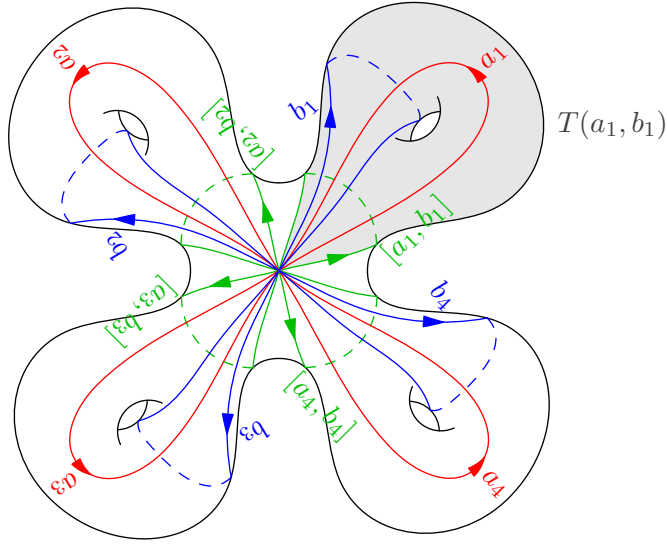

\begin{asy}
  import geometry;
  unitsize(1pt);
  
  picture anse0, anse1, anse2, anse3, anse4;
  
  real r = 30, R = 110, ct = 70; //% petit rayon, grand rayon, abcisse du centre du trou
  point p1 = 74*dir(-29), p2 = 74*dir(32);
  path GC = r*dir(-45){dir(45)}..p1..(R,0){up}..p2..r*dir(45){dir(135)}; //% grand contour
  path BT = (ct-10,4)..(ct,-3)..(ct+10,4); //% bas du trou
  path HT = relpoint(BT,0.2)..(ct,2.6)..relpoint(BT,0.8);
  path c11 = (0,0){dir(30)}..relpoint(GC,0.96){dir(40)}; //% debut du chemin c_1
  path c12 = relpoint(GC,0.96){dir(-45)}..relpoint(GC,0.04){dir(-135)}; //% milieu du chemin c_1
  path c13 = relpoint(GC,0.04){dir(160)}..(0,0){dir(150)};
  path b1 = (0,0){dir(-12)}..(ct,-16)..(95,0)..(ct,16)..(0,0){dir(186)};
  path a11 = (0,0){dir(20)}..30*dir(20)..relpoint(GC,0.8);
  path a12 = relpoint(GC,0.8){dir(0)}..(ct+3,10)..relpoint(HT,0.5){dir(-150)};
  path a13 = relpoint(HT,0.5){dir(150)}..35*dir(5)..(0,0){dir(185)};
  
  path d11 = (0,0){dir(-7)}..40*dir(-7)..relpoint(BT,0.3); //% debut du chemin delta_1
  path d12 = relpoint(BT,0.3){dir(-45)}..(ct-8,-23){left}..20*dir(-45){dir(-135)}..ct*dir(-96){down}..(-I)*relpoint(GC,0.6);
  path d13 = (-I)*relpoint(GC,0.6){dir(50)}..50*dir(-70)..(0,0){dir(104)};
  
  path GrandContourT = buildcycle(c11, reverse(GC), c13);
  path PetitContourT = buildcycle(BT,reverse(HT));
  fill(anse0, GrandContourT, lightgrey); fill(anse0, PetitContourT, white);
  
  draw(anse1, GC); draw(anse1, BT); draw(anse1, HT); draw(anse1, c11, heavygreen);
  draw(anse1, c12, heavygreen+dashed); draw(anse1, reverse(c13), heavygreen, Arrow(Relative(0.7)));
  draw(anse1, b1, red, Arrow(Relative(0.52)));
  draw(anse1, a11, blue, Arrow(Relative(0.8)));
  draw(anse1, a12, blue+dashed); draw(anse1, a13, blue);
  add(anse2, anse1); add(anse3, anse1); add(anse4, anse1);
  
  label(anse1, "\small $a_1$",(103,0),red);
  label(anse1, "\small $b_1$",relpoint(a11,0.8),NW,blue);
  label(anse1, "\small $[a_1,b_1]$",relpoint(c13,0.1),SE,heavygreen);
  label(anse2, "\small $a_4$",(103,0),red);
  label(anse2, "\small $b_4$",relpoint(a11,0.8),NW,blue);
  label(anse2, "\small $[a_4,b_4]$",relpoint(c13,0.1),SE,heavygreen);
  label(anse3, "\small $a_3$",(103,0),red);
  label(anse3, "\small $b_3$",relpoint(a11,0.8),NW,blue);
  label(anse3, "\small $[a_3,b_3]$",relpoint(c13,0.1),SE,heavygreen);
  label(anse4, "\small $a_2$",(103,0),red);
  label(anse4, "\small $b_2$",relpoint(a11,0.8),NW,blue);
  label(anse4, "\small $[a_2,b_2]$",relpoint(c13,0.1),SE,heavygreen);
  
  add(scale(1.1)*rotate(45)*anse0,(0,0));
  add(scale(1.1)*rotate(45)*anse1,(0,0));
  add(scale(1.1)*rotate(-45)*anse2,(0,0));
  add(scale(1.1)*rotate(-135)*anse3,(0,0));
  add(scale(1.1)*rotate(135)*anse4,(0,0));
  label("\small $T(a_1,b_1)$", (125, 55), heavygrey);
\end{asy}
\caption{Standard generators on the genus $g$ surface ($g=4$)}
\label{fig:BaseStandard}
\end{figure}

We will make extensive use of systems of curves that look like those in Figure~\ref{fig:Chaine}.
\begin{figure}[hbt]
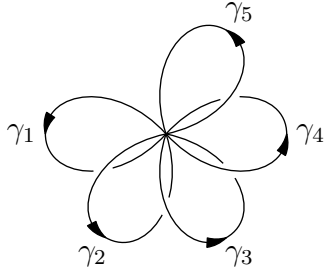

\begin{asy}
  import geometry;
  real t = 48;
  real T=60;
  //% Les courbes
  path courba = (0,0){dir(180-t)}..(-45,0)..(0,0){dir(t)};
  path courbb = rotate(T)*courba, courbc = rotate(T)*courbb;
  path courbd = rotate(T)*courbc, courbe = rotate(T)*courbd;
  //% Les intersections
  point pointab = intersectionpoint(subpath(courba,0.1,1.9),courbb);
  point pointbc = intersectionpoint(subpath(courbb,0.2,1.8),courbc);
  point pointcd = intersectionpoint(subpath(courbc,0.2,1.8),courbd);
  point pointde = intersectionpoint(subpath(courbd,0.2,1.8),courbe);
  //% On dessine
  draw("$\gamma_1$",courba,Arrow(Relative(0.5))); dot(pointab,7pt+white);
  draw("$\gamma_2$",courbb,Arrow(Relative(0.5))); dot(pointbc,7pt+white);
  draw("$\gamma_3$",courbc,Arrow(Relative(0.5))); dot(pointcd,7pt+white);
  draw("$\gamma_4$",courbd,Arrow(Relative(0.5))); dot(pointde,7pt+white);
  draw("$\gamma_5$",courbe,Arrow(Relative(0.5)));
\end{asy}
\caption{A directed chain of length $5$}
\label{fig:Chaine}
\end{figure}
Accordingly, we will say that a tuple $(\gamma_1,\ldots,\gamma_k)$ of elements
of $\Gamma_g$ is an {\em oriented, directed $k$-chain} if these elements of
$\Gamma_g$ can be realised by differentiable based loops,
$[0,1]\to\Sigma_g$, that do not intersect outside the base point, and with
cyclic order
$(\gamma_1'(0),\gamma_2'(0),-\gamma_1'(1),\gamma_3'(0),-\gamma_2'(1),\gamma_4'(0),\ldots,-\gamma_k'(1))$. 
In other words, an oriented, directed $k$-chain is a $k$-tuple of loops
arising from an orientation-preserving embedding of the graph of
Figure~\ref{fig:Chaine} (note that we do not require this embedding to be
$\pi_1$-injective). If we do not insist that the embedding be
orientation-preserving, we call it a {\em directed $k$-chain}, and, similarly, 
 $(\gamma_1,\ldots,\gamma_k)$ is simply a {\em $k$-chain} if
there exist signs $\epsilon_1,\ldots,\epsilon_k$ such that
$(\gamma_1^{\epsilon_1},\ldots,\gamma_k^{\epsilon_k})$ is a directed $k$-chain.
Also, we will say that a (oriented and/or directed) $k$-chain is
{\em completable} if it sits in the middle of a (orientable and/or directed)
$k+2$-chain.

For example, $(a_1^{-1}b_1a_1,a_1,b_1^{-1})$ is a non-completable $3$-chain
in $\Sigma_g$, and the collection
$(a_1,\delta_1,a_2,\delta_2,\ldots,\delta_{g-1},a_g,b_g^{-1})$
(as well as its sub-chains), where we have set
$\delta_i=a_{i+1}^{-1}b_{i+1}a_{i+1}b_i^{-1}$, forms a directed chain.
Also, the family $(a_1^{-1}b_1a_1,a_1,\delta_1,a_2,b_2^{-1})$ forms a
(non-completable) $5$-chain that will be handy in
Section~\ref{ssec:OrbitesFinies}.

If two simple closed loops $a,b\in\Gamma_g$ do not intersect outside of the
base point, we will write $i(a,b)=1$ if $(a,b)$ is an oriented, directed
$2$-chain, and we will write $i(a,b)=-1$ if $i(b,a)=1$. Otherwise we will
write $i(a,b)=0$; if $a$ and $b$ are nonseparating; this is equivalent to the
existence of a curve $c$ such that $(a,c,b)$ is a $3$-chain.  Though reminiscent of the algebraic intersection number,  
$i(a,b)$ is an {\em ad hoc} definition, as we do not
define $i(a,b)$ for most pairs $(a,b)$ of elements of $\Gamma_g$.

Finally, if two curves $a,b\in\Gamma_g$ satisfy $i(a,b)=\pm 1$, we will
denote by $T(a,b)$ the genus $1$ subsurface of $\Sigma_g$ defined by $a$ and
$b$ (Figure~\ref{fig:BaseStandard} illustrates $T(a_1,b_1)$). While
$T(a,b)$ is only defined up to based homotopy, it still makes sense to
say, for example, that a curve $\gamma$ is {\em disjoint} from $T(a,b)$,
if $i(a,\gamma)=i(b,\gamma)=i([a,b],\gamma)=0$.

We conclude this paragraph with
some considerations on complexes of pairs of based curves.
\begin{lemma}\label{lem:BallonsConnexe}
  Let $G_0$ denote graph whose vertices are the pairs $(a,b)\in\Gamma_g^2$ with
  $i(a,b)=\pm 1$, with an edge between two pairs $(a,b)$ and $(b,c)$
  whenever $(a,b,c)$ is a $3$-chain.  Then $G_0$ is connected.
\end{lemma}

The main results of this article do not depend on this lemma, as
we will simply need to work on a connected component of this graph (our proof
in the companion article \cite{AuMoinsG} follows this strategy).  However, the lemma is quite elementary, so here we take the honest approach of giving the proof and using the whole connected graph instead of making reference to a connected component.

The proof of Lemma \ref{lem:BallonsConnexe} is divided into two main observations.  
It essentially copies the proof of Proposition~6.7 of~\cite{Mod2}, but corrects a minor mistake there, where the complex
of {\em based} curves should have been used instead of the standard curve complex.

\begin{observation}\label{obs:Ballons1}
  Let $G_1$ be the graph whose vertices are the elements of $\Gamma_g$ 
  represented by simple, non-separating curves, and with edges between $a$ and $b$ if and only if $i(a,b)=\pm 1$.
  Then $G_1$ is connected.  
\end{observation}
\begin{proof}
  Let $G_2$ be the graph with the same vertices, but with edge between
  $a$ and $b$ whenever $i(a,b)$ is well defined.
  Let $G_3$ be the graph with vertex set consisting of the elements
  of $\Gamma_g$ represented by simple curves (possibly separating),
  with an edge between $a$ and $b$ whenever $i(a,b)$ is well defined.

  By drilling a puncture in $\Sigma_g$ at the base point, $G_3$
  can be identified with the arc graph of the surface $\Sigma_g^1$, which is
  well-known to be connected (see eg \cite{Hatcher91}). 
  Given a path in $G_3$ between two vertices of $G_2$, every time a separating curve appears 
  we may either delete it or
  replace it by a nonseparating curve, producing a new path in $G_2$. 
  Thus, $G_2$ is connected.
  Finally, we prove that any path in $G_2$ can be promoted to a path
  in $G_1$.  Let
  $a_1-a_2$ be an edge of $G_2$ which is not in $G_1$, ie,
  we have $i(a_1,a_2)=0$. Then a neighborhood of
  the curves $a_1$ and $a_2$ in $\Sigma_g$ is a pair of pants $P$, with three
  boundary components, freely homotopic to $a_1$, $a_2$ and $a_1a_2^{\pm 1}$.
  If $\Sigma$, $\Sigma'$ and $\Sigma''$ are, respectively, the connected
  components of $\Sigma_g\smallsetminus P$ separated from $P$ by $a_1$, $a_2$
  and $a_1a_2^{\pm 1}$, then we cannot have $\Sigma'\neq\Sigma''$, for
  otherwise $a_1$ or $a_2$ would be separating. Hence, there exists a curve
  $b$ such that $a_1-b-a_2$ is a path in $G_1$.
  %$i(a_1,b)=\pm 1$ and $i(b,a_2)=\pm 1$. In particular,
  %$a_1-b-a_2$ is a path in $G_1$; this proves that any path in $G_2$ can
  %be promoted to a path of $G_1$ by inserting some terms, and this proves that
  %$G_1$ is connected.
\end{proof}
\begin{observation}\label{obs:Ballons2}
  Let $a,b$ and $a'$ be such that $i(a,b)=\pm 1$ and $i(a',b)=\pm 1$.
  Then $(a,b)$ and $(a',b)$ lie in the same connected component of the graph $G_0$ from
  Lemma~\ref{lem:BallonsConnexe}.
\end{observation}
\begin{proof}
  Let $\sim$ denote the equivalence relation on vertices of $G_0$ 
  of being in the same connected component.  
  Let $a,b,a'$ be as in the statement of the observation, and let 
  $N$ be the (geometric) minimum number of disjoint intersections,
  besides the base point, between the based curves $a$ and $a'$.
  We will proceed by
  induction on $N$, starting with the base case $N=0$. In this case $i(a,a') \in \{0, \pm1\}$.  
  If $i(a,a')=0$, then $(a,b,a')$ is a $3$-chain and
  $(a,b)\sim (b,a')$. If $i(a,a')=\pm 1$, then for some $\epsilon\in\{-1,1\}$,
  we have $i(b^\epsilon a,a')=0$ (this is seen by looking at a
  neighborhood of the base point), hence $(b^\epsilon a,b,a')$ is a
  $3$-chain and $(b^\epsilon a)\sim(b,a')$. Now $(b^\epsilon a,b)\sim(a,b)$,
  because there exists a curve $c$ such that $(b^\epsilon a,b,c)$
  and $(a,b,c)$ are both $3$-chains. This proves the base case.  
  
  Now, suppose $N\geq 1$. Orient the curves $a$ and $a'$ so that their tangent vectors at $t=0$ 
  are on the same side of $b$ at the base point. 
  Let $(x_1,\ldots,x_N)$ be the intersection points
  of $a$ and $a'$, as ordered along the path $a$.
  Let $a''$ be the path obtained from following $a'$ 
  until we hit $x_N$ (actually, any of the $x_i$ would do),
  and then following the end of the path $a$.
  Then we have $i(a,b)=\pm 1$, $i(a',b)=\pm 1$, $i(a'',b)=\pm 1$ and
  the intersections of $a$ and $a'$ with $a''$ outside the base point are
  strictly less than $N$; this concludes our induction.
\end{proof}
\begin{proof}[Proof of Lemma~\ref{lem:BallonsConnexe}]
  Let $(a,b)$ and $(c,d)$ be such that $i(a,b)=\pm 1$ and $i(c,d)=\pm 1$.
  There exists a path between $b$ and $c$ in $G_1$,
  which can be extended to a path
  $\gamma_1-\gamma_2-\cdots-\gamma_n$ in $G_1$ with 
  $(a,b,c,d)=(\gamma_1,\gamma_2,\gamma_{n-1},\gamma_n)$.
  By Observation~\ref{obs:Ballons2},
  for all $j\in\{1,\ldots,n-2\}$, $(\gamma_j,\gamma_{j+1})$ is
  connected to $(\gamma_{j+1},\gamma_{j+2})$ in $G_0$, hence $(a,b)$ is
  connected to $(c,d)$. 
\end{proof}
Finally, we will also use the following easy variation of Lemma~\ref{lem:BallonsConnexe}.
\begin{lemma}\label{lem:BallonsConnexe2}
  Let $G$ denote graph whose vertices are the pairs $(a,b)\in\Gamma_g^2$ with
  $i(a,b)=\pm 1$, with an edge between two pairs $(a,b)$ and $(b,c)$
  whenever $(a,b,c)$ is a {\em completable} $3$-chain. Then $G$ is connected.
\end{lemma}
\begin{proof}
  First, observe that whenever $T(a,b)$ and $T(c,d)$ are disjoint,
  $(a,b)$ and $(c,d)$ are in the same connected component of $G$.
  Now, observe that if $(a,b,c)$ is a directed $3$-chain, then it is
  completable if and only if $ca$ is nonseparating.
  (The reader may find it helpful to draw a picture.)
  It follows that, if $(a,b,c)$ is a non-completable $3$-chain in $\Sigma_g$,
  then there exists a pair $(d,e)$ such that $a,b,c$ do not enter $T(d,e)$.
  Hence, $(a,b)$ and $(b,c)$ are connected to $(d,e)$ in $G$, and it follows
  that $G$ is connected.
\end{proof}

%%%%%%%%%%%%%%%%%%%%%%%%%%%%%%%%%%%%%%%%%%%%%%%%%%%%%%%%%%%%%%%%%%%%%%
%%%%%%%%                                           %%%%%%%%%%%%%%%%%%%
%%%%%%%%  BOITE A OUTILS - HOMEOS ET DEFORMATIONS  %%%%%%%%%%%%%%%%%%%
%%%%%%%%                                           %%%%%%%%%%%%%%%%%%%
%%%%%%%%%%%%%%%%%%%%%%%%%%%%%%%%%%%%%%%%%%%%%%%%%%%%%%%%%%%%%%%%%%%%%%

\subsection{Actions on the circle}\label{ssec:Per}

\subsubsection{Basic dynamics of circle homeomorphisms}\label{sssec:NotationPer}

We quickly review some definitions for the purpose of setting notation.
For more detailed background on this material, the reader may consult
\cite{Ghys87,KatieSurvey,Ghys01,Navas} for example.

We denote by $\HOZ$ the group of homeomorphisms
of $\R$ commuting with translation by 1; we have a natural
central extension \[ \Z \to \HOZ \to \HOS. \]
The {\em translation number} (or rotation number) of an element
$f\in\HOZ$ is defined as
$\rotild(f):=\lim_{n\to \infty}\frac{f^n(0)}{n}\in\R$,
and the Poincar\'e {\em rotation number} of an element $f\in\HOS$ is defined
as $\rot(f):=\rotild(\widetilde{f})~\mathrm{mod}~\Z$, where $\widetilde{f}$
is any lift of $f$.

We assume the reader is familiar with these invariants,
and with their essential properties.
Those that we will use most frequently are that $\rot$ and $\rotild$
are homomorphisms when restricted to abelian (eg. cyclic)
subgroups, that $\rot(f) = p/q \in \Q$ mod~$\Z$ if and only if $f$ has a
periodic orbit of period $q$, and that $\rotild$, and hence $\rot$,
are invariant under semi-conjugacy.
(The definition of semi-conjugacy is recalled
in Section \ref{ssec:CharSpace} where we will be using it.)

We denote by $\Per(f)=\{x \in S^1 \mid \exists n \in \Z,\ f^n(x) = x\}$
the set of periodic points of $f$.
If $n=1$, we also denote this by $\Fix(f)$.
For $\widetilde{f} \in \HOZ$, we use $\Per(\widetilde{f})$ to denote the set
of all lifts of points of $\Per(f)$ to $\R$.

For $f \in \HOS$ with $\Per(f) \neq \emptyset$, let $q(f)$ denote the
smallest non-negative integer such that $\Fix(f^{q(f)}) \neq \emptyset$,
and let $p(f)$ be the least non-negative integer
such that $f$ has rotation number equal to $\frac{p(f)}{q(f)}$ mod~$\Z$.

An \emph{attracting periodic point} for $f$ is a point $x \in \Per(f)$ with
a neighborhood $I$ of $x$ such that $f^{n q(f)}(I) \to x$ as $n \to \infty$.
A \emph{repelling} periodic point of $f$ is defined as an attracting periodic
point of $f^{-1}$. The sets of attracting and repelling periodic points
will be denoted $\Per^+(f)$ and $\Per^-(f)$ respectively.

\subsubsection{One-parameter families and bending deformations}\label{sss:bending}

Let $\gamma\in\Gamma_g$ be a based, simple loop. Cutting $\Sigma_g$ along
$\gamma$ decomposes $\Gamma_g$ into an amalgamated product
$\Gamma_g=A\ast_{\langle\gamma\rangle}B$, or an HNN-extension
$A\ast_{\langle\gamma\rangle}$, depending on whether $\gamma$ is separating.

In both cases, if $\rho\colon\Gamma_g\to\HOS$ is a representation and
if $(\gamma_t)_{t\in\R}$ is a continuous family of homeomorphisms
commuting with $\rho(\gamma)$ we may define a deformation of $\rho$, as
follows. If $\gamma$ is separating and
$\Gamma_g=A\ast_{\langle\gamma\rangle}B$, we define $\rho_t$ to agree with
$\rho$ on $A$, while setting $\rho_t(\delta)=\gamma_t\rho(\delta)\gamma_t^{-1}$
for all $\delta\in B$. If $\gamma$ is nonseparating, we may write
$a_1=\gamma$ and complete it into a standard generating system
$(a_1,\ldots,b_g)$, and set $\rho_t$ to agree with $\rho$ on all the
generators except $b_1$, and put $\rho_t(b_1)=\gamma_t\rho(b_1)$.

In both cases, we call this deformation a {\em bending along $\gamma$}.  
This is the analog of a bending deformation from the theory of quasi-Fuchsian
or Kleinian groups, in a very special case as
we bend only along one simple curve.

Most of the time (but not all) we will use these bendings with a
one-parameter group $\gamma_t$, ie, a morphism $\R\to\HOS$,
$t\mapsto\gamma_t$, as provided by Lemma~\ref{lem:PosDeform} below.
In the spacial case when $\rho(\gamma)=\gamma_1$, then the deformation defined
above, at $t=1$, is the precomposition of $\rho$ with
${\tau_\gamma}_\ast$, where $\tau_\gamma$ is the Dehn twist along $\gamma$.
However, for a Dehn twist to make sense as an automorphism of
$\Gamma$ (not up to inner automorphisms), we will use the following
convention.
\begin{convention}
  Suppose we are given a directed $k$-chain $(\gamma_1,\ldots,\gamma_k)$,
  and wish to write a Dehn twist along the loop $\gamma_i$.
  Then we will always do so by pushing $\gamma_i$ outside the base point in
  such a way that it intersects only $\gamma_{i-1}$ and $\gamma_{i+1}$
  (if these curves exist), in a neighborhood of the chain.
  Accordingly, if $\rho$ is a given representation and $\gamma_i^t$ is a
  one-parameter family commuting with $\rho(\gamma_i)$, then the deformation
  leaves $\gamma_j$ unchanged for $|j-i|\geq 2$ and $j=i$, and changes
  $\rho(\gamma_{i-1})$ into $\gamma_i^{-t}\rho(\gamma_{i-1})$
  and $\rho(\gamma_{i+1})$ into $\rho(\gamma_{i+1})\gamma_i^t$.
\end{convention}

Not all elements of $\HOS$ embed in a one-parameter subgroup.
In fact, if $\rot(f)$ is irrational, 
then $f$ embeds in such a subgroup if and only if $\Per(f) = S^1$, in which case $f$ 
is conjugate to a rotation.
However, elements with rational rotation number do have large centralizers,
giving us some flexibility in the use of bending deformations.
We formalize this in the next lemma. Here, and later on, it will be
convenient to fix a section of $\HOS$ in $\HOZ$.

\begin{notation}\label{not:Chapeau}
  For $f \in \HOS$, let $\widehat{f} \in \HOZ$ be the (unique) lift
  of $f$ with $\rotild(\widehat{f}) \in [0,1)$;
  we will call it the {\em canonical lift} of~$f$.
  Later, we will also need to refer to the lift of $f$ with translation number
  in $(-1, 0]$, this we denote by $\widecheck{f}$. Note that
  $\widehat{f}^{-1}=\widecheck{f^{-1}}$.
\end{notation}

\begin{lemma}(Positive 1-parameter families)\label{lem:PosDeform}
  Let $f\in\HOS$ have rational rotation number, and suppose
  $\Per(f)\neq S^1$. Then
  there exists a one-parameter group $(f_t)_{t\in\R}$, which commutes
  with $f$, such that $\forall t\neq 0$, $\Fix(f_t)=\partial\Per(f)$,
  and for all $t>0$ and $x \in \R \smallsetminus \partial\Per(\widetilde{f})$,
  we have $\widehat{f_t}(x)>x$.
\end{lemma}
\noindent Here and in what follows $\partial X$ denotes the \emph{frontier}
of a subset $X$ of $\R$ or $S^1$.
\begin{proof}
  The set  $S^1\smallsetminus\partial\Per(f)$ consists of a union of open
  intervals permuted by $f$. Choose a single representative interval $I_\alpha$
  from each orbit. Note that $f^{q(f)}(I_\alpha) = I_\alpha$ for any such
  interval, and the restriction of $f^{q(f)}$ to $S^1 \smallsetminus \Per(f)$
  is either fixed point free or the identity.
  Thus, we may identify each $I_\alpha$ with $\R$ such that $f^{q(f)}$, in
  coordinates, is $x\mapsto x+C$, for some $C\in\{-1,0,1\}$.
  Define $s_t$ on $I_\alpha$ to be $x \mapsto x+t$.
  Since these $I_\alpha$ are in different orbits of the action of $f$ on
  $S^1$, we may extend $s_t$ equivariantly to a 1-parameter family of
  homeomorphisms of $S^1$.
\end{proof}

In all the rest of this text, if $f\in\HOS$, any family $f_t$ as in
Lemma~\ref{lem:PosDeform} will be called a
{\em positive one-parameter family commuting with $f$}, or simply
a {\em positive one-parameter family} if $f$ is understood.

\subsubsection{Periodic sets under deformations}\label{sssec:Twist}

We now make some observations on how periodic sets change under bending
deformations using positive one-parameter families.  The main application of these comes in Section~\ref{ssec:P},  
but they will also make a few earlier appearances.  

Let $f$ and $g \in \HOS$ have rational rotation numbers. It follows immediately
from the definition of canonical lift that
\[ x \in\Per(\widehat{f})\Leftrightarrow \widehat{f}^{q(f)}(x)=x+p(f). \]   
Let $f_t$ be a positive one-parameter family 
commuting with $f$. Let $g_t : = f_t\circ g$,
and let $\widetilde{g_t}=\widehat{f_t}\circ\widehat{g}$.
Note that $\widetilde{g_t}=\widehat{g_t}$, provided the rotation number
of $g_t$ is constant as $t$ varies.

For all $(x,t)\in S^1\times\R$, we set
\[ \Delta_{f,g}(x,t_1,\ldots,t_{q(g)}) =
\widetilde{g_{t_{q(g)}}}\circ\cdots\circ\widetilde{g_{t_1}}(\tilde{x})-
\tilde{x}-p(g), \]
\[\text{and }\hfill
\delta_{f,g}(x, t) =
\Delta_{f,g}(x,t,\cdots,t)
=(\widetilde{g_t})^{q(g)}(\tilde{x})-\tilde{x} - p(g). \hfill \]
This does not depend on the lift $\tilde{x}\in\R$ of $x$, but does depend on
the choice of the one-parameter family $f_t$ (so we are somewhat abusing
notation). Further, we set
\[ P(f,g)=\left\lbrace x \in S^1\,|\,
  \forall t\in\R, \, \delta_{f,g}(x, t) = 0 \right\rbrace, \]
\[ N(f,g)=\left\lbrace x \in S^1\,|\,
  \forall t\in\R, \, \delta_{f,g}(x, t) \neq 0 \right\rbrace, \]
\[\text{ and } U(f,g)=\left\lbrace x\in S^1\,|\,
  \exists! t\in\R, \, \delta_{f,g}(x, t) = 0 \right\rbrace. \]
Unlike $\delta_{f,g}$, these
sets do not depend on the choice of the positive one-parameter family (provided that it is chosen as in Lemma \ref{lem:PosDeform}).  

Assuming $\rot(g_t)$ is constant, then
$P(f,g) = \bigcap_{t \in \R} \Per(g_t)$ is the set of {\em persistent} periodic
points; $N(f,g)$ is the set of points that are \emph{never} periodic for any
$g_t$, and $U(f,g)$ is the set of points that lie in $\Per(g_t)$ for a
\emph{unique} time $t$.

Let $T_{f,g}\colon U(f,g)\to\R$ be the map that assigns 
to each $x\in U(f,g)$, the unique time $t\in\R$ for which $\delta_{f,g}(x,t)=0$.
\begin{lemma}\label{lem:ProprietesPNU}
  Suppose $g_t$ has constant rotation number. Then we have the following
  properties.
  \begin{enumerate}
  \item The set $P(f,g)$ is closed, moreover
    \[ P(f,g)=\Per(g)\cap\bigcap_{k=0}^{q(g)-1}g^k(\partial\Per(f)); \]
    in particular, if $\rot(f)=0$ then every element of
    $P(f,g)$ has a finite orbit under the group $\langle f,g\rangle$.
  \item The sets $P(f,g)$, $N(f,g)$ and $U(f,g)$ partition the circle.
  \item The set $U(f,g)$ is open, and the map $T_{f,g}\colon U(f,g)\to\R$
    is continuous.
   \item For any $\varepsilon >0$, there exists $t_0$ such that
  $\Per(f_t\circ g)$ lies in the $\varepsilon$-neighborhood of
  $P(f,g) \cup \partial N(f,g)$ for all $t > t_0$.
  \end{enumerate}
\end{lemma}

\begin{proof}
  By construction, the map $\Delta_{f,g}(x,\cdot)$ is
  (separately, in each variable $t_j$) constant if
  $\widetilde{g_{t_{j-1}}}\circ\cdots\circ\widetilde{g_{t_1}}(\tilde{x})
  \in\partial\Per(f)$, and strictly increasing otherwise. Monotonicity
  implies that the subsets $\Delta_{f,g}(x,\R^{q(g)})$ and
  $\delta_{f,g}(x,\R)$ of $\R$ coincide.
  The affirmations $(1)$ and $(2)$ are easy
  consequences of these observations. Let us prove $(3)$.
  Let $x_0\in U(f,g)$, and write $t_0=T(x_0)$, so $\delta(x_0,t_0)=0$.
  Fix $\varepsilon>0$. Since $x_0\in U(f,g)$, we have
  $\delta(x_0,t_0+\varepsilon)>0$, and $\delta(x_0,t_0-\varepsilon)<0$.
  Since the maps $x\mapsto\delta(x,t_0+\varepsilon)$ and $x\mapsto\delta(x,t_0-\varepsilon)$
  are continuous, there exists $\eta>0$ such that, for all
  $x\in(x_0-\eta,x_0+\eta)$ we have $\delta(x,t_0+\varepsilon)>0$ and
  $\delta(x,t_0-\varepsilon)<0$. Thus, for each $x\in(x_0-\eta,x_0+\eta)$,
  the map $t\mapsto\delta(x,t)$ takes positive and negative values,
  hence has a (unique) zero in the interval $(t_0-\varepsilon,t_0+\varepsilon)$.
  In other words, $(x_0-\eta,x_0+\eta)\subset U(f,g)$, and for all
  $x\in(x_0-\eta,x_0+\eta)$, we have $|T(x)-T(x_0)|<\varepsilon$.
  
  For statement (4), fix $\varepsilon >0$. Let $I_1$, \ldots, $I_n$ denote
  the (finitely many) connected components of $U(f,g)$ of length
  $>\varepsilon$. Let $K \subset U(f,g)$ be the set of points of $U(f,g)$
  that are distance at least $\varepsilon$ from $P\cup\partial N$.
  Then, $K\subset \bigcup_i I_i$, and it follows that $K$ is compact.
  Since $T$ is continuous,  its restriction to $K$ takes values in some
  segment $[-t_0,t_0]$, this gives the $t_0$ from the statement.
\end{proof}

The next proposition describes the topology of the sets $P(f,g), N(f,g)$ and $U(f,g)$ in more detail.  

%%%%%%%%%%%%%%%%%%%%%%%%%%%%%%%%%%%%%%%%%%%%%%%%%%%%%%%%%%%%%%%%%%%%%
\begin{proposition} \label{prop:dNinP}
  Suppose $g_t$ has constant rotation number.
  Then all accumulation points of $\partial N(f,g)$ lie in $P(f,g)$.
\end{proposition}

The bulk of the proof of this is accomplished by the following lemma.  

\begin{lemma} \label{lem:PasDAcc}
  Let $x_0 \in S^1 \smallsetminus\Per(g)$ and let $I$ be a small interval
  containing $x_0$. Suppose there exists $u_k \in U(f,g) \cap I$ converging
  to $x_0$ from the right.
  Then there exists $\varepsilon > 0 $ such that
  $(x_0, x_0+\varepsilon) \subset U(f,g)$.
\end{lemma}
Of course the symmetric statement, with sequences converging to $x_0$ at
the left, holds as well, with a symmetric proof.
\begin{proof}
Let $x_0 \notin \Per(g)$, so we have $d:=d(x_0, g^{q(g)}(x_0))>0$.
First, suppose for contradiction that for all $j\in\{1,\ldots,q(g)\}$,
$g^j(x_0)$ is accumulated on the right by points of $\partial\Per(f)$.
Choose
$z_{q(g)}\in\partial\Per(f)\cap(g^{q(g)}(x_0),g^{q(g)}(x_0)+\frac{d}{2})$,
and, inductively for $j = q(g)-1, q(g)-2, ... , 1$ define
$z_j\in\partial\Per(f)\cap(g^j(x_0),g^{-1}(z_{j+1}))$ for
$j\in\{1,\ldots,q(g)-1\}$, and set $\delta=g^{-1}(z_1)-x_0$.
Then, for all $t>0$ we have
$(f_tg)^j(x_0,x_0+\delta)\subset(g^j(x_0),z_j)$, hence
$(f_{t} g)^{q(g)}(x_0,x_0+\delta) \subset
(g^{q(g)}(x_0),g^{q(g)}(x_0)+\frac{d}{2})$.
Now let $k\geq 0$ be such that $u_k\in(x_0,x_0+\delta)$,
and choose inductively $y_1\in (g(x_0),g(u_k))\cap\partial\Per(f)$,
and $y_j\in (g^j(x_0),g(y_{j-1}))\cap\partial(f)$, for $j\geq 2$.
Then, for all $t\in\R$ we have
$(f_tg)^{q(g)}(u_k)\in(y_{q(g)},z_{q(g)})$, hence
$(f_t g)^{q(g)}(u_k)\in(g^{q(g)}(x_0),g^{q(g)}(x_0)+\frac{d}{2})$;
this contradicts that $u_k\in U(f,g)$.

Thus, if a sequence $u_k \in U(f,g)$ converges to $x_0$ from the right,
then there exists some $j\in\{1,\ldots,q(g)\}$
such that $g^j(x_0)$ is not accumulated on the
right by points of $\partial\Per(f)$.
Let $j$ be the minimum such index, and let $y$ be such that
$(g^j(x_0), y] \subset S^1 \smallsetminus \partial\Per(f)$.
Let $k$ be large enough so that
$g \circ (f_t \circ g)^{j-1}(u_k) \subset (g^j(x_0), y]$ holds for all
$t\in\R$.
(Such $k$ exists using the argument above, since $g^i(x_0)$ is accumulated
on the right by $\partial\Per(f)$ for all $i<j$.)
Let $z \in (x_0, u_k)$. We will now show that $z \in U(f,g)$.

Since $f_t$ acts transitively on $(g^j(x_0), y]$, for $T$ sufficiently large we have $f_T \circ g \circ (f_T \circ g)^{j-1}(z) > g \circ (f_T \circ g)^{j-1}(u_k)$.  If $T > T(u_k)$, this guarantees that $\delta_{f,g}(z, T) > 0$.  
Similarly, if $T'$ is small enough, we will have $f_{T'} \circ g \circ (f_{T'} \circ g)^{j-1}(z) < g \circ (f_T \circ g)^{j-1}(u_{k'})$ for any given $u_{k'} \in (x_0, z)$, and choosing $T' < T(u_k')$ ensures that $\delta_{f,g}(z, T') < 0$.   This shows that $z \in U(f,g)$, as desired.
\end{proof}

\begin{proof}[Proof of Proposition \ref{prop:dNinP}]
  Let $x_0$ be an accumulation point of $\partial N(f,g)$. 
  If $x_0 \notin \Per(g)$, then by Lemma \ref{lem:PasDAcc}, on any side of
  $x_0$ containing a sequence of points in $\partial N(f,g)$, there is a
  neighborhood of $x_0$ containing no points of $U(f,g)$.
  Since $P(f,g), N(f,g)$ and $U(f,g)$ partition $S^1$, it follows that there
  is also a sequence of points in $P(f,g)$ approaching $x_0$ from this side.
  Since $P(f,g)$ is closed, $x_0 \in P(f,g) \subset \Per(g)$, a contradiction.
  
  It follows that $x_0 \in \Per(g)$.
  If also $x_0 \notin P(f,g)$, then $x_0 \in U(f,g)$ since $x_0$ is a
  periodic point for $f_0 \circ g = g$. But $U(f,g)$ is open, a contradiction.
\end{proof}
All the discussion above describes the variation of $\Per(g)$ upon deforming
$g$ by composition with $f_t$ on the left.  However, one may equally well replace $g$ by $g f_t$ and define sets $P$, $N$, and $U$ with the same properties (indeed, replacing  $g$ by $g f_t$ is equivalent to replacing $g^{-1}$ by $f_{-t} g^{-1}$).
There is no reason to privilege left-side deformations in the definition of bending, and we will occasionally make use of deformations on the right.

%%%%%%%%%%%%%%%%%%%%%%%%%%%%%%%%%%%%%%%%%%%%%%%%%%%%%%%%%%%%%%%%%%%%%%
%%%%%%%%%%%%%%%                             %%%%%%%%%%%%%%%%%%%%%%%%%%
%%%%%%%%%%%%%%%  CARACTERES   TOPOLOGIQUES  %%%%%%%%%%%%%%%%%%%%%%%%%%
%%%%%%%%%%%%%%%                             %%%%%%%%%%%%%%%%%%%%%%%%%%
%%%%%%%%%%%%%%%%%%%%%%%%%%%%%%%%%%%%%%%%%%%%%%%%%%%%%%%%%%%%%%%%%%%%%%

\subsection{The character space for $\mathbf \HOS$}  \label{ssec:CharSpace}

Recall that, for any group $\Gamma$, two homomorphisms $\rho_1$ and $\rho_2 \in \Hom(\Gamma, \HOZ)$ are said to be \emph{semi-conjugate}
\footnote{Note that this definition is \emph{not} the usual notion of semi-conjugacy from topological dynamical systems (eg. as in \cite{KatokHasselblatt}), which is not a symmetric relation.}
 if there exists a monotone (possibly non-continuous or non-injective) map
$h\colon\R\to\R$ such that $h(x+1)=h(x)+1$ for all $x\in\R$, and 
$h\circ\rho_1(\gamma)=\rho_2(\gamma)\circ h$ for all $\gamma\in\Gamma$.
Similarly, $\rho_1$ and $\rho_2 \in \Hom(\Gamma, \HOS)$
are semi-conjugate if there is such a map $h\colon\R\to\R$
such that for all $\gamma$, there are lifts $\widetilde{\rho_1(\gamma)}$
and $\widetilde{\rho_2(\gamma) \in \HOZ}$ which are semi-congugate by this map $h$.
Ghys \cite{Ghys87} proved that, under this definition, semi-conjugacy is an equivalence relation 
%(see also \cite{KatieSurvey} for an expository account)
.  

In Section~1 of \cite{CalegariWalker}, Calegari and Walker describe an analogy between 
rotation numbers of elements of $\HOS$ and characters of linear representations. 
Much as characters capture the dynamics of a linear representation; rotation numbers capture representations up to semi-conjugacy:

\begin{theorem}[Ghys \cite{Ghys87}, Matsumoto \cite{Matsumoto86}]\label{thm:GhysRot}
Let $\Gamma$ be any group, and let $S$ be a generating set for $\Gamma$.  For $f, g \in \HOS$, define $\tau(f,g):= \rotild(\tilde{f} \tilde{g}) - \rotild(\tilde{f}) -\rotild(\tilde{g})$ for any lifts $\tilde{f}$ and $\tilde{g} \in \HOZ$.  
With this notation, two representations $\rho_1$ and $\rho_2$ in $\Hom(\Gamma, \HOS)$ are semi-conjugate if and only if the following two conditions hold:
\begin{enumerate}[i)]
\item $\rot(\rho_1(s)) = \rot(\rho_2(s))$ for each $s \in S$,
\item $\tau(\rho_1(a), \rho_1(b)) = \tau (\rho_2(a), \rho_2(b))$ for all $a$ and $b$ in $\Gamma$. 
\end{enumerate}
\end{theorem}

We observe here that one can recover Calegari and Walker's analogy from our more general definition of 
character spaces for arbitrary groups.  For a topological group $G$, recall that $X(\Gamma,G)$ denotes the largest Hausdorff quotient of $\Hom(\Gamma,G)/G$.  
Let $G/\!/G$ denote the space $X(\Z,G)$; then there is, for each $\gamma\in\Gamma$
a natural, continuous map $X(\Gamma,G)\to G/\!/G$, which sends
the class of a representation $\rho$ to the class of $\rho(\gamma)$.
For example, when $G=\mathrm{SL}(2,\C)$, these are precisely the trace functions.
The next proposition says that when $G=\HOS$, these are the \emph{rotation numbers}, 
and the space $X(\Gamma,G)$ is,
as a set, exactly the set of semi-conjugacy classes of representations.

\begin{proposition}\label{prop:ChiSemiConj}
  Let $\Gamma$ be any group. Two representations $\rho_1$ and $\rho_2$
  in $\Hom(\Gamma,\HOS)$ are semi-conjugate if and only if they are equivalent in $X(\Gamma, \HOS)$. 
\end{proposition}

Following this analogy, the ``character variety'' for $\HOS$ not only comes
with its ``ring of functions'' (the rotation number functions), but
with an underlying topological space as well. This gives
the most natural setting to speak of rigidity, or to pose Question~\ref{q:CC}.

We defer the proof of Proposition \ref{prop:ChiSemiConj} in order to make some preliminary observations.  
The first is the important remark that Proposition \ref{prop:ChiSemiConj} has no analog in $\Homeo^+(\R)$ -- a group may have many dynamically distinct actions on the line, but the character space is a single point:  

\begin{proposition}\label{prop:ChiPourR}
  For any discrete group $\Gamma$, the space $X(\Gamma, \Homeo^+(\R))$ consists of a single point.
\end{proposition}

\begin{proof}
  Let $\rho \in \Hom(\Gamma, \Homeo^+(\R))$.
  Let $S$ be a finite, symmetric subset of $\Gamma$.
  Given $\varepsilon > 0$, we will conjugate $\rho$ so that
  $|\rho(s)(x) - x| < \epsilon$ holds for all $s \in S$ and $x \in \R$, hence show that conjugates of $\rho$ approach the trivial representation in the compact-open topology.  
   
  As a first case, assume also that the subgroup generated by $S$ has no
  global fixed points in $\R$. Then define $h(0) = 0$, and iteratively,
  for $n \in \Z$ define
  $h(n \varepsilon/2) = \max_{s \in S} s (h((n-1) \varepsilon/2)))$ if $n>0$,
  and 
  $h(n \varepsilon/2) = \min_{s \in S} s (h((n+1) \varepsilon/2)))$
  if $n<0$.
  Extend $h$ over the interior of each interval
  $[n\varepsilon/2,(n+1)\varepsilon/2]$ as an affine map.  Since $S$ has no global fixed point, this map $h$ is surjective, 
  hence it is an orientation-preserving homeomorphism.  
  Furthermore, we have
  $h s h^{-1}(n\varepsilon/2) \in [(n-1)\varepsilon/2,(n+1)\varepsilon/2]$
  for all $s \in S$.
  Thus, $|h s h^{-1}(x) - x| < \varepsilon$ holds for all $x \in \R$.
  
  If instead the subgroup generated by $S$ does have a global fixed point,
  we may define $h$ to be the identity on the set $F$ of global fixed points,
  and define it as above on each connected component of $\R\smallsetminus F$.
\end{proof}

%
%Note that the same result holds more generally for spaces of continuous representations when $\Gamma$ is a topological group, with the modification that $S$ should be a compact set.  

Recall that the action of any group on $S^1$ is either minimal, or has a finite orbit, or has a closed, invariant set (called the \emph{exceptional minimal set}), homeomorphic to a Cantor set, on which the restriction of the action is minimal.  
The following is an easy consequence of the definition of semi-conjugacy, which we will use in the proof of Proposition \ref{prop:ChiSemiConj}.

\begin{observation} \label{obs:SCMinimal}
Every action $\rho_1$ with an exceptional minimal set is semi-conjugate to a minimal action $\rho_2$, by a \emph{continuous} semi-conjugacy $h$ satisfying $h\circ\rho_1(\gamma)=\rho_2(\gamma)\circ h$.   Furthermore, if $\rho_2$ is minimal, and $\rho_1$ arbitrary, then any $h$ satisfying this equation is necessarily continuous.  In particular, a semi-conjugacy $h$ between two minimal actions is invertible, and hence a \emph{conjugacy}.
\end{observation} 

\begin{proof}[Proof of Proposition \ref{prop:ChiSemiConj}]
  For one direction, it suffices to prove that the quotient of the space
  $\Hom(\Gamma,\HOS)$ by semi-conjugacy is Hausdorff. This follows
  from Theorem \ref{thm:GhysRot}, since the maps $\rot$ and $\tau$ in the theorem are continuous,
  well defined on semi-conjugacy classes, take values in the (Hausdorff)
  spaces $S^1$ and $\R$, and distinguish semi-conjugacy classes. 
  %hence, distinct semi-conjugacy classes are separated by
  %invariant open sets in $\Hom(\Gamma,\HOS)$.
  
  For the converse, if $\rho$ has a finite orbit, then we can employ a similar strategy to the
  proof of Proposition~\ref{prop:ChiPourR} to conjugate it arbitrarily close
  to an action on the circle by rigid rotations.  
  Hence, there is a
  unique element of the character space corresponding to the semi-conjugacy
  class of~$\rho$.
 
  Now suppose instead that $\rho$ has an exceptional minimal set.   By Observation~\ref{obs:SCMinimal} there is a minimal action $\rho'$ and continuous map $h$ such that each $\gamma \in \Gamma$ has lifts satisfying  
$\widetilde{\rho'(\gamma)} \circ h = h \circ \widetilde{\rho(\gamma)}$ as in the definition of semi-conjugacy. 
  Let $S$ be a finite subset of $\Gamma$, and fix $\varepsilon>0$.  
  Let $\delta\in(0,\varepsilon)$ be small enough so that for all
$s\in S$ and all $x,y\in S^1$, $|x-y|<\delta$ implies
$|\rho'(s)(x)-\rho'(s)(y)|<\varepsilon$.

Since $h$ is continuous and commutes with $x \mapsto x+1$, we can approximate it by a homeomorphism $h' \in \HOZ$ at $C^0$ distance at most $\delta$ from $h$.  
Let $s \in S$ and $x \in \R$, and take the lifts $\widetilde{\rho'(s)}$ and $\widetilde{\rho(s)}$ as above.   Then we have
\[ | \widetilde{\rho'(s)}(x) -
  \widetilde{\rho'(s)}\circ (h \circ h'^{-1})(x)| < \varepsilon \]
and 
\[ | h\circ \widetilde{\rho(s)} \circ h'^{-1}(x) -
  h'\circ \widetilde{\rho(s)} \circ h'^{-1}(x) |< \varepsilon, \]
hence the definition of semi-conjugacy and the triangle inequality gives
\[ | \widetilde{\rho'(s)}(x) - h' \circ \widetilde{\rho(s)} \circ h'^{-1}(x)|
< 2\varepsilon. \]
This proves that every representation without finite orbit is
$\chi$-equivalent to the minimal representation in its
semi-conjugacy class.
\end{proof}

We conclude this section with two observations and a short lemma that will be useful later on.  The observations are simple consequences of Observation~\ref{obs:SCMinimal}.

 \begin{observation} \label{obs:CardinalityPer}
Let $\rho_2 \in \Hom(\Gamma, \HOS)$ be minimal, and let $\rho_1$ be any action semi-conjugate to $\rho_2$, as in Observation \ref{obs:SCMinimal}..  Then for any $\gamma \in \Gamma$, 
$\Per(\rho_2(\gamma)) = h \Per(\rho_1(\gamma))$, hence
$|\Per( \rho_2(\gamma)) | \leq | \Per(\rho_1(\gamma))|$.
\end{observation}
 
\begin{observation}\label{obs:DeformationsConj} 
  Suppose that $\rho$ is minimal and path-rigid, and let $a,b$ have
  $i(a,b)=-1$. Since $b^{q(b)}$ lies in a 1-parameter family, there is a
  bending deformation replacing $\rho(a)$ with $\rho(b^{N q(b)}a)$ for any
  $N \in \Z$, which is realized by precomposition with a Dehn twist (see Section \ref{sss:bending}).  Thus,
  the new representation has the same image as $\rho$; in particular
  it is minimal, hence \emph{conjugate} to $\rho$.
\end{observation}

\begin{lemma}\label{lem:PerCommuns}
  Let $f,g\in\HOS$ be two homeomorphisms with rational rotation number.
  The property that $f$ and $g$ share a periodic point depends only on the semi-conjugacy class of $\langle f, g \rangle$.  
\end{lemma}

\begin{proof}
For $f_1,\ldots,f_n\in\HOS$, let
$\tau(f_1,\ldots,f_n) =
\rotild(\widetilde{f_n}\circ\cdots\circ\widetilde{f_1})
-\sum_i\rotild(\widetilde{f_i})$
(which obviously does not depend on the choices of lifts).
Note that 
\[ \tau(f_1,\ldots,f_n)=\tau(f_1,f_n\circ\cdots\circ f_2)
-\sum_{j=2}^{n-1}\tau(f_j,f_n\circ\cdots\circ f_{j+1}), \]
so this function can be recovered from the two-variable $\tau$ of Theorem \ref{thm:GhysRot}.

To prove the lemma, we prove the stronger statement that $f$ and $g$ sharing a periodic point is equivalent to the following assertion: 
\begin{quote}{\em For any $\ell\geq 1$ and any integers
    $n_1,m_1,\ldots,n_\ell,m_\ell$, we have \\ $\tau(f^{n_1 q(f)},g^{m_1 q(g)},\cdots,
    f^{n_\ell q(f)},g^{m_\ell q(g)})=0.$}
    \end{quote}
Applying Theorem~\ref{thm:GhysRot} gives the desired conclusion.  

The assertion is clearly true if $f$ and $g$ share a periodic point.  
  For the converse, suppose $\Per(f) \cap \Per(g) = \emptyset$, so $S^1\smallsetminus(\Per(f)\cup\Per(g))$ is a union of intervals.
  As $\Per(f)$ and $\Per(g)$ are closed, disjoint sets,
  only finitely many of these complementary intervals have one boundary
  point in each of $\Per(f)$ and $\Per(g)$. Those bounded on the 
  right by a point of $\Per(f)$ and at their left by a point of $\Per(g)$
  alternate with the others (with the roles of right and left reversed), in particular there are an even number of
  such complementary
  intervals. 
  Let $I_1,\ldots,I_{2\ell}$ denote these intervals, in their cyclic order on
  the circle, and let $I_j = (x_j, y_j)$.
  Up to
  shifting the indices cyclically, we have
  $x_i,y_{i+1}\in\Per(g)$ and $x_{i+1},y_i\in\Per(f)$ for all $i$
  even.
  
  Choose a point $x$ in $I_1$.
  Since the interval $(x_1,y_2)$ contains only points of $\Per(g)$,
  there exists $n_1$  such that $f^{n_1 q(f)}(x)\in I_2$. Similarly, there exists
  a power $n_2$ of $g^{q(g)}$ which maps $f^{n_1 q(f)}(x)$ into $I_3$, and so on for $n_i$, $i>2$. 
  The last operation can be done so that the image of $x$, under
  a suitable word
  $g^{n_\ell q(g)}f^{n_\ell q(f)}\cdots g^{n_2 q(g)}f^{n_1 q(f)}$,
  lies to the right of $x$ in $I_1$. Then, choosing the canonical lifts of
  $f^{n_i q(f)}$ and $g^{m_i q(g)}$, we observe that
  $\tau(f^{n_1 q(f)},g^{m_1 q(g)},\cdots,
  f^{n_\ell q(f)},g^{m_\ell q(g)}) \geq 1$.
\end{proof}

\begin{remark} \label{rem:PerNonCommun} 
In the case $\Per(f) \cap \Per(g) = \emptyset$, the integer $\ell$ in the proof above also only depends on $\tau$ -- in fact, it is the \emph{minimal} integer such that there exist $m_i, n_i \in \Z$ with 
$\tau(f^{n_1 q(f)},g^{m_1 q(g)},\cdots,
  f^{n_\ell q(f)},g^{m_\ell q(g)}) \geq 1$.

%To see this, fix $h < \ell$.  
%For any $x_j$ even, and any $M, N \in \Z$ we have
%$\widehat{g^{M q(g)}} \widehat{f^{N q(f)}}(\widetilde{x_j}) <
%\widetilde{x_{j+2}}$,
%where $\widetilde{x_j}$ and $\widetilde{x_{j+2}}$ are consecutive lifts
%of $x_j$ and $x_{j+2}$. Thus, for any choice of integers $n_i, m_i$
%we have
%\[\widehat{g^{n_h q(g)}}\widehat{f^{n_h q(f)}}\cdots
%\widehat{g^{n_2 q(g)}}\widehat{f^{n_1 q(f)}}(\widetilde{x_j}) <
%\widetilde{x_{j+2h}}\] 
%(again, taking consecutive lifts), and since $h< \ell$ this implies that  
%\[\widehat{g^{n_h q(g)}}\widehat{f^{n_h q(f)}}\cdots
%\widehat{g^{n_2 q(g)}}\widehat{f^{n_1 q(f)}} (x) < x+1 \]
%for all $x \in \R$, whence
%$\tau(f^{n_1 q(f)},g^{m_1 q(g)},\cdots, f^{n_\ell q(f)},g^{m_\ell q(g)})<1$.
\end{remark}

%%%%%%%%%%%%%%%%%%%%%%%%%%%%%%%%%%%%%%%%%%%%%%%%%%%%%%%%%%%%%%%%%%%%%%
%%%%%%%%%%%%%%%                                   %%%%%%%%%%%%%%%%%%%%
%%%%%%%%%%%%%%%  Classe d'Euler et Geom => Rigid  %%%%%%%%%%%%%%%%%%%%
%%%%%%%%%%%%%%%                                   %%%%%%%%%%%%%%%%%%%%
%%%%%%%%%%%%%%%%%%%%%%%%%%%%%%%%%%%%%%%%%%%%%%%%%%%%%%%%%%%%%%%%%%%%%%

\subsection{The Euler class}\label{ssec:Euler}

Recall that the \emph{(integer) Euler class} for circle bundles is the generator $e$ of $H^2(\HOS; \Z) \cong \Z$; and the \emph{Euler number} of a representation $\rho: \Gamma_g \to \HOS$ is the integer $\langle \rho^\ast(e), [\Gamma_g]  \rangle$, where $[\Gamma_g]$ denotes the fundamental class, i.e. a generator of $H_2(\Gamma_g, \Z)$.
Although this definition only makes sense for fundamental groups of closed surfaces, (a surface with boundary has free fundamental group, and $H_2(F_n; \Z) = 0$) there is a \emph{relative} Euler number for surfaces with boundary, which is additive when such subsurfaces are glued together.
This can be made precise in the language of bounded cohomology as explained in \cite[\S~4.3]{BIW}.
(Compare also Goldman \cite{Goldman84} and Matsumoto \cite{Matsumoto87}.)
Following \cite{BIW}, we make the following definition. 

\begin{definition}[Euler number for pants.]\label{def:Euler}
Let $P\subset\Sigma_g$ be a subsurface homeomorphic to a pair of pants; equip it with three based curves $a$, $b$, $c$ as in
Figure~\ref{figure:pantalon}.  (If $P$ does not contain the basepoint, choose a path in $\Sigma_g$ from the base point to a chosen point in $P$, and use it to define the curves $a$, $b$ and $c$.)
Let $\rho\colon\pi_1\Sigma_g\to\HOS$, and let $\widetilde{\rho(a)}$,
$\widetilde{\rho(b)}$ be any lifts of $\rho(a)$ and $\rho(b)$ to
$\HOZ$, and let
$\widetilde{\rho(c)}=\left(\widetilde{\rho(b)}\widetilde{\rho(a)}\right)^{-1}$.
Then the contribution of $P$ to the Euler number of $\rho$ is
\[
\eu_P(\rho):=\rotild\left(\widetilde{\rho(a)}\right)+
\rotild\left(\widetilde{\rho(b)}\right)+\rotild\left(\widetilde{\rho(c)}\right).
\]
\end{definition}
\begin{figure}[!h]
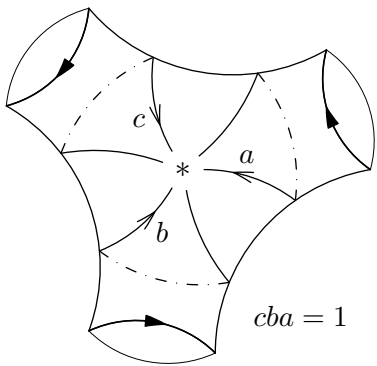

\begin{asy}
  import geometry;
  
  real long=70, ouv=40, corr1=10, corr2=35;
  
  pair p0=(long,0), p1=p0*dir(ouv), p2=p0*dir(120);
  
  path chemin1=p1{dir(ouv+180-corr1)}..{dir(120+corr1)}p2;
  path chemin2=rotate(120)*chemin1;
  path chemin3=rotate(120)*chemin2;
  
  path bordgros1=p0{dir(90+0.5*ouv+corr2)}..p1;
  path bordfin1=p0{dir(90+0.5*ouv-corr2)}..p1;
  
  path bordgros2=rotate(120)*bordgros1;
  path bordfin2=rotate(120)*bordfin1;
  
  path bordgros3=rotate(120)*bordgros2;
  path bordfin3=rotate(120)*bordfin2;
  
  pair pa1=intersectionpoint(chemin1,(0.4*long,0)--(0.4*long,50));
  pair pa2=intersectionpoint(chemin3,(0.6*long,0)--(0.6*long,-50));
  
  path chema1=8*dir(ouv){dir(ouv)}..pa1;
  path chema2=pa1{dir(-50)}..pa2;
  path chema3=pa2..{dir(180)}(8,0);
  
  draw(chemin1);
  draw(chemin2);
  draw(chemin3);
  
  draw(bordgros1,MidArrow);
  draw(bordgros1,black+0.6);
  draw(bordfin1,black+0.25);
  
  draw(bordgros2,MidArrow);
  draw(bordgros2,black+0.6);
  draw(bordfin2,black+0.25);
  
  draw(bordgros3,MidArrow);
  draw(bordgros3,black+0.6);
  draw(bordfin3,black+0.25);
  
  label("$*$",(0,0));
  
  draw(chema1);
  draw(chema2,dashdotted);
  draw(chema3,Arrow(SimpleHead,Relative(0.7)));
  label("\small $a$\normalsize",(24,5));
  
  draw(rotate(120)*chema1);
  draw(rotate(120)*chema2,dashdotted);
  draw(rotate(120)*chema3,Arrow(SimpleHead,Relative(0.7)));
  label("\small $c$\normalsize",(24,5)*dir(120));
  
  draw(rotate(240)*chema1);
  draw(rotate(240)*chema2,dashdotted);
  draw(rotate(240)*chema3,Arrow(SimpleHead,Relative(0.7)));
  label("\small $b$\normalsize",(24,5)*dir(240));
  
  label("\small $cba=1$\normalsize",(44,-55));
\end{asy}
\caption{A pair of pants with standard generators of its fundamental group}
\label{figure:pantalon}
\end{figure}
If the surface $\Sigma_g$ is cut into pairs of pants, the Euler class of
$\rho$ is the sum of the contributions of these pants. See
See \cite[\S~4.3]{BIW} for a detailed discussion, and
\cite{AuMoinsG} for a short exposition and proof that this does not depend
on the decomposition.
Definition~\ref{def:Euler} extends naturally to one-holed tori: if
$T=T(a,b)\subset\Sigma_g$ is a one-holed torus, cutting $T$ along a
simple closed curve (say, freely homotopic to $a$ or $b$) yields the formula
$\eu_T(\rho)= \rotild\left(\widetilde{\rho(b)}^{-1}\widetilde{\rho(a)}^{-1} \widetilde{\rho(b)} \widetilde{\rho(a)}\right)$,
which, in turn, gives Milnor's classical formula~\cite{Milnor},
$\eu(\rho)=\prod_{i=1}^g[\widetilde{\rho(a_i)},\widetilde{\rho(b_i)}]$,
where $(a_1,\ldots,b_g)$ is a standard system of curves, and where
the lifts are taken arbitrarily.

\section{A first statement}\label{sec:Exercices}

This section proves the main theorem under a strong additional hypothesis.
We will show that  if $\rho$ is path-rigid and if for every
$a, b \in \Gamma_g$ with $i(a,b) = \pm 1$,
$\rho(a)$ and $\rho(b)$ resemble, dynamically, a geometric representation,
then $\rho$ is in fact geometric.
In other words, the local condition that $\rho$ ``looks geometric'' on pairs
$a, b$ with $i(a,b) = \pm 1$ implies global geometricity.
To formalize this, we introduce some definitions.

\begin{definition}
  Say that an element $f \in \PSLk$ is \emph{hyperbolic} if its projection
  to $\PSL$ is hyperbolic. Equivalently, all its periodic points are
  hyperbolic in the sense of classical smooth dynamics.
\end{definition}  

\begin{definition} \label{def:Sk}
Let $a,b \in \Gamma_g$ and $\rho: \Gamma_g \to \HOS$.   Denote by $S_k(a,b)$ (the notation $\rho$ is suppressed) the property that 
\begin{enumerate}[i.]
\item $i(a,b) = \pm1$ and $\rho(a)$ and $\rho(b)$ are each separately conjugate to a hyperbolic element of $\PSLk$, and
\item their periodic points \emph{alternate} around the circle, meaning that each pair of points of $\Per(a)$ are separated by $\Per(b)$, and vice versa.  
\end{enumerate}
If all pairs $a,b$ with $i(a,b) = \pm1$ have $S_k(a,b)$, then we say that $\rho$ has property $S_k$. 
\end{definition}

With this notation we can state the main result of this section.

\begin{theorem}\label{thm:SkImpliqueGeom}
  Let $\rho$ be a path-rigid, minimal representation, and suppose $\rho$
  satisfies $S_k$ for some $k$. Then $\rho$ is geometric.
\end{theorem}

Before embarking on the proof, we discuss some other variations on hyperbolicity to be used later in the section.  

Let  $f\in\HOS$. We say that an open interval $I\subset S^1$
is {\em attracting for $f$} if $f(\overline{I})\subset I$. We say that $I$ is {\em repelling for $f$} if it is attracting for $f^{-1}$.
Matsumoto \cite{Matsumoto87} calls homeomorphisms that do not admit attracting intervals \emph{tame}.  In line with his terminology, we call those homeomorphisms which do \emph{savage}.  More specifically, we have:

\begin{definition}\label{def:Intervalles}
A homeomorphism $f\in\HOS$ is {\em $n$-savage} if there exist $2n$ open intervals with pairwise disjoint closures, indexed in cyclic order by $I_1^-$, $I_1^+$, \ldots, $I_n^-$, $I_n^+$ 
such that 
\[
f(S^1\smallsetminus(\cup_{j=1}^n \overline{I_j^-}))=\cup_{j=1}^n I_j^+.
\]
\end{definition}
\noindent In this sense, savage means $1$-savage.

The next observation is an immediate consequence of the definition, we leave the proof to the reader.
\begin{observation}\label{obs:Sauvagerie} 
If $f$ is $n$-savage, then $f^k$ is also $n$-savage for any $k \in \Z \smallsetminus \{0\}$.  Furthermore, $\rot(f^n)=0$ and $f$ has least one periodic point in each interval
$I_j^+$ and $I_j^-$.
\end{observation}

As a concrete example, note that if $f$ is conjugate to a hyperbolic
element in $\PSLk$, then $f$ is $n$-savage if and only if $n\leq k$.  

The intervals $I_j^+$ and $I_j^-$ in the definition of savage are by no way unique, 
but it will be convenient to use the notation $I^+(f):=\cup_{j=1}^n I_j^+$ and
$I^-(f):=\cup_{j=1}^n I_j^-$, even if this these sets depend on choices.
We also set $I(f):=I^+(f)\cup I^-(f)$.

\begin{definition}
Two $n$-savage homeomorphisms $f,g\in\HOS$ are {\em in $n$-Schottky
position} if their respective attracting and repelling intervals $I_j^\pm$ can be chosen so that
$I(f)$ and $I(g)$ have disjoint closures.
\end{definition}

Note that, if $f$ and $g$ are $n$-Schottky, then $f^{-1}$ and $g$ are
$n$-Schottky as well.  
Note also that the condition $S_k(a,b)$ is not equivalent to $k$-Schottky,
although $S_k(a,b)$ does imply that $a^N$ and $b^N$ are $k$-Schottky for
sufficiently large $N$.
We will prove however that hypothesis $S_k$ on a path-rigid representation
$\rho$ implies that $a$ and $b$ are indeed $k$-Schottky whenever
$i(a,b)=\pm 1$.

%%%%%%%%%%%%%%%%%%%%%%%%%%%%%%%%%%%%%%%%%%%%%%%%%

\subsection{Outline of Proof of Theorem \ref{thm:SkImpliqueGeom}}

We start in Section \ref{ssec:ExoMax} with a series of lemmas that use
rigidity and property $S_k$ to show the cyclic order of periodic points of
various non-separating curves agrees with that of a geometric representation,
and that certain pairs of curves are $k$-Schottky.
Following this, we show in Section \ref{subsec:Euler} that the Euler number
of a path-rigid, minimal, $S_k$ representation agrees with a geometric one,
i.e. is equal to $\pm\frac{2g-2}{k}$.
From there, we need to improve this essentially combinatorial result to the
fact that the representation is actually geometric.
Our main tool is existing work of Matsumoto on \emph{Basic Partitions}.
%His technique uses a decomposition of the surface into one-holed tori and pairs of pants, involving many separating curves, which on the level of fundamental groups corresponds to a decomposition of $\Gamma_g$ into a tree of groups.  
%However, our hypothesis $S_k$ is a condition on \emph{non-separating} curves (which is much more natural for us, given that we are working with standard generators, chains, and bending deformations).   Thus, before invoking Matsumoto, we must pass from non-separating to separating curves; this is done in Section \ref{ssec:ExoKt}.
%(A more direct approach would be to adapt Matsumoto's techniques to a general decomposition of the
%surface group into a graph of groups, but this Pyrrhic victory would actually make our proof longer and more difficult.)  

We are now ready to embark on the proof.
Throughout, we make the following assumption.

\begin{assumption}
For the rest of this section, $\rho$ denotes a path-rigid minimal representation of $\Gamma_g$ that satisfies $S_k$.   To simplify notation, we often omit $\rho$, identifying $a \in \Gamma_g$ with $\rho(a) \in \HOS$.  Thus, we will speak of $\Per(a)$, denote an attracting point of $\rho(a)$ by $a^+$, etc.
\end{assumption}

%%%%%%%%%%%%%%%%%%%%%%%%%%%%%%%%%%%%%%%%%%%%%%%%%

\subsection{Order of periodic points}\label{ssec:ExoMax}

Property $S_k$ makes it much easier to understand periodic points under deformations.   We start with several lemmas to this effect.  

\begin{lemma}\label{lem:PerLabiles}
  Let $i(a,b)=1$, let $F \subset S^1$ be a countable set, and let $b_t$
  be a positive one-parameter family commuting with $b = \rho(b)$.  
  Then for some $t\in\R$, we have $\Per(b_t\rho(a))\cap F=\emptyset$.
\end{lemma}

\begin{proof}
We use the notation from Section \ref{sssec:Twist}.
Path-rigidity of $\rho$ implies that $\rot(b_ta)$ is constant, and
Property $S_k$ implies that $P(b,a) = \emptyset$, so we need only worry
about points in $U = U(b,a)$.
Thus, provided $t\not\in T_{b,a}(F)$, we have $\Per(b_ta)\cap F=\emptyset$.
\end{proof} 
\begin{lemma}[Disjoint curves have disjoint $\Per$]\label{lem:PerDisjoints}
  Let $(a,b,c)$ be a completable directed $3$-chain.
  Then $\Per(a)\cap\Per(c)=\emptyset$.
  In fact, $\Per(c)\cap b^n(\Per(a))=\emptyset$ for all~$n\in\Z$.
\end{lemma}

\begin{proof}
  Fix $n \in \N$. Complete $(a,b,c)$ to a directed $4$-chain $(a,b,c,d)$,
  and apply a bending deformation replacing $c$ with $d_tc$
  (leaving the action of $a$ and $b$ unchanged, hence $b^n \Per(a)$
  unchanged), for a positive family $d_t$.
  By Lemma \ref{lem:PerLabiles}, there is some $t$ such that
  $\Per(d_t c) \cap b^n \Per(a) = \emptyset$.
  Now the conclusion follows from path-rigidity of $\rho$, together with
  Lemma~\ref{lem:PerCommuns}.
\end{proof}

Note that, if $i(a,b) = \pm1$, then for any $n \in \Z$ we also have
$i(b^n a, b) =\pm1$, hence $S_k(b^n a, b)$ holds. The next lemma describes
the position of the periodic points of $S_k(b^n a, b)$ for large $n$.
This is particularly useful since there exist bending deformations
replacing the pair $a, b$ with $b^n a, b$ provided that $q(b)$ divides $n$
(see Observation \ref{obs:DeformationsConj}).

\begin{lemma}[Movement of $\Per$ by bending] \label{lem:ConvergencePer}  
  Suppose $i(a,b) =\pm1$.
  Then as $N \to +\infty$, the points of
  $\Per^+(b^{N}a)$ approach $\Per^+(b)$, and $\Per^-(b^{N}a)$ approaches
  $a^{-1} \Per^-(b)$; similarly, as $N \to -\infty$, $\Per^+(b^{N}a)$
  approaches $\Per^-{b}$ and $\Per^-(b^{N}a)$ approaches
  $a^{-1} \Per^+(b)$.
\end{lemma}

\begin{proof} 
  The conclusion of the lemma is an easy exercise, provided that
  $a^{-1} \Per(b) \cap \Per(b) = \emptyset$.
  We claim that path-rigidity of $\rho$ implies this extra provision.
  To see this, suppose for example that $i(a,b)=1$, and let
  $(c, a, b)$ be a completable directed 3-chain.
  By Lemma \ref{lem:PerDisjoints},
  $\Per(c) \cap \Per(b) = \emptyset$. Thus, we can make a positive bending
  deformation replacing $a$ with $ac_t$, until
  $(ac_t)^{-1} \Per(b) \cap \Per(b) = \emptyset$.
\end{proof}

\begin{notation}\label{not:OrdrePer}
Let $f$ and $g$ be homeomorphisms of $S^1$. When talking about cyclic order
of periodic points, we use the notation $\bbrac{f^+, g^+, g^-, f^-}_k$ to mean
that, in cyclic order, there is one attracting point for $f$, followed by an
attracting point for $g$, followed by a repelling point for $g$, followed by
an attracting point for $f$, with this pattern repeating $k$ times.
The notation $f^{\pm}$ means any point from $\Per(f)$.
We also use other obvious variations, such as
$\bbrac{f^\pm, g^-, f^\pm, g^+}_k$, and extend this naturally to periodic
points of three or more homeomorphisms.

When such a cyclic order is given, we call an interval $I \subset S^1$ of
\emph{type $(f^+, g^-)$} if it is bounded on the left (proceeding
anti-clockwise, using the natural orientation of $S^1$) by a point of
$\Per^+(f)$ and on the right by a point of $\Per^-(g)$, and if it does not
contain a proper subinterval with this property.
We also use other obvious variations.
\end{notation}

\begin{lemma}[Periodic points of 3-chains] \label{lem:OrdrePer}
  Let $(a, b, c)$ be a completable directed $3$-chain. Then, up to reversing
  the orientation of the circle, the periodic points of $a$, $b$ and $c$
  come in the following cyclic order:
  \[  \bbrac{a^-,b^-,a^+,c^\pm,b^+,c^\pm}_k,  \]
\end{lemma}

\begin{proof}
Up to reversing orientation of $S^1$, we may suppose that the cyclic order of points in $\Per(a) \cup \Per(b)$ is $\bbrac{a^-,b^-,a^+,b^+}_k$.
Choose two consecutive points of $\Per(b)$ (in cyclic order), and denote these by $b^-$ and $b^+$.  
%(To avoid unnecessary subscripts or superscripts, we will often use $f^\pm$ to denote a \emph{specific} attracting/repelling fixed point of a homeomorphism $f$, and the double-bracket notation above to indicate the cyclic order of a set of fixed points.)
%\note{M: D\'ecid\'ement, j'aimerais bien
%supprimer cette parenth\`ese, qui me para\^it faire double-emploi avec la
%notation \ref{not:OrdrePer}.}
Let $a^+$ be the point of $\Per(a)$ between $b^-$ and $b^+$, and let
$c^\pm$ be the periodic point
of $c$ in this interval (there is exactly one by hypothesis $S_k$).
We know that the points of $\Per(a)$ in the interval $(b^-,b^+)$ are in
cyclic order $(b^-,a^+,b^{q(b)}(a^+),b^+)$.

By Lemma~\ref{lem:PerDisjoints}, $c^\pm$ cannot
be equal to $a^+$ or $b^{q(b)}(a^+)$. Suppose for contradiction that
$c^\pm$ lies in the interval $(b^-,a^+)$, or in the interval
$(b^{q(b)}(a^+),b^+)$.
Then the closed segment $[a^+,b^{q(b)}(a^+)]$ does not contain any
periodic point of $c$. Let $(c_t)_{t\in\R}$ be a positive 1-parameter family
commuting with $c$, and use this to perform a bending along $c$ as in
Section~\ref{sssec:Twist}. Using the notation from this section, we have
$\delta_{c,b}(a^+,0)>0$, but for $t$ sufficiently negative, we have
$\Delta_{c,b}(a^+,0,\ldots,0,t)<0$.
Thus, for some $t_0 < 0$, we have $\delta_{c,b}(a^+,t_0)=0$,
i.e. $a^+\in\Per(c_{t_0} b)\cap\Per(a)$. This, together with
Lemma~\ref{lem:PerCommuns} and the
path-rigidity of $\rho$, yields a contradiction.

The same argument applies to an interval of the form $(b^+,b^-)$, where $b^+$ and $b^-$ denote two other consecutive points of $\Per(b)$.  In that case, the argument shows that the (unique) periodic point
of $c$ in this interval lies between points of the form $b^{q(b)}(a^-)$ and $a^-$, proving the lemma.
\end{proof}

In particular, for all pairs $a,c\in\Gamma_g$ such that there exists a
completable $3$-chain $(a,b,c)$, Lemma~\ref{lem:OrdrePer} provides information
about the periodic sets of $a$ and $c$.

\begin{corollary}
  Let $a$ and $c$ be two non separating curves with $i(a,c)=0$, and suppose
  $c$ is not conjugate to $a$ or $a^{-1}$.
  Then their periodic points are in cyclic order
  $\bbrac{a^\pm,a^\pm,c^\pm,c^\pm}_k$.
\end{corollary}

\begin{proposition}
\label{prop:OrdrePer}
  Let $(a, b, c)$ be a completable directed $3$-chain. Then, up to
  reversing the orientation of the circle, the
periodic points of $a, b$ and $c$
  and the $b$-preimages of $\Per(c)$
  are in cyclic order
  \[ \bbrac{a^-,b^{-1}(c^\pm),b^-,b^{-1}(c^\pm),a^+,c^\pm,b^+,c^\pm}_k.\]
\end{proposition}

\begin{proof}[Proof of Proposition~\ref{prop:OrdrePer}]   
Apply a bending deformation of $\rho$ replacing $b$ with $c^{Nq(c)}b$, and
leaving the action of $c$ and $a$ unchanged.
By Lemma~\ref{lem:ConvergencePer}, for $N$ sufficiently large,
$\Per^-(c^{Nq(c)}b)$ approaches $b^{-1}\Per^-(c)$, and $\Per^-(c^{-Nq(c)}b)$
approaches $b^{-1}\Per^+(c)$. Since $\rho$ is path-rigid, the cyclic order
of periodic points is invariant under these deformations, hence the points
$b^{-1}(c^{\pm})$ all must lie in intervals of type $(a^-,a^+)$.

Now up to replacing $c$ with $c^{-1}$ (its orientation is unimportant in this
proof) we may assume that the order of periodic points given by
Lemma~\ref{lem:OrdrePer} is $\bbrac{a^-,b^-,a^+,c^+,b^+,c^-}_k$.
Then $b^{-1}\Per^-(c)$ lies in the intervals of type $(b^+,b^-)$,
as $b$ preserves these intervals. Thus, points of $b^{-1}\Per^-(c)$ are
between consecutive points of $\Per^-(a)$ and $\Per^-(b)$.
Similarly, the points $b^{-1}(c^+)$ are between consecutive points of the
form $b^-$ and $a^+$.
\end{proof}

The following variation is proved using the same style of argument. 

\begin{lemma}\label{lem:Ordreabcd}
  Let $a,b,c\in\Gamma_g$ be three non-separating curves such that $i(a,b)=-1$
  and $c$ is disjoint from $T(a,b)$.
  Up to reversing the orientation of $S^1$, we may suppose that the periodic
  points of $a$ and $b$ are in the order $\bbrac{a^-,b^+,a^+,b^-}_k$.
  Then the periodic points of $c$ all lie in intervals of type $(b^-,a^-)$.
\end{lemma}

Note that the order in which we prefer to take the periodic points of
$a$ and $b$ is different here than in the two preceding statements,
because here $i(a,b)=-1$.

\begin{proof}
Similar to the proof of Proposition~\ref{prop:OrdrePer}, we perform bending
deformations.  Since $\rho$ is path-rigid, the cyclic order of periodic
points does not change after the bending deformation replacing $b$ with
$a^{Nq(a)}b$ (leaving $a$ and $c$ unchanged). 
The effect of these deformations is to push $\Per^+(b)$ as close as we want
to either $\Per^+(a)$ or $\Per^{-1}(a)$.
Applying Lemma \ref{lem:ConvergencePer} as in the proof of
Proposition~\ref{prop:OrdrePer} shows that periodic points of $c$
cannot be in the intervals of type $(a^-,b^+)$ or $(b^+,a^+)$ --
as the argument is entirely analogous, we omit the details.
The same argument again using the deformation replacing $a$ by $b^{Nq(b)}a$
shows that the periodic points of $c$ cannot be in the intervals of type
$(a^+,b^-)$, either.
\end{proof}

\begin{proposition}\label{prop:Schottky}
  Let $a$, $c$ be two non separating curves with $i(a,c)=0$, and suppose
  $c$ is not conjugate to $a$ or $a^{-1}$.
  Then $\rho(a)$ and $\rho(c)$ are in $k$-Schottky position.
\end{proposition}

\begin{proof}
Up to changing the orientation of $c$, we may
choose non separating curves $b$ and $d$ such that $(a,b,c,d)$
is the beginning of a standard basis of $\pi_1\Sigma_g$. 

Using a deformation as in Lemma \ref{lem:PerDisjoints}, path-rigidity of $\rho$ implies that the points of 
$\Per^-(d)$, $c^{-1} \Per^+(d)$, $\Per^-(b)$, and $a^{-1} \Per^+(b)$ are all distinct.  
Fix small disjoint neighborhoods $U^+$ of $\Per^-(d)$, $U^-$ of $c^{-1} \Per^+(d)$, and also $V^+$ of $\Per^-(b)$, and $V^-$ of $a^{-1} \Per^+(b)$.  

By Lemma \ref{lem:ConvergencePer}, if $n$ is large enough, $d^{-nq(d)}c(S^1 \smallsetminus U^-) \subset U^+$ and $b^{-nq(b)}a(S^1 \smallsetminus V^-) \subset V^+$, so we may find $2k$ disjoint attracting and repelling intervals for $d^{-nq(d)}c$ and $b^{-nq(b)}a$ as in the definition of $k$-Schottky.   
Now there exists a bending deformation that replaces $c$ with $d^{-nq(d)}c$ and $a$ with 
$b^{-nq(b)}a$, and it follows from Observation~\ref{obs:DeformationsConj} that this deformation is conjugate to the original action.  Thus, $a$ and $c$ are $k$-Schottky.
\end{proof}

\begin{proposition}\label{prop:Schottky2}
  Let $a$, $c$ be two non separating curves with $i(a,c)=\pm 1$.
  Then $\rho(a)$ and $\rho(c)$ are in $k$-Schottky position.
\end{proposition}

\begin{proof} Choose $b$ and $d$ so that $(b,a,c,d)$ is a 4-chain.  Now follow the proof above.
\end{proof}

From Proposition \ref{prop:Schottky} we deduce an enhanced version of
Lemma~\ref{lem:OrdrePer}.

\begin{proposition}\label{prop:OrdrePer2}
  Let $(a, b, c)$ be a completable directed $3$-chain.
  Then, up to reversing the orientation of the circle, the periodic points of
  $a$, $b$ and $c$ are in cyclic order $\bbrac{a^-,b^-,a^+,c^-,b^+,c^+}_k$.
\end{proposition}

\begin{proof}
  Following Proposition~\ref{prop:OrdrePer}, we need only discard the
  possibility that the order is
  $\bbrac{a^-,b^-,a^+,c^+,b^+,c^-}_k$. Suppose for contradiction that this
  order does hold.
  By Proposition~\ref{prop:Schottky}, we know that $a$ and $c$
  each have $2k$ intervals as in Definition~\ref{def:Intervalles}, with
  pairwise disjoint  closures. As $| \Per(a) | = | \Per(c) | = 2k$, each
  of these intervals contains exactly one periodic point, so their cyclic
  order is specified by the order of periodic points given above.
  
  Note that $ca$ is non-separating, as the $3$-chain $(a,b,c)$ is
  completable.  Also, 
  $\rho(ca)$ is $k$-savage,
  and we may take $I^-(ca)\subset I^-(a)$
  and $I^+(ca)\subset I^+(c)$.  With the same argument as above,
  $\rho(ca)$ has exactly one repelling periodic point in each interval of
  $I^-(ca)$, and one attracting periodic point in each interval of $I^+(ca)$.
  
  If $\Per(b)$ is disjoint from $I^-(a) \cup  I^+(c)$, then this is enough to
  imply that the periodic points of $ca$ and $b$ alternate, contradicting
  Lemma~\ref{lem:OrdrePer}, since $i(ca,b)=0$.
  Thus, it only remains to prove that $\Per(b)$ can be made disjoint
  from $I^-(a) \cup  I^+(c)$ to finish the proof.
  This can be done in the same manner as that of Proposition~\ref{prop:Schottky}.
  First, complete $(a,b,c)$ into a directed $5$-chain $(\alpha,a,b,c,\gamma)$.
  Then,
  consider a bending deformation of $\rho$, where $b$ is unchanged but
  the action of $a$ is replaced by that of $a\alpha^{N q(\alpha)}$
  and the action of $c$ by $\gamma^{N q(\gamma)} c$ for $N$ large.
  By Observation~\ref{obs:DeformationsConj} this new action is conjugate to
  $\rho$.
  Now, provided $N$ is large enough, we can choose our Schottky intervals
  to be as narrow as we want, around the points $\alpha^-$, $a(\alpha^+)$,
  $\gamma^+$ and $c^{-1}(\gamma^-)$ which, using Lemma~\ref{lem:PerDisjoints},
  are disjoint from $\Per(b)$.
\end{proof}

%%%%%%%%%%%%%%%%%%%%%%%%%%%%%%%%%%%%%%%%%%%%%%%%%%%%%%%
%%%%%%%%%%%%%%%%%%%%%%%%%%%%%%%%%%%%%%%%%%%%%%%%%%%%%%%%
%%%%%%%%%%%%%%%%%%%%%%%%%%%%%%%%%%%%%%%%%%%%%%%%%%

\subsection{Euler number}\label{subsec:Euler}

As a consequence of the work in the previous section, we show that the Euler number of $\rho$ agrees with a geometric representation.   

\begin{theorem}\label{theo:ExoMaxime}
Let $\rho$ be path-rigid, minimal and satisfy $S_k$. Then
$|\eu(\rho)|=\frac{2g-2}{k}$.  
\end{theorem}

In fact, we will show the following stronger statement, which implies
Theorem~\ref{theo:ExoMaxime} by additivity of the Euler number on
subsurfaces.

\begin{theorem}\label{thm:Pants}
  Up to changing the orientation of the circle, for every pair-of-pants
  subsurface $P \subset \Sigma_g$, the relative Euler class of $\rho$ on
  $P$ is $\frac{-1}{k}$.
\end{theorem} 
%This is true in particular for one-holed torus subsurfaces, which will be
%of interest in the next subsection.

\begin{definition}
  Let $i(a,b)=1$. We say that the ordered pair $(a,b)$ is of type~$+$
  if the periodic points of $a$ and $b$ are in the cyclic order
  $\bbrac{a^-,b^-,a^+,b^+}_k$. Otherwise, we say that $(a,b)$ is of type~$-$.
\end{definition}

As a consequence of Proposition~\ref{prop:OrdrePer2}, for every oriented,
completable directed $3$-chain $(a,b,c)$, the pairs $(a,b)$ and $(b,c)$
have the same type.
Thus, Lemma~\ref{lem:BallonsConnexe2} implies that
all one-holed tori have the same type. Thus, up to conjugating $\rho$ by an
orientation-reversing homeomorphism, we may suppose the type is always~$+$.

\begin{proof}[Proof of Theorem \ref{thm:Pants}]
  We begin by proving the claim for a pair of pants $P$, such that
  at least two boundary components of $P$ are non-separating.
  Denote by $a^{-1}$, $c^{-1}$, and $ac$ the three boundary components of $P$,
  with the convention of Figure~\ref{figure:pantalon}, and suppose that
  $a$ and $c$ are non-separating. With these choices of orientations,
  the Euler number of $\rho$ on $P$ will be equal to
  $\rotild(\widehat{a}\widehat{c})-\rotild(\widehat{a})-\rotild(\widehat{c})$,
  and there exists a curve $b$ such that $(a,b,c)$ is an oriented, completable,
  directed $3$-chain (the end of the proof of Observation~\ref{obs:Ballons1}
  justifies the existence such a curve $b$).

  Since $(a,b)$ is of type~$+$, it follows from
  Proposition~\ref{prop:OrdrePer2} that the periodic points of $a$ and $c$
  are in cyclic order $\bbrac{a^-,a^+,c^-,c^+}_k$; and by
  Proposition~\ref{prop:Schottky}, they are in $k$-Schottky position, with
  Schottky intervals $I_j^\pm(a)$ and $I_j^\pm(c)$. Lift these to intervals
  $\tilde{I}_j^\pm(a)$ and $\tilde{I}_j^\pm(c) \subset \R$, indexed by
  integers, and in order 
  \[ ... \tilde{I}_j^-(a), \tilde{I}_j^+(a), \tilde{I}_j^-(c),
  \tilde{I}_j^+(c), \tilde{I}_{j+1}^-(a), ... \]
  such that the projection to $S^1$ is given by taking indices mod $k$.
  It follows easily from the definition of Savage
  (see also Observation~\ref{obs:Sauvagerie}) that
  $\widehat{a}(\tilde{I}_j^+(a)) \subset \tilde{I}_{j+\ell}^+(a)$ for some
  $\ell$
  (which depends on $a$)
  and in this case
  $\ell/k = \rotild(\widehat{a})$. An analogous statement holds also for $c$;
  let $m/k$ denote its translation number.

  Since $a$ and $c$ are in $k$-Schottky position, their product $ac$ is
  $k$-savage, and we can take $I^-(ac)=I^-(c)$ and $I^+(ac)\subset I^+(a)$.
  Note that each of the $k$ intervals of $I^+(ac)$ is contained in a
  different interval of $I^+(a)$. We now track images of intervals to
  compare translation numbers.
  Set the indexing of the intervals $\tilde{I}^\pm(ac)$ so that
  $\tilde{I}_1^+(a) = \tilde{I}_1^+(ac)$.
  This lies between $\tilde{I}_{0}^+(c)$ and $\tilde{I}_1^-(c)$, so we have
  \[ c(\tilde{I}_1^+(ac)) \subset \tilde{I}_m^+(c) \]
  and similarly, since $\tilde{I}_m^+(c)$ lies between $\tilde{I}_m^+(a)$
  and $\tilde{I}_{m+1}^-(a)$, we have
  \[ ac(\tilde{I}_1^+(ac)) \subset a(\tilde{I}_m^+(c)) \subset
  \tilde{I}^+_{m+\ell}(a)= \tilde{I}^+_{m+\ell}(ac). \]
  Thus, $k\cdot \rotild(\widehat{a}\widehat{c}) =
  m + \ell -1 = k\cdot \rotild(\widehat{a}) + k\cdot \rotild(\widehat{c}) -1$ and hence
  $k(\rotild(\widehat{a} \widehat{c})-\rotild(\widehat{a})-
  \rotild(\widehat{c})) = -1$,
  as desired.
  
  This implies Theorem~\ref{theo:ExoMaxime}, as we can cut the surface
  $\Sigma_g$ into pairs of pants whose boundary components are all
  non-separating.
  
  Now, if $P$ is a pair of pants with possibly more than one separating
  boundary component, then $\Sigma_g \smallsetminus P$ admits a pants
  decomposition whose pants all have at most one separating
  boundary component.
  The fact that the contribution of $P$ to the Euler class of $\rho$ is
  $\frac{-1}{k}$ is then a consequence of Theorem~\ref{theo:ExoMaxime}
  and the additivity of the Euler class.
\end{proof}

\subsection{Basic partitions and combinations}\label{subsec:Matsumoto}

Fix disjoint, nonseparating curves $C_1$, \ldots, $C_{3g-3}$ so that
$\Sigma_g \smallsetminus \left( \bigcup_i C_i \right)$ is a 
disjoint union of pairs of pants $P_1$, \ldots, $P_{2g-2}$. 
For concreteness, the reader may use the decomposition suggested in
Figure~\ref{fig:GrapheDeGroupes}.

We briefly part from the convention for the presentation of $\pi_1\Sigma_g$
that was given in paragraph~\ref{ssec:Marquage}, and instead present
$\pi_1\Sigma_g$ as the fundamental group of a graph of groups.
Choose base points $x_i \in P_i$ and $y_j \in C_j$, identifying $x_1$ with
the base point of $\pi_1\Sigma_g$.
Also, choose paths in $P_i$ from $x_i$ to each basepoint of each boundary
component of $P_i$.
This collects all the base points of the pants and curves as the vertices
of a graph $G$ embedded in $\Sigma_g$; fix an orientation for each of its
edges, and a spanning tree $T \subset G$.
\begin{figure}[htb]
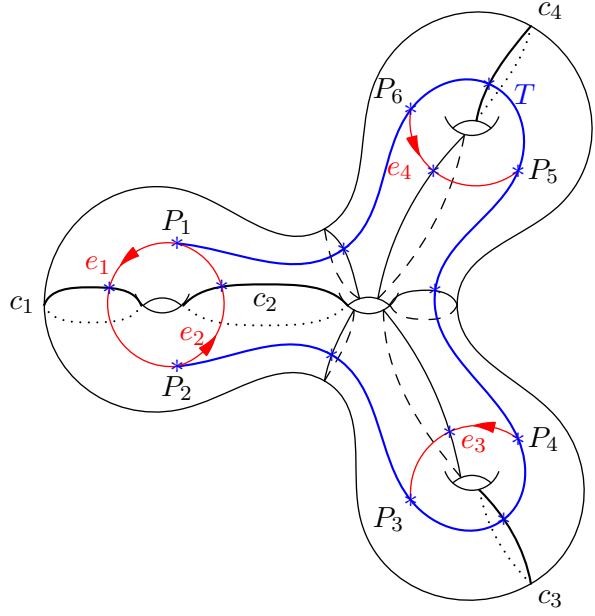

\begin{asy}
  import geometry;
  unitsize(1pt);
  
  picture anse1, anse2, anse3, letrou;
  
  real r = 30, R = 110, ct = -70; //% petit rayon, grand rayon, abcisse du centre du trou
  point p1 = 74*dir(180+29), p2 = 74*dir(180-32);
  path GC = r*dir(-120){dir(150)}..p1..(-R,0){up}..p2..r*dir(120){dir(30)}; //% grand contour
  path BT = (-10,4)..(0,-3)..(0+10,4); //% bas du trou
  path HT = relpoint(BT,0.2)..(0,2.6)..relpoint(BT,0.8);
  
  path c11 = (-R,0){dir(70)}::(-95,6){right}::(ct,0)+relpoint(HT,0){dir(-70)}; //% debut du chemin c_1
  path c12 = reflect((0,0),(1,0))*c11; //% fin du chemin c_1
  path c21 = (ct,0)+relpoint(HT,2){dir(50)}::(-35,7)::relpoint(HT,0){dir(-50)};
  path c22 = reflect((0,0),(1,0))*c21;
  path c71 = relpoint(HT,2){dir(70)}::(20,5){right}::(r,0){dir(-70)};
  path c72 = reflect((0,0),(1,0))*c71;
  
  path c61 = relpoint(BT,0.7){dir(-50)}..relpoint(HT,0.25)+70*dir(-60){dir(-80)};
  path c62 = relpoint(BT,0.7){dir(-90)}..relpoint(HT,0.25)+70*dir(-60){dir(-50)};
  path c51 = relpoint(HT,0.7){dir(80)}..relpoint(BT,0.4)+70*dir(60){dir(45)};
  path c52 = relpoint(HT,0.7){dir(45)}..relpoint(BT,0.4)+70*dir(60){dir(80)};
  
  path c31 = relpoint(BT,0.6)+70*dir(-60){dir(-45)}..R*dir(-60){dir(-80)};
  path c32 = relpoint(BT,0.6)+70*dir(-60){dir(-80)}..R*dir(-60){dir(-45)};
  path c41 = relpoint(HT,0.6)+70*dir(60){dir(80)}..R*dir(60){dir(50)};
  path c42 = relpoint(HT,0.6)+70*dir(60){dir(55)}..R*dir(60){dir(80)};
  
  path c81 = relpoint(HT,0.3){dir(105)}..relpoint(GC,1){dir(150)};
  path c82 = relpoint(HT,0.3){dir(150)}..relpoint(GC,1){dir(100)};
  path c91 = relpoint(GC,0){dir(80)}..relpoint(BT,0.3){dir(60)};
  path c92 = relpoint(GC,0){dir(60)}..relpoint(BT,0.3){dir(90)};
  
  point bP1=(-65,21), bP2=(-65,-21), bP3=dir(120)*bP1, bP4=dir(120)*bP2;
  point bP5=dir(-120)*bP1, bP6=dir(-120)*bP2;
  point bC1=relpoint(c11,0.65), bC2=relpoint(c21,0.25), bC3=relpoint(c31,0.35);
  point bC4=relpoint(c41,0.35), bC5=relpoint(c51,0.75), bC6=relpoint(c61,0.75);
  point bC7=relpoint(c71,0.65), bC8=relpoint(c81,0.65), bC9=relpoint(c91,0.35);
  
  path T= bP1{dir(-5)}..bC8..bP6..bC4..bP5..bC7..bP4..bC3..bP3..bC9..bP2{dir(185)};
  path e1=bP1..bC1..bP2, e2=bP2..bC2..bP1, e3=bP4..bC6..bP3, e4=bP6..bC5..bP5;
  
  //% On trace dans les environnements picture
  draw(anse1, GC);
  draw(letrou, BT); draw(letrou, HT);
  add(anse1,letrou,(ct,0));
  
  add(anse2, anse1);
  draw(anse1, c11, black+0.8pt); draw(anse1, c12, 0.8pt+black+dotted);
  draw(anse1, c71); draw(anse1, c72, black+dashed);
  draw(anse1, c21, black+0.8pt); draw(anse1, c22, 0.8pt+black+dotted);
  
  draw(anse1, c61); draw(anse1, c62, black+dashed);
  draw(anse1, c51); draw(anse1, c52, black+dashed);
  draw(anse1, c31, black+0.8pt); draw(anse1, c32, 0.8pt+black+dotted);
  draw(anse1, c41, black+0.8pt); draw(anse1, c42, 0.8pt+black+dotted);

  draw(anse1, c81); draw(anse1, c82, black+dashed);
  draw(anse1, c91); draw(anse1, c92, black+dashed);
  
  add(anse3, anse2);
  
  draw(anse1, T, blue+0.8pt);
  
  label(anse1, "\small $*$", bP1, blue); label(anse1, "\small $*$", bP2, blue);
  label(anse1, "\small $*$", bP3, blue); label(anse1, "\small $*$", bP4, blue);
  label(anse1, "\small $*$", bP5, blue); label(anse1, "\small $*$", bP6, blue);
  label(anse1, "\small $*$", bC1, blue); label(anse1, "\small $*$", bC2, blue);
  label(anse1, "\small $*$", bC3, blue); label(anse1, "\small $*$", bC4, blue);
  label(anse1, "\small $*$", bC5, blue); label(anse1, "\small $*$", bC6, blue);
  label(anse1, "\small $*$", bC7, blue); label(anse1, "\small $*$", bC8, blue);
  label(anse1, "\small $*$", bC9, blue);
  
  draw(anse1, e1, red, Arrow(Relative(0.3)));
  draw(anse1, e2, red, Arrow(Relative(0.3)));
  draw(anse1, e3, red, Arrow(Relative(0.3)));
  draw(anse1, e4, red, Arrow(Relative(0.3)));
  
  label(anse1, "\small $P_1$", bP1,N);
  label(anse1, "\small $P_2$", bP2,S);
  label(anse1, "\small $P_3$", bP3,SW);
  label(anse1, "\small $P_4$", bP4,E);
  label(anse1, "\small $P_5$", bP5,E);
  label(anse1, "\small $P_6$", bP6,NW);
  label(anse1, "\small $c_1$", (-R,0),W);
  label(anse1, "\small $c_2$", relpoint(c21,0.5),S);
  label(anse1, "\small $c_3$", R*dir(-60),SE);
  label(anse1, "\small $c_4$", R*dir(60),NE);
  label(anse1, "\small $e_1$", relpoint(e1,0.3),W,red);
  label(anse1, "\small $e_2$", relpoint(e2,0.3),W,red);
  label(anse1, "\small $e_3$", relpoint(e3,0.3),S,red);
  label(anse1, "\small $e_4$", relpoint(e4,0.3),SW,red);
  label(anse1, "\small $T$", relpoint(T,0.32),E,blue);
  
  add(scale(1.1)*anse1,(0,0));
  add(scale(1.1)*letrou,(0,0));
  add(scale(1.1)*rotate(120)*anse2,(0,0));
  add(scale(1.1)*rotate(-120)*anse3,(0,0));
\end{asy}
\caption{A decomposition of $\pi_1\Sigma_4$ into a graph of groups}
\label{fig:GrapheDeGroupes}
\end{figure}
This data gives a {\em graph of groups}:
the vertex (resp. edge) groups are the fundamental groups of the based
pairs of pants (resp. curves), and for each edge $C_j$,
the chosen paths define monomorphisms $\phi_j$ and $\psi_j$ from
$\pi_1 C_j\simeq\Z$ to the fundamental groups of the two adjacent
(initial and final endpoints of the edge, respectively) pairs of pants.
%(the index $0$ is for the initial pair of pant, with
%respect to the chosen orientations of edges of $G$, and the $1$ is for the
%target pair of pant).
The Seifert-Van Kampen theorem then identifies $\pi_1\Sigma_g$ with the
fundamental group of this graph of groups; this is the group
generated by the union of the $\pi_1 P_i$, as well as one extra generator
$e_j$ for each edge that is not in $T$, subject to the relations that
for each edge $C_j$ (in $T$ or not), and each $\gamma\in\pi_1C_j$,
we have $\phi_j(\gamma)=e_j^{-1}\psi_j(\gamma)e_j$
(taking $e_j=1$ for the edges in $T$).
%\note{M: c'est affreux,
%je croyais que le th\'eor\`eme de Van Kampen \'etait \'ecrit exactement
%comme \c{c}a chez Jean-Pierre Serre... Mais on dirait que non. Ce n'est
%pas nonplus \'ecrit exactement comme \c{c}a chez Hatcher... Tu connais
%une bonne r\'ef\'erence ? (Sinon c'est pas un drame: il me semble que
%c'est assez bien connu pour le dire comme \c{c}a.)  K: \c ca m'\'etonne, mais tu as raison, ce n'est pas \'ecrit comme \c ca chez Serre.  En tout cas, c'est bien connu, il ne faut pas donner une r\'ef\'erence.}

Our representation $\rho$ gives rise to a representation of each
$\pi_1(P_i)$, by using the spanning tree $T$ to identify based curves in
$P_i$ with based curves in $\Sigma_g$. Similarly, each additional edge
generator $e_j$ can be identified with a closed, based loop in $\Sigma_g$,
hence to an element $\rho(e_j)$.

We now define a geometric representation that will be our candidate for a
representation semi-conjugate to $\rho$.
As a consequence of Theorem~\ref{thm:Pants}, $\frac{2g-2}{k}$ is an
integer, hence a Fuchsian representation of Euler class $2g-2$ can be lifted
to $\PSLk$. The choice of such a lift amounts to the
choice of rotation numbers (in $\frac{1}{k}\Z~\mathrm{mod}~\Z$) for the
elements of a homology basis of $\pi_1\Sigma_g$.
Hence, there exists a geometric representation $\rho_0$ with the same
Euler class as $\rho$, and with $\rot(\rho_0(\gamma))=\rot(\rho(\gamma))$ for
each $\gamma$ in $\{c_1,\ldots,c_g,e_1,\ldots,e_g\}$ (with the notation of
Figure~\ref{fig:GrapheDeGroupes}).
This also holds for each $\gamma\in\{c_{g+1},\ldots,c_{3g-3}\}$. Indeed,
the contribution of the Euler class of $\rho$ and $\rho_0$ on each
pairs of pants are equal, and they are sums of rotation numbers, so we
can propagate these equalities to the whole family of cutting curves.

To show $\rho$ and $\rho_0$ are semi-conjugate, we use (an adaptation of)
Matsumoto's theory of basic partitions and combinations.

\begin{definition}[Matsumoto \cite{Matsumoto16}]
  Let $\Gamma$ be a group generated by a finite symmetric set $S$, and let
  $\rho: \Gamma \to \HOS$.
  A \emph{Basic Partition (BP)} for $\rho(\Gamma)$ is a collection $P$ of
  disjoint closed intervals of $S^1$ satisfying
  \begin{enumerate}[i)]
  \item for each $I \in P$, there is a unique $s_I \in S$ such that
    $\rho(s_I)(I)$ is a union of $m = m(I)$ elements of $P$ and $m -1$
    complementary intervals to $P$,
  \item for any $s \neq s_I$ in $S$, the image $\rho(s)(I)$ is a proper
    subset of an element of $P$, and
  \item for any complementary interval $J$ to $P$ and $s \in S$, either
    $\rho(s)(I)$ is contained in the interior of $P$, or is a complementary
    interval to $P$.
  \end{enumerate}
\end{definition}
Following the last condition, we may put the complementary intervals to $P$
into a directed graph, with an edge from $J_1$ to $J_2$ if
there is a generator sending $J_1$ to $J_2$.   A Basic Partition is
called {\em pure} if this graph consists of disjoint nontrivial cycles.

Applying this to our context, for each pair of pants $P_i$, choose two
``preferred'' boundary components as generators for $\pi_1 P_i$
(identified with a subgroup of $\Gamma_g$ via $T$). Let $a_i^{-1}$ and
$c_i^{-1}$ denote these elements, and consider their images under $\rho$.
%Let $P_i$ be a pair of pants of the decomposition, and let $a^{-1}$,
%$c^{-1}$ and $ac$ its generators (with $a^{-1}$ and $c^{-1}$ being the two
%prefered boundary components).
The proof of Theorem~\ref{thm:Pants}
shows that the periodic points of $a_i$, $c_i$ and $c_ia_i$ are in the
cyclic order
\[ \bbrac{a_i^-,a_i^+,(a_ic_i)^+,(a_ic_i)^-,c_i^-,c_i^+,(c_ia_i)^+,(c_ia_i)^-}_k \]
and that the $4k$ intervals of types $(a_i^+,(a_ic_i)^+)$, $((a_ic_i)^-,c_i^-)$,
$(c_i^+,(c_ia_i)^+)$ and $((c_ia_i)^-,a_i^-)$ form a
{\em pure basic partition} for the action of $\pi_1 P_i$
on the circle with
respect to the symmetric generating set $(a_i,c_i,a_i^{-1},c_i^{-1})$.
This conclusion rested only upon rigidity and the hypothesis $S_k$, and the combinatorics of the BP (the images of intervals and complementary intervals following conditions i-iii of the definition) depends only on the rotation numbers of the generators.  Thus, $\rho_0$ admits a basic partition with the same combinatorics as $\rho$, i.e. there exists a cyclic order preserving map
sending the basic partition of one to the other, which intertwines the two actions.   In this case,  \cite[Theorem~4.7]{Matsumoto16}  states that the restrictions of $\rho$ and $\rho_0$ to $\pi_1 P_i$ are semi-conjugate. 

It remains to improve this to a global semi-conjugacy between $\rho$ and
$\rho_0$.
With the notation above, in a pair of pants $P_i$, let
$J_a$ (resp, $J_c$, resp. $J_{ac}$) denote the union of all intervals of
type $(a_-,a_+)$ (resp $(c_-,c_+)$, resp. $((ac)_+,(ac)_-)$.
These are called the {\em entries} of the Basic Partition described above;
their stabilizers in $\pi_1 P_i$ are the cyclic groups
generated by $a$, $c$ and $ac$, respectively.

Now consider an edge $e_j$ of $G$ (in $T$ or not). It serves to conjugate
one generator of $\pi_1 P_i$, for some $i$, $a^{-1}$, $c^{-1}$ or $ac$,
into {\em the inverse} of the corresponding generator of this boundary component on the adjacent pair
of pants.
%\note{M: Figure ?  K: pas n\'ecessarie.}
It follows that, if, say, $a_i$ and $a_{i'}$ are the generators of
$\pi_1 P_i$ and $\pi_1 P_{i'}$ on each side of an edge $e_j$, then the images
$J_{a_i}$ and $\rho(e_j)(J_{a_{i'}})$ form a partition of $S^1$,
up to the finitely many periodic points of $a_i$. 
In this situation, Matsumoto says that the two entrances $J_{a_i}$ and
$J_{a_{i'}}$ are {\em combinable}.  
More generally, given a graph of groups decomposition of a group $\Gamma$ as ours, and pure Basic Partitions for each vertex group that have
combinable entrances for every edge, Matsumoto says the collection of all Basic Partitions for the vertex groups form a {\em Basic Configuration} for the action $\rho(\Gamma)$ on the circle.  (Matsumoto works with trees of groups; but this definition generalizes immediately to the graph setting.)  

As we already argued for the $\pi_1 P_i$, the equalities between rotation numbers of $\rho$ and $\rho_0$ on
the curves $C_i$ and on the edge elements $e_j$ imply that they admit Basic
Configurations with the same combinatorics; in other words there exists a
cyclic order preserving bijection which maps the Basic Partitions of $\rho$ to those of  
$\rho_0$, intertwining the actions.

Matsumoto's main result is that a cyclic order preserving bijection
between Basic Configurations can be promoted to a semi-conjugacy 
between $\rho$ and $\rho_0$ \cite[Theorem 6.7]{Matsumoto16}.   We comment briefly on the proof.  
To produce a semi-conjugacy, it suffices to show
that some orbit of $\rho$ and some orbit
of $\rho_0$ are in the same cyclic order.  Matsumoto's proof strategy begins by 
showing this property holds for elements of vertex groups (ie of some $\pi_1 P_i$) -- this is the content of
\cite[Theorem 4.7]{Matsumoto16} cited above.
%These vertex group elements are exactly those whose
%translation distance, as isometries of the Bass-Serre tree corresponding
%to the graph of groups decomposition of $\pi_1\Sigma_g$, 
%is zero.
He then proceeds with elements
%with translation distance is $1$, ie, words
of the form $\gamma_i e_j \gamma_{i'}$ (where $\gamma_i\in P_i$
and $\gamma_{i'}\in P_{i'}$ belong to adjacent pairs of pants),
then of the form $\gamma_{i_3}e_{j_2}\gamma_{i_2}e_{j_1}\gamma_{i_1}$,
and so
on, inductively.
While
%the
his
proof is
not
carried out
in
%without
the language of Bass-Serre theory, 
and the context is specialized to a tree of groups
decompositions of $\pi_1\Sigma_g$, the arguments adapt without
modification.
%\note{M: Je me permets d'affirmer \c{c}a... Discutons-en
%si tu veux.  K: oui, discutons-en!}

%%%%%%%%%%%%%%%%%%%%%%%%%%%%%%%%%%%%%%%%%%%%%%%%%%%
%%%%%%%%%%%%%%%%%%%%%%%%%%%%%%%%%

%This, together with a few other minor technical conditions (for brevity, we do not give the full list here, but refer the reader to \cite[Assumption~6.3]{Matsumoto16}), ensures that the vertex-group BP's satisfy a ``combination theorem'' reminiscent of the Klein--Maskit combination theorem, for any two vertex groups that share an edge.

%The result is that the combinatorial data of a BC determines the cyclic order of the orbit of a point under $\rho(G)$, and hence determines the semi-conjugacy class of the group action.  More precisely, given a decomposition of $G$ as a tree of groups and $\rho_1$, $\rho_2 \in \Hom(G, \HOS)$, Matsumoto shows the following.  

%%%%%%%%%%%%%%%%%%%%%%%%%%%%%%%%%%%%%%%%%%%%%%%%%%%%%%%%%%%%%%%%%%%%%%
%%%%%%%%%%%%%%%%%%%%%%%%%%%%%%%%%%%%%%%%%%%%%%%%%%%%%%%%%%%%%%%%%%%%%%
%%%%%%%%%%%%%%%                                   %%%%%%%%%%%%%%%%%%%%
%%%%%%%%%%%%%%%  IV. PERIODIC CONSIDERATIONS      %%%%%%%%%%%%%%%%%%%%
%%%%%%%%%%%%%%%                                   %%%%%%%%%%%%%%%%%%%%
%%%%%%%%%%%%%%%%%%%%%%%%%%%%%%%%%%%%%%%%%%%%%%%%%%%%%%%%%%%%%%%%%%%%%%
%%%%%%%%%%%%%%%%%%%%%%%%%%%%%%%%%%%%%%%%%%%%%%%%%%%%%%%%%%%%%%%%%%%%%%

\section{Periodic considerations}\label{sec:Per}

The content of this section is the proof of the following two statements.

\begin{proposition}\label{prop:RatRot}
  If a representation $\pi_1\Sigma_g\rightarrow G$ is path-rigid, then all
  nonseparating simple closed curves have rational rotation number.
\end{proposition}

\begin{theorem}\label{theo:PerDisjointsSk}
  Suppose $\rho$ is path-rigid and minimal. Then, for all $a,b$ with
  $i(a,b)=\pm 1$, we have the implication
  \[ \Per(a)\cap\Per(b)=\emptyset \Rightarrow S_k(a,b)\text{ for some }k. \]
\end{theorem}

\begin{proof}[Proof of Proposition~\ref{prop:RatRot}]
  Suppose for contraction that there exists a non-separating simple
  curve $a$ with $\rho(a) \notin \Q$.   After semi-conjugacy, we may assume that $\rho$ is minimal.  If $\rho(a)$ is conjugate into $\SO(2)$, then it lies in a 1-parameter subgroup $a_t$ of rotations, and for any $b$ with $i(a,b)=1$, the bending deformation $a_t\rho(b)$ has nonconstant rotation number, contradicting rigidity.  
  Thus, $\rho(a)$ has an invariant minimal cantor set, which we denote by $K$.  We next show that $K$ is $\rho(b)$-invariant, for any curve $b$ with $i(a,b)=1$.  This suffices to prove the Proposition since $\Gamma_g$ is generated by $\{a\} \cup \{b \mid i(a,b)=1\}$, whence $K$ is $\rho(\Gamma_g)$-invariant, contradicting minimality of $\rho$.  
 
 To show invariance, suppose for contradiction that $\rho(b)(K)\not\subset K$ (the case $\rho(b^{-1})(K)\not\subset K$ is analogous).  
 Let $K' \subset K$ be the set of two-sided accumulation points of $K$.  
  Since $\overline{K'}=K$, there exists $x \in K'$ such that
  $\rho(b)(x)\not\in K$.
  Let $I$ be the connected component of $S^1\smallsetminus K$
  containing $\rho(b)(x)$.  Minimality of the action of $\rho(a)$ on $K$ implies there exists $N\in\Z$ such that $\rho(a)^N(I)\subset\rho(b)^{-1}(I)$, and in particular $\rot(\rho(a^Nb)) = 0$.  We work now with the pair $(a, a^Nb)$ with intersection number $\pm1$.  
  Let $\beta_t$ be a positive one-parameter family commuting with $\rho(a^N b)$. 
  Since $\rho(a^N b)$ does not preserve $K$, we can find a  
  connected component $J$ of $S^1 \smallsetminus \Fix(\rho(a^N b))$ such that $J \cap K' \neq \emptyset$ and then
  find $M \in \Z$ such that $\rho(a)^M(J) \cap J \neq \emptyset$. 
  
  Let $\tilde{x} \in \R$ be a lift of a point in $\rho(a)^M(J) \cap J$.  
  Adapting the notation from Section~\ref{sssec:Twist}
  set
  \[ \Delta(\tilde{x},t_1,\ldots,t_M)=\widehat{\beta_{t_M}}\widehat{\rho(a)}
  \circ\cdots\circ\widehat{\beta_{t_1}}\widehat{\rho(a)}(\tilde{x})-\tilde{x}
  -k, \] 
   where $k$ is chosen so that $\widehat{\rho(a)}^M(\widetilde{J}) \cap (\widetilde{J}+k) \neq \emptyset$ for any lift of $J$,
    and we set
  $\delta(\tilde{x},t)=\Delta(\tilde{x},t,\ldots,t)$.
  Up to reversing orientation, we can suppose that
  $\delta(\tilde{x},0)>0$.
  Since $\tilde{J}$ contains both $\tilde{x}$ and
  $\widehat{\rho(a)}^M(\tilde{x})$, there exists
  $t<0$ such that $\Delta(\tilde{x},0,\ldots,0,t)<0$, hence
  $\delta(\tilde{x},t)<0$. Thus, there exists $t_0$ such that
  $\delta(\tilde{x},t_0)=0$, hence $\rot(\rho_{t_0}(a))=\frac{k}{M} \in \Q$, contradicting rigidity.   \end{proof}

%%%%%%%%%%%%%%%%%%%%%%%%%%%%%%%%%%%%%%%%%%%%%%%%%%%%%%%%%%%%%%%%%%%%%
%%%%%%%%%%%%%%%%%%%%%%%%%%%%%%%%%%%%%%%%%%%%%%%%%%%%%%%%%%%%%%%%%%%%%
%%%%%%%%%%%%%%%%%%%%%%%%%%%%%%%%%%%%%%%%%%%%%%%%%%%%%%%%%%%%%%%%%%%%%

\subsection{Proof of Theorem~\ref{theo:PerDisjointsSk}}

For this subsection, we assume $\rho$ is path-rigid, $i(a,b) = \pm 1$, and $\Per(a) \cap \Per(b) = \emptyset$.  Recall from Proposition \ref{prop:RatRot} that $\Per(a)$ and $\Per(b)$ are nonempty. 
We will first establish some properties that do not use minimality, so
are robust under deformations of $\rho$.  We add the hypothesis that $\rho$ is
minimal only at the end of the proof.   

%The strategy of the proof is to first show that $a$ and $b$ \emph{look like} they are semi-conjugate to elements satisfying $S_k(a,b)$, namely, if one takes each interval in $S^1 \smallsetminus  \Per(a)$ that does not contain a point of $\Per(b)$ and collapses its closure to a point (and then repeats with the roles of $a$ and $b$ reversed), the remaining periodic points are attracting and repelling and alternate around the circle.  We will then show that this collapse can be realized by a deformation of $\rho$, and we attain $S_k(a,b)$ if $\rho$ is minimal.
%Much of this work will depend on the observation that the hypothesis
%that $\Per(a) \cap \Per(b) = \emptyset$ is invariant under deformations of the action. 

%We begin by discussing basic combinatorics of periodic sets.  
Borrowing notation from the previous section, say that 
a connected component of $S^1 \smallsetminus (\Per(a) \cup \Per(b))$
is of type $(x,y)$ if
%with its natural orientation from the circle, 
it is bounded to the left by a point of $\Per(x)$ and to the right by a point of
$\Per(y)$, for $x, y \in \{a, b\}$.  

\begin{definition}
Let $X_a$ denote the set of connected components of $S^1 \smallsetminus \Per(a)$ that contain points of $\Per(b)$.  
We say an element $I$ of $X_a$ is \emph{positive} if $a^{q(a)}$ is
increasing on the interval $I$, and \emph{negative} otherwise.  
 \end{definition} 
 
The set $X_b$ and its positive and negative elements are defined by reversing the roles of $a$ and $b$ above.  
Since each $(a, b)$ interval in $S^1 \smallsetminus (\Per(a) \cap \Per(b))$ is followed by a collection -- perhaps empty --
of $(b, b)$ intervals, and then a $(b, a)$ interval, and $\Per(a)$ and $\Per(b)$ disjoint closed sets,
there exists an integer $m  = m(\rho) \geq 1$ such that $S^1$ contains exactly $m$ intervals of type
$(a, b)$ and $m$ intervals of type $(b, a)$, alternating around the circle, and thus $| X_a | = |X_b| = m(\rho)$.  
By Remark~\ref{rem:PerNonCommun}, $m$ depends only on the
semi-conjugacy class of~$\rho$.

\begin{lemma}\label{lem:aPreserveSa}
  The set $X_a$ is $\rho(a)$-invariant, and the subset of positive
  (respectively, negative) intervals in $X_a$ is also $\rho(a)$-invariant.   
\end{lemma}

\begin{proof}
  Let $I \in X_a$ be a positive interval; we show that its image under
  $a$ is another positive interval in $X_a$.  The negative case is analogous.  
  Since $a(I)$ is an interval between two consecutive points of $\Per(a)$ on which
  $a^{q(a)}$ is increasing, we need only show that $a(I) \cap \Per(b) \neq \emptyset$.   
  
  Suppose for contradiction that $a(I) \cap \Per(b) = \emptyset$.  
  Then $a(\overline{I})\subset J$ for some $J\in X_b$.   Let $b_t$ be a positive one-parameter family commuting with $b$, 
  let $x\in I \cap \Per(b)$, and take lifts $\tilde{x}\in\tilde{I}$ of $x$ and $I$ to $\R$.
  Positivity implies $\delta_{b,a}(\tilde{x},0)>0$.  
  If $t<0$ is negative enough that
  $b_t(a(I)) \cap a(I) = \emptyset$, then we have
  $\widehat{b_t}(\widehat{a}(\tilde{x})) < \widehat{a}(\tilde{I})$;
  it follows that $\delta_{b,a}(x,t)<0$.
  %\textcolor{blue}{hence 
  %$\widehat{a}^{q(a)-1}(\widehat{b_t}(\widehat{a}(\tilde{x}))) < \tilde{I}+k$.}  \note{K: j'enleve volontiers cette phrase.  Sinon, il faut d\'efinir $k$} 
  %It follows that $\Delta_{b,a}(\tilde{x},t,0,\ldots,0)<0$, so $\delta_{b,a}(\tilde{x},t)<0$.
  Therefore, there exists $t_0\in\R$ such
  that $\delta_{b,a}(x,t_0)=0$, i.e. $x\in\Per(b_{t_0}a)\cap\Per(b)$.  This contradicts path-rigidity via
  Lemma~\ref{lem:PerCommuns}.
\end{proof}

Obviously, reversing the roles of $a$ and $b$ above shows the positive and negative intervals of $X_b$ are $b$-invariant.
%Remark \ref{rem:PerNonCommun} and path-rigidity of $\rho$ showed that the cardinality of $X_a$ and $X_b$ are constant under deformation.  
The next lemma shows $X_a$ and $X_b$ are invariant under particular bending deformations.  

\begin{lemma}\label{lem:SbStable}
  Let $b_t$ be a positive one-parameter family commuting with $b$. For $t\in\R$,
  let $X_b(t)$ denote the set of connected components $I$ of
  $S^1\smallsetminus\Per(b)$ such that $I\cap\Per(b_ta)\neq\emptyset$.
  Then $X_b(t)=X_b(0)$ for all $t$.
\end{lemma}

\begin{proof}
  Let $X_b(t)$ be as in the statement of the lemma, and let $X_a(t)$ denote
  the set of connected components of $S^1 \smallsetminus \Per(b_ta)$
  containing points of $\Per(b)$, for $t \in \R$.   By our discussion above, path-rigidity of $\rho$ implies that the cardinality of $X_b(t)$ is constant.  
  Let $K_a = \{ (x, t) \in S^1 \times \R \mid x \in \Per(b_ta)\}$,
  and $K_b = \Per(b) \times \R$.  These are closed, disjoint sets, and their
  intersections with each horizontal slice $S^1 \times \{t\}$ are the periodic sets of $b_t a$ and $b$, respectively. 
     
  For each connected component $I\subset S^1\smallsetminus\Per(b)$, we set
  \[ T_I=\lbrace t\in\R \mid  I\in X_b(t)\rbrace
    = \lbrace t\in\R \mid I\cap\Per(b_ta)\neq\emptyset\rbrace. \]
  Note that $T_I$ is the projection of $K_a \cap (\overline{I}\times\R)$ onto the $\R$-factor, so in particular is closed.   
  We claim $T_I$ is also open. To see this, let $t_0\in T_I$, and let
  $I_2,\ldots, I_m$ be the other components of $S^1 \smallsetminus \Per(b)$ such that $t_0\in T_{I_j}$.
  If $d>0$ is the distance (for the product metric) between
  the disjoint compact sets $(S^1\times[t_0-1,t_0+1])\cap K_a$ and
  $(S^1\times[t_0-1,t_0+1])\cap K_b$, let $I_{m+1},\ldots, I_N$ be the
  remaining connected components of $S^1\smallsetminus\Per(b)$ of length $\geq d$.
  Any component $J$ of shorter length 
  tautologically satisfies $T_J\cap[t_0-1,t_0+1]=\emptyset$.  Since the sets
  $T_{I_j}$ are closed, there exists $\varepsilon>0$ such that
  $(t_0-\varepsilon, t_0+\varepsilon)\cap T_{I_j}=\emptyset$ for all $j\geq m+1$, hence $(t_0-\varepsilon, t_0+\varepsilon) \subset T_I$, for otherwise $|X_b(t)|$ would fail to be constant.
  This proves that $T_I$ is open, hence equal to $\emptyset$ or $\R$,
  and the intervals in $X_b(t)$ do not depend on~$t$.
\end{proof}

The next two lemmas establish some properties of $a$ and $b$ which are, in particular, held by pairs of homeomorphisms semi-conjugate to hyperbolic elements of $\PSLk$ satisfying $S_k(a,b)$.  Of course, both lemmas also hold with the roles of $a$ and $b$
exchanged.

\begin{lemma}\label{lem:DynAlt}
  Any two consecutive intervals of $X_a$ %(for the natural cyclic order from $S^1$) 
  have opposite sign.
  In particular, $m(\rho)=2k$ for some $k\geq 1$.  
%  and the positive and negative intervals of $X_a$ alternate.  
\end{lemma}
\begin{proof}
  Let $b_t$ be a positive one-parameter family commuting
  with $\rho(b)$.
  Suppose for contradiction that $X_a$ has two successive positive
  intervals $I_1$ and $I_2$ (the negative case is analogous). 
  Let $I\in X_b$ be the interval such that $I_1 \cap I \neq \emptyset$
  and $I_2 \cap I \neq \emptyset$.
  Take $x\in I_1\smallsetminus I$ such that $a^{q(a)}(x)\in I$.
  For $t$ large enough, we have 
  $a^{q(a)} b_t a^{q(a)}(x)\in I_2\smallsetminus I$. Since $b_t$ has
  positive dynamics, it follows that $(b_t a^{q(a)})^2$ moves every point
  of $I$ to the right, thus, $\Delta_{b,a}(y,0,\ldots,0,t)>0$ for all
  $y\in I$, and $\Per(b_ta)\cap I=\emptyset$ for $t$ large
  enough.  But this contradicts Lemma~\ref{lem:SbStable}.
\end{proof}

%As a consequence of this lemma, the intervals of $X_a$ and $X_b$ can be
%labelled --borrowing the notation from Secion~\ref{sec:Exercices}-- in cyclic order 
%%(the cyclic order in which their leftmost points appear on $S^1$), 
%as
%$\bbrac{I_a^+,I_b^+,I_a^-,I_b^-}_k$ or
%$\bbrac{I_a^+,I_b^-,I_a^-,I_b^+}_k$.

\begin{lemma}\label{lem:Preserv2}
 Let $I \in X_b$ have left endpoint in a positive interval of $X_a$.  Then $a(I) \subset J$ for some $J \in X_b$.
If instead $I \in X_b$ has left endpoint in a negative interval of $X_a$, then $a^{-1}(I) \subset J$ for some $J \in X_b$.
    \end{lemma}
Note that Lemma \ref{lem:DynAlt} implies that, in both cases, 
$J$ is a positive interval of $X_b$ if and only if $I$ is.  

\begin{proof} 
  Let $x_1, x_2, \ldots x_6$ be points in cyclic order such that $(x_1,x_3)$
  and $(x_4,x_6)$ are consecutive (positive and negative, respectively)
  intervals in $X_a$, and $I = (x_2,x_5) \in X_b$.
  Let $y_i = a(x_i)$ for $i = 1, 3, 4, 6$.
  Then $(y_1,y_3)$ and $(y_4,y_6) \in X_a$ and both intersect some
  interval of $X_b$,
  say $(y_2,y_5)$.
  The statement of the lemma is that $a(x_5)\leq y_5$ and $a(x_2)\geq y_2$.
  
  Similar to the Proof of Lemma \ref{lem:aPreserveSa}, we assume the contrary
  and find a deformation with a common periodic point for $a$ and $b$.
  Suppose $a(x_5)>y_5$ (the proof of the
  other inequality is symmetric), and choose a positive one-parameter family
  $b_t$ commuting with $b$.
  Since $a^{-1}(y_5)\in(x_2,x_5)$, there exists $t\in\R$ such that
  $b_ta^{-1}(y_5)\in(x_1,x_3)$. As $(y_1,y_3)$ is
  $a^{q(a)}$-invariant, it follows that $a^{-q(a)+1}b_ta^{-1}(y_5)<y_5$,
  ie, $\Delta_{b,a}(y_5,0,\ldots,0,t,0)>0$. On the other hand, as
  $(y_4,y_6)$ is a negative interval of $X_a$, we have
  $\delta_{b,a}(y_5,0)<0$. Thus, there exists $t_0\in\R$, such that
  $y_5 \in \Per(b_{t_0}a)$.  Since $y_5 \in \Per(b)$, this contradicts
  path-rigidity by Lemma~\ref{lem:PerCommuns}. The statement concerning $\rho(a)^{-1}$
  is symmetric, and proved in the same manner.  
\end{proof}

Now we state a lemma of purely technical nature, that will allow us
to compress the periodic sets in each interval of $X_a$ or of $X_b$ to
singletons. 
In the statement and proof, we let $\tau_t\colon\R\to\R$ denote the
translation $x\mapsto x+t$.

\begin{lemma}\label{lem:EHN}
  Let $n\geq 1$, and for all $i=1,\ldots,n$, let $f_i$ be an increasing
  homeomorphism from $\R$ to some interval $(a_i,b_i)\subset\R$.
  Assume that $a_i>-\infty$ for at least one $i$, and that $b_j<+\infty$
  for at least one $j$.
  For all $t\in\R$, we set
  $F_t=\tau_t\circ f_n\circ\cdots\circ\tau_t\circ f_1$.
  Then, there exists a subset $N\subset\R$ of finite Lebesgue measure and
  consisting of a countable union of segments, such that for
  all $t\not\in N$, the map $F_t$ admits a unique fixed point in $\R$.
\end{lemma}

The
statement of this Lemma came from our attempt to better
understand the argument in the first four lines of page~644 in \cite{EHN}.
In particular, the case $n=1$ gives an alternative end to the proof of \cite[Lemma~2.7]{EHN}.
We defer the proof of Lemma \ref{lem:EHN}
to the next paragraph, and use it
now to finish the proof of Theorem~\ref{theo:PerDisjointsSk}.

\begin{proof}[Proof of Theorem~\ref{theo:PerDisjointsSk}]
  Assume now that $\rho$ is minimal.
  Let $b_t$ be a positive one-parameter family commuting with 
  $b$. We will first find $t$ such that $b_ta$ has exactly $2k$ periodic
  points; the conclusion will then follow easily.

  Let $X_a^+$ denote the set of positive intervals of $X_a$.
  As observed in Lemma~\ref{lem:aPreserveSa}, 
  $\rho(a)$ induces a permutation of $X_a^+$; which has 
  say, $\ell$ orbits, all of cardinality $n=k/\ell$.
  Fix an interval $I_0 \in X_b$ whose left endpoint lies in an element of $X_a^+$.  Successive applications of  
  Lemma \ref{lem:Preserv2}, for $m = 1, 2, \ldots n-1$, gives 
  $\rho(a)^m(I_0) \subset I_m$ for some $I_m \in X_b$.  Also, $\rho(a)^n(I_0) \subset I_0$ because $\rho(a)^n$ fixes $X_a^+$.  
  Note that, there exists some $m$ such that
  $\rho(a)(I_{m-1}) \subset I_m$ is a strict inclusion at the left endpoint
  (and similarly, another for the right endpoint) as otherwise some endpoint of $I_0$ would lie in $\Per(a) \cap \Per(b)$.
  
  For each $j$, let $\phi_j\colon I_j\to\R$ be a homeomorphism such that
  $\phi_j\circ b_t\circ\phi_j^{-1} = \tau_t$, and for $j\in\{1,\ldots,n\}$,
  set $f_j=\phi_{j+1}\circ a \circ\phi_j^{-1}$, using cyclic notation for the
  last index.
  Then, Lemma~\ref{lem:EHN} applies, giving a set $N_{I_0}\subset\R$
  of finite Lebesgue measure, such that for all $t\not\in N_{I_0}$,
  $(b_ta)^n=\phi_1^{-1}\circ F_t\circ\phi_1$ has a unique
  fixed point in $I_0$.
  
  We repeat this procedure for each element $I$ of $X_b$, using $a^{-1}$,
  instead of $a$ for the intervals of $X_b$ whose left endpoint lies in an
 element of $X_a^-$.
  The resulting, finitely many, sets $N_I$, each of finite Lebesgue
  measure, cannot cover $\R$, hence there exists $t\in\R$ such that each
  element of $X_b$ intersects $\Per(b_t a)$ as a singleton.  
  By Lemma~\ref{lem:SbStable}, $\Per(b_t a) \subset \cup(X_b)$, hence $b_ta$ has exactly $2k$ periodic points.
  As $b_ta$ is obtained by a bending deformation that does not change $a$,
  by Lemma~\ref{lem:DynAlt} these $2k$ periodic points have alternating
  attracting and repelling dynamics.
  One may now repeat the same procedure reversing the roles of $a$ and $b$,
  to obtain a further deformation where $b$ has exactly $2k$ periodic points,
  alternately attracting and repelling.  Minimality of $\rho$ and
  Observation~\ref{obs:CardinalityPer} implies the original action of $\rho(a)$ and $\rho(b)$ also had this dynamics.  
\end{proof}

%\subsection{Proof of Lemma~\ref{lem:EHN}}

%To give an outline of the major ideas, we sketch the proof in the (easier) case $n=1$.  

%Suppose, for all $t,x\in\R$ we have $F_t(x)=f(x)+t$, where $f$ is an
%increasing
%homeomorphism from $\R$ to a bounded interval $(a,b)$. Define $T(x):=x-f(x)$;
%so $T(x)$ is the unique number such that $x$ is a fixed point of $F_{T(x)}$.
%We want to prove that most $t$ are realized as $T(x)$ by a unique $x$.
%If $T(x_1)=T(x_2)$ for some $x_1 < x_2$, then we have
%$f(x_2)-f(x_1)=x_2-x_1$, ie, the intervals $[f(x_1),f(x_2)]$
%and $[x_1,x_2]$ have the same length, say $\ell$.
%Since $T$ cannot increase faster than the identity map, all elements
%$x\in[x_1,x_2]$ have $|T(x)-T(x_1)|\leq\ell$. 
%In summary, although we have found a segment $[T(x_1) - \ell, T(x_1) + \ell]$ of length $2\ell$ that contains $t$'s which may have
%several preimages under $T$, it has cost us the length $\ell$
%in the image of $f$, and this reservoir is finite.
%Now let us turn this into a precise proof, in the general case $n\geq 1$.}

\begin{proof}[Proof of Lemma~\ref{lem:EHN}]
We suggest the reader takes $n=1$ at first reading, as the argument is
less technical in that case.
We will show that there exists a countable union of segments,
$N_+\subset\R_+$, of finite Lebesgue measure, such that $F_t$ has a unique fixed point for all
$t\in\R_+\smallsetminus N_+$.
The case for $t<0$ is symmetric and left to the reader.

Let $j$ be an index such that $b_j<+\infty$.
Let $A_t=\tau_t\circ f_j\circ\cdots\circ\tau_t\circ f_1$, and
$B_t=\tau_t\circ f_n\circ\cdots\circ\tau_t\circ f_{j+1}$.
For fixed $t$, both maps $A_t$ and $B_t$ are
homeomorphisms to their images so 
$F_t=B_t\circ A_t$ has a unique fixed point $x$ if and only if $A_t\circ B_t$
has a unique fixed point (namely, $B_t(x)$).  
In other words, we may suppose without loss of generality that $j=n$. 
%For the $t<0$ case, one supposes instead that $a_n>-\infty$.

Let
$G(t,x)
%f_n\circ\tau_t\circ f_{n-1}\circ\cdots\circ\tau_t\circ f_1(x)
=F_t(x)-t$.
Then $G$ is strictly increasing in $x$, and increasing
(strictly, if $n\geq 2$) in $t$.
The monotonicity of $G$, and the assumptions $\sup(a_j)>-\infty$ and
$b_n<+\infty$, imply that the range of the map
$G\colon\R_{\geq 0}\times\R\to\R$ is a bounded interval, say $(a_0,b_0)$, where $b_0=b_n$.

If $x\geq b_0$, the map $t\mapsto F_t(x)$ is a homeomorphism between
$\R_{\geq 0}$ and $[F_0(x),+\infty)$, and $F_0(x)=G(x,0)<b_0$.
Hence, there is a unique $t = T(x)$ such that $F_t(x) = x$.
This defines a function $T\colon[b_0,+\infty)\to(0,+\infty)$.

\begin{sublemma}\label{lem:PropertyT} The map $T$ satisfies the following
  inequalities.
  \begin{itemize}
  \item[(T1)]For every $x\in[b_0,+\infty)$, we have $a_0 < x-T(x) < b_0$.
  \item[(T2)]For all $x_1,x_2\in[b_0,+\infty)$ such that $x_1<x_2$, we have
    \[ f_1(x_1)-f_1(x_2)<T(x_2)-T(x_1)<x_2-x_1. \]
  \end{itemize}
\end{sublemma}
In particular, $T$ is continuous, at bounded distance from the identity,
and its rate of increase is bounded above by 1.

The proof of Lemma \ref{lem:PropertyT} is a straightforward consequence of the definition of $T$, the defining identity $F_{T(x)}(x)=x$, and monotonicity of $G$.  We leave it as an exercise, noting for (T2) that the first inequality is trivially satisfied if $T(x_2)\geq T(x_1)$, and the second if $T(x_2)\leq T(x_1)$.

For the next step, define a map $H\colon\R_{\geq b_0}\to [T(b_0),+\infty)$ by
$H(x)=\sup\{T(x'),\, x'\leq x \}$.  The reader may verify (easily) that $H$ is 
%
%\begin{lemma}\label{lem:PropertyH}
%  The map $H\colon\R_{\geq b_0}\to[T(b_0),+\infty)$ is 
continuous, surjective,
  and for all $A\geq T(b_0)$, the set $H^{-1}(A)$ is a segment of the form $[a,b]$,
  (possibly $a=b$) with $T(a)=T(b)=A$.
%\end{lemma}
%This is an easy exercise in undergraduate analysis; we leave the details to
%the reader.

Now let $W = \{w \in[T(b_0),+\infty) \mid H^{-1}(w) \text{ is not a singleton} \}$, and for all
$w\in W$ denote $H^{-1}(w)$ by $[a_w,b_w]$.  Since these segments are disjoint and of positive length, $W$ is countable.  
By definition of $H$, we have $F_w(a_w)=a_w$, i.e. $G(w,a_w)+w=a_w$; and the same holds for $b_w$ in place of $a_w$.  
Thus, the segment
$[G(w,a_w),G(w,b_w)]$ has the same length $b_w-a_w$.  The reader may now easily deduce from monotonicity of $G$ that these segments are disjoint; as they are contained in $[a_0, b_0]$, this implies $\sum_{w\in W}b_w-a_w\leq b_0-a_0$.  

%\begin{sublemma}\label{lem:WaterLevel}
%  The set $W$ is countable, and $\sum_{w\in W}b_w-a_w\leq b_0-a_0$.
%\end{sublemma}
%
%\begin{proof}
%  First, since the segments $[a_w,b_w]$ all have positive length and are
%  disjoint in $[b_0,+\infty)$, there can only be countably many of them.
%  Now, for any $w\in W$, we have $F_w(a_w)=a_w$ and $F_w(b_w)=b_w$, by
%  Lemma~\ref{lem:PropertyH}. In other words, $G(w,a_w)+w=a_w$ and
%  $G(w,b_w)+w=b_w$. Thus,
%  the segment
%  $[G(w,a_w),G(w,b_w)]$ has the same length $b_w-a_w$.
%  
%  Since the image of $G$ is an interval of length $b_0-a_0$,
%  it suffices to prove that the
%  segments $[G(w,a_w),G(w,b_w)]$ are pairwise disjoint. 
%  To do this, take $w_1, w_2\in W$, with $w_1<w_2$.
%  Since $H$ is increasing, we have $b_{w_1}<a_{w_2}$.
%  By monotonicity of $G$, we have 
%  $G(w_1,b_{w_1})<G(w_2,a_{w_2})$ and these segments are disjoint indeed.
%\end{proof}

Finally, for all $w\in W$, define $N_w:=[w-(b_w-a_w),w]$, and define
\[  N_+ = [0,b_0-a_0]\cup \bigcup_{w\in W}N_w.\]
This may not be a disjoint union, but, the remarks above imply 
%by Lemma~\ref{lem:WaterLevel},
this countable union of segments has finite Lebesgue measure. 
Hence,
the proof of Lemma~\ref{lem:EHN} amounts to the following sublemma.
\begin{sublemma}
  For all $t\in\R_{\geq 0}\smallsetminus N_+$, the map
  $F_t$ has a {\em unique} fixed point.
\end{sublemma}

\begin{proof}
  Let
  $t> b_0-a_0$ be such that $F_t$ has at least two distinct fixed points, say
  $x_1,x_2$ with $x_1<x_2$.  %We want to prove that $t\in N_w$ for some $w\in W$.
  By definition, these satisfy $G(t,x_i)+t=x_i$. 
  Since $G(t,x)>a_0$ for all $x$, and
  $t>b_0-a_0$,
  this implies $x_1,x_2\in[b_0,+\infty)$.  %]$
  By definition of $T$, we have $T(x_1)=T(x_2)=t$. Let
  $x_0=\min\{x\leq x_2\,|\, T(x)=H(x_2)\}$.
  Then $x_0<x_2$. Indeed, if $H(x_2)=t$ then $x_0\leq x_1$,
  and if $H(x_2)>t$ then the maximum $H(x_2)$ is reached at some point
  to the left of $x_2$. Thus, $x_0=a_w$ for some $w\in W$, and we also
  have $b_w\geq x_2$.  
  
  We claim now that $t\in N_w$. 
  Since $x_2 \leq b_w$, by definition of $H$ we have
  $w = H(b_w) \geq t = T(x_2)$.
  Applying inequality
  (T2) to $x_2$ and $b_w$ now gives
  $w-t\leq b_w-x_2$, so $w-t\leq b_w-a_w$, hence $t\geq w-(b_w-a_w)$.
  Thus we indeed have $t\in N_w$.
\end{proof}

This concludes the proof of Lemma~\ref{lem:EHN}.
\end{proof}

\section{Proof of Theorem~\ref{theo:Main}}\label{sec:Section5}

In this section we finish the proof of the main result for path-rigid
representations, showing that a path-rigid representation $\rho$ of
$\pi_1\Sigma_g$ is either geometric, or has Euler class zero and a genus
$g-1$ subsurface whose fundamental group has finite orbit under $\rho$.
(We believe the latter case cannot actually occur.)
As in Section \ref{sec:Exercices}, we will frequently drop the notation
$\rho$ when the context is clear, using $a$ to denote $\rho(a)$.

Recall from the introduction that, if $\rho$ is a given representation
and $T\subset\Sigma_g$ is a one-holed torus, we say that $T$ is a
{\em good torus} if it contains a nonseparating simple closed curve $a$
with $\rot(a)=0$, and that $T$ is {\em bad} otherwise. We say $T$ is
{\em very good} if $\pi_1(T)$ has a finite orbit in $S^1$.
%\begin{definition} 
%  Let $\rho: \Gamma_g \to \HOS$, and let $T \subset \Sigma_g$ be a
%  one-holed torus.
%  We say $T$ is a \emph{good torus} if $T$ contains a nonseparating simple
%  closed curve
%  $a$ with $\rot(\rho(a)) = 0$, and \emph{bad} otherwise.
%  
%  We say $T$ is \emph{very good} if $\rho(\pi_1(T))$ has a
%  finite orbit in $S^1$.
%\end{definition}

Note that very good implies good:  if $T(a,b)$ is very good, then $\rot: \pi_1(T) \to \R/\Z$ is a homomorphism onto a finite subgroup, so if $0 \neq |\rot(a)| \leq |\rot(b)| < 1$, one may find $n$ such that $|\rot(a^n b)| <  |\rot(a)|$.  Iterating this process produces a simple closed curve with rotation number zero. 

%
%\begin{observation}[Very good implies good]\label{obs:TresBonImpliqueBon}
%Suppose that $T$ is very good, and let $a$ and $b$ be free generators for $\pi_1(T)$, represented by simple closed curves on $\Sigma_g$ with $i(a,b) =\pm1$. Since $T$ is very good, $\rot$ is a homomorphism to a finite subgroup of $\R/\Z$.   If $0 \neq |\rot(\rho(a))| \leq |\rot(\rho(b))| < 1$, then there exists $n$ such that $|\rot(\rho(a^n b))| <  |\rot(\rho(a))|$.   Since $a, a^nb$ are again free generators represented by simple closed curves, one may repeat this procedure until the process terminates with a simple closed curve of rotation number zero. 
%\end{observation}

\begin{assumption} \label{as:standing}
  For the remainder of this section, we assume that
  $\rho\colon\Gamma_g \to \HOS$ is path-rigid.
\end{assumption}

%%%%%%%%%%%%%%%%%

\subsection{Bad tori}\label{ssec:BadTori}

This subsection contains the proof of Proposition \ref{prop:BadToriIntro}: under Assumption \ref{as:standing} we show that if $\Sigma_g$ contains a bad torus $T$, then $\Sigma_g \smallsetminus T$ contains only very good tori.   
%We begin with some preliminaries on rotation numbers and a proof that a bad torus $T$ always contains simple closed curves with rotation number arbitrarily close to zero.  

% This means in particular that there are points
%``almost fixed'' by simple closed curves in $T$. We then study these ``almost fixed''
%points for some specific sequences of generators for $\pi_1(T)$,
%and leverage properties of these sets to show that there cannot exist two
%disjoint bad tori.

\begin{definition}
Let $f,g\in\HOZ$. We say that $g$ {\em dominates $f$}, and write $f<g$,
if $f(x)<g(x)$ for all $x\in\R$.
\end{definition} 

Note that $<$ is a left- and right-invariant partial order on
$\HOZ$, and satisfies the following obvious properties.
\begin{enumerate}
  \item $\forall f, g\in\HOZ, f > g \Leftrightarrow f^{-1} < g^{-1}$;
  \item $\forall f\in\HOS,\ \widehat{f}>\mathrm{Id}
  \Leftrightarrow \rot(f)\neq 0;$
  \item $\forall f,g\in\HOZ,
  \left\lbrace\begin{array}{l}f<g \Rightarrow \rotild(f)\leq\rotild(g),\text{ and }\\
   \left( f<g\text{ or }g<f \right) \Leftrightarrow \rotild(f^{-1}g)\neq 0.\end{array}\right.$
\end{enumerate}
Property (2) uses the notation  $\widehat{f}$ from Notation~\ref{not:Chapeau}, which is also adopted throughout this 
section.
The following easy observation will be handy; it follows directly from
Property~(3) above.
\begin{observation}\label{obs:VersRotZero}
  Let $f,g\in\HOZ$. Suppose that $\rotild(f)<\rotild(g)$ and
  $\rotild(g^{-1}f)\neq 0$. Then $f<g$.
\end{observation}

%\begin{proof}
%  If not, there exists $x_0\in\R$ such that $f(x_0)\geq g(x_0)$, hence
%  $\rotild(g^{-1}f)\geq 0$. Thus, $\rotild(g^{-1}f)>0$, and this implies
%  $g^{-1}f>\mathrm{Id}$, hence $f>g$ and $\rotild(f)\geq \rotild(g)$, a
%  contradiction.
%\end{proof}

Building on this observation, we have the following.  

\begin{lemma}\label{lem:SmallRot}
  Let $(a,b)$ be standard generators of a bad torus $T$.
  Then, there exist integers $m,n$,
  unique and well-defined modulo $q(a)$, with
  $(n-m)p(a)=1\mathrm{~mod~}q(a)$,
  and such that
  for all $j$ not divisible by $q(a)$, we have
  $\widehat{a^nb}<\widehat{a^{j}}$, and $\widecheck{a^{j}}<\widecheck{a^{m}b}$.
  Moreover, if $p(a)=1$,
  then we have $\widehat{a^nb}^2<\widehat{a}$,
  or $\widecheck{a^{n-1}b}^{-2}<\widehat{a}$, or both.
\end{lemma}

%Of course, the same statement holds with the roles of $a$ and~$b$ exchanged. 
Assumption \ref{as:standing} is used in the proof only to guarantee that all non-separating simple closed curves have rational rotation number (Prop.~\ref{prop:RatRot}). 

\begin{proof}
  Let $F$ be a finite orbit of $a$. If there exists some point
  $x\in F\cap b^{-1}(F)$, then there exists $N>0$ such that
  $\rho(a)^N \rho(b)(x)=x$, thus $\rot(a^Nb)=0$, contradicting the fact
  that $T$ was bad. Thus, $F\cap b^{-1}(F) =\emptyset$.
  
  Now we claim that $F$ and $b^{-1}(F)$ alternate. Suppose for contradiction
  that some connected component $I = (x_1,x_2)$ of $S^1\smallsetminus F$
  contains at least two points of $b^{-1}(F)$.  Let $y_1\in b^{-1}(F)$
  be the leftmost point of $b^{-1}(F)$ in $I$, and $y_2$ be the second
  leftmost such point. Then there exists $N>0$ such that $a^N b(y_1)=x_1$.
  It follows that $a^N b(y_2)=x_2$ and $(a^Nb)^{-1}(I) = (y_1, y_2) \subset I$, so $\rot(a^N b)=0$, giving the desired contradiction.  
  
  Now that we know these sets alternate, choose $x\in b^{-1}(F)$, and let
  $y_\ell, y_r \in F$ be the left and right endpoints of the component of $S^1 \smallsetminus F$ containing $x$. 
  Then there exists a unique pair $(n,m)\in\{0,\ldots,q(a)-1\}^2$
  such that $a^n b(x)=y_r$ and
  $a^m b(x)=y_\ell$.  In particular, $(n-m)p(a)=1\mathrm{~mod~}q(a)$.
  These $m,n$ are obviously the only candidates, modulo $q(a)$,
  for the dominations
  $\widehat{a^nb}<\widehat{a^j}$ and
  $\widecheck{a^mb}>\widecheck{a^{-j}}$, for an integer $j$ such that
  $a^j(y_\ell)=y_r$.
  (This shows $m$ and $n$ do not depend on $F$).
  We claim that this pair $(n,m)$ satisfies the statement of the lemma.
  
  To see this, lift $F$ to $\widetilde{F} \subset \R$ and let
  $x_1< x_2 < \ldots < x_{q(a)}$ be consecutive points of $\widetilde{F}$. 
  Then $\widehat{a^nb}(x_i)\leq x_{i+1}$ for all $i$, hence
  $\rotild(\widehat{a^nb}^{q(a)})\leq 1$ and 
  $\rotild(\widehat{a^nb})\leq\frac{1}{q(a)}$.  
  Also, for any integer
  $j$ not divisible by $q(a)$ we have
  $\rotild(\widehat{a^nb})\leq\rotild(\widehat{a^{j}})$.
  Since $T$ is good, $\rotild(\widehat{a^{j}}^{-1}\widehat{a^nb}) \neq 0$, so we must have $\widehat{a^nb}<\widehat{a^{j}}$.
  An essentially identical argument shows that
  $\widecheck{a^{m}b}>\widecheck{a^{j}}$.
  
  It remains only to prove the statement regarding the case $p(a)=1$. 
  We know that $\widehat{a}>\widehat{a^nb}$
  and $\widehat{a}>\widehat{b^{-1}a^{1-n}}=\widecheck{a^{n-1}b}^{-1}$, and
  this immediately implies 
  $\widehat{a}=\widehat{a^nb}\cdot \widehat{b^{-1}a^{1-n}}$.
  As $(a,a^nb)$ and hence $(b^{-1}a^{1-n},a^nb)$ are also standard generating sets of $\pi_1(T)$, we must either have
  $\widehat{b^{-1}a^{1-n}}>\widehat{a^nb}$, or
  $\widehat{b^{-1}a^{1-n}}<\widehat{a^nb}$, otherwise the non-separating
  simple closed curve $a^{n-1}ba^nb$ would have rotation number zero.
  The statement follows.
\end{proof}

As a consequence, we have the following.
\begin{proposition}\label{prop:SmallRot}
  Let $(a,b)$ be a standard generating set for a bad torus. Let
  $(a_k,b_k)_{k\geq 0}$ be the sequence of standard generating sets,
  defined inductively as follows.
  \begin{itemize}
  \item Define $(a_0,b_0)=(a,b)$.
  \item If $k$ is even, let $a_{k+1}=a_k$ and $b_{k+1}=a_k^{n(k)}b_k$, where
    $0 \leq n(k) \leq q(a_k)-1$ is the integer given by Lemma~\ref{lem:SmallRot}
    applied to the generators $(a_k,b_k)$.
  \item If $k$ is odd, let $b_{k+1}=b_k$ and $a_{k+1}=b_k^{n(k)}a_k$, where
    $0 \leq n(k) \leq q(a_k)-1$ is obtained, similarly, by inputting
    $(b_k,a_k)$ into Lemma~\ref{lem:SmallRot}.
  \end{itemize}
  Then for all $k\geq 0$ even 
  we have $\widehat{a_{k+1}}>\widehat{b_{k+1}}$, and for $k \geq 0$ odd we have 
  $\widehat{a_{k+1}}<\widehat{b_{k+1}}$.
  
  Moreover, for all $k\geq 0$, we have $\widehat{a_k}>\widehat{a_{k+2}}^2$,
  and $\widehat{b_k}>\widehat{b_{k+2}}^2$. In particular, both sequences
  $(\rot(a_k))_{k\geq 0}$ and $(\rot(b_k))_{k\geq 0}$ converge to zero.
\end{proposition}
Note that the sequence $(a_k, b_k)$ is built so that both $\rot(a_k)$ and $\rot(b_k)$ converge to zero from above. This choice is arbitrary.
\begin{proof}
  The first consideration follows immediately from the first statement of
  Lemma~\ref{lem:SmallRot}. Let us prove the second.
  Let $k\geq 0$ be even.
  If $p(a_k)\geq 2$, let $n = n(k) \geq 0$ be such that $np(a_k)=1\textrm{~mod~}q(a_k)$, as in Lemma~\ref{lem:SmallRot}.
  Then $\rot(a_k^n)=\frac{1}{q(a_k)}$, and
  $\widehat{a_k^n}^{p(a_k)}=\widehat{a_k}$. By lemma~\ref{lem:SmallRot}
  we have $\widehat{b_{k+1}}<\widehat{a_k^n}$, hence
  $\widehat{b_{k+1}}^{p(a_k)}<\widehat{a_k}$, and
  $\widehat{a_{k+2}}^2<\widehat{a_k}$.
  
  Otherwise, $p(a_k)=1$, and again we take $n(k)$ as in Lemma~\ref{lem:SmallRot}. If
  $\widehat{a_k^{n(k)}b_k}^2<\widehat{a_k}$ then we may conclude as above. Otherwise,
  $\widehat{b_k^{-1}a_k^{1-n}}^2<\widehat{a_k}$, i.e.,  
  $\widehat{b_{k+1}^{-1}a_{k+1}}^2<\widehat{a_k}$.
  Thus, either
  $n(k+1)$ equals $-1$ modulo $q(b_{k+1})$,
  or not, in which case
  $\rotild(\widehat{b_{k+1}^{n(k+1)}a_{k+1}}) < 
  \rotild(\widehat{b_{k+1}^{-1}a_{k+1}})$, and then
  $\widehat{b_{k+1}^{n'}a_{k+1}}<\widehat{b_{k+1}^{-1}a_{k+1}}$.
  In either case we conclude that $\widehat{a_{k+2}}^2<\widehat{a_k}$.
  
  The argument is symmetric for $k$ odd, and for $b_k$ instead of $a_k$.
  Note that $\widehat{a_{k+2}}^2<\widehat{a_k}$ implies in particular that
  $0<\rotild(\widehat{a_{k+2}})<\frac{1}{2}\rotild(\widehat{a_k})$, hence
  the sequences $(\rotild(\widehat{a_k}))$ and $(\rotild(\widehat{b_k}))$
  converge to zero from above.
\end{proof}

Let $T = T(a,b)$ be a bad torus, and let
$(a_k,b_k)$ be the sequence furnished by Proposition~\ref{prop:SmallRot}.
Let $x\in S^1$, and let $\tilde{x}\in\R$ be a lift of $x$.
Then, by Proposition~\ref{prop:SmallRot},
the sequence $(\widehat{a_k}(\tilde{x}))_k$
is decreasing, bounded below by $\tilde{x}$, hence it converges to some
real number that we denote by $\tilde{x}+j_T(x)$. Note that $j_T(x)$ does not depend on the choice of the lift of $x$.
We define
\[ \AF_T := \{ x\in S^1,\  j_T(x)=0 \}.  \]
The reader should interpret this as the set of points that are moved
arbitrarily small distances by elements of $\{a_k\}$.
Although the notation $(a,b)$ is suppressed, $\AF_T$ as defined is dependent 
on the
generating set we started with. (But see Step 1 of the proof of Proposition \ref{prop:NoTwoBadTori} below).
As usual, we let $\widetilde{\AF_T}$ denote the preimage of $\AF_T$ in $\R$.  

\begin{proposition}[Properties of $\AF_T$] \label{prop:Jump}
  \begin{enumerate}
  \item $\AF_T$ is a non-empty, proper subset of $S^1$, with no
    isolated points, hence is infinite.
  \item  For every $x\in S^1$, we have
    $\min \{ \widetilde{\AF_T} \cap [\tilde{x}, \infty) \}= \tilde{x} + j_T(x)$. 
    In particular, $x+j_T(x)\in\AF_T$ for all $x$.
  \item The commutator $[a,b]$ fixes $\AF_T$ pointwise.
    %(in particular, it has rotation number zero).
  \end{enumerate}
\end{proposition}

\begin{proof}
  Let $x\in\R$. For all $k\geq 0$ we have $\widehat{a_k}(x)> x+j_T(x)$,
  hence, $\widehat{a_k}^2(x)> x+j_T(x)+j_T(x+j_T(x))$.
  But $\widehat{a_{k-2}}(x)>\widehat{a_k}^2(x)$, and, by definition,
  $\widehat{a_{k-2}}(x)$ converges to $x+j_T(x)$. This proves that
  $x+j_T(x)\in\AF_T$, which, thus, is non-empty.
  Further, if the open interval $(x,x+j_T(x))$ contained a point 
  $y \in \widetilde{\AF_T}$, then for large $k$ we would have 
  $x+j_T(x) > \widehat{a_k} (y) > y > x$, contradicting that $a_k$
  preserves orientation. This proves property (2).

  To prove property (3), let $x\in\widetilde{\AF_T}$ and observe,
  as above, that the sequence $\widehat{a_k}^4(x)$ also converges to $x$.
  Fix $\varepsilon>0$, and let $k$ be even, and large enough so that $x_1=x$,
  $x_2=\widehat{a_k}(x)$, $x_3=\widehat{a_k}^2(x)$ and $x_4=\widehat{a_k}^3(x)$
  all lie in the interval $[x,x+\varepsilon]$. By Lemma~\ref{lem:SmallRot},
  $a_{k+1}=a_k$ and $\widehat{b_{k+1}}$ is dominated by $\widehat{a_{k+1}}$.
  Thus, $\widehat{b_{k+1}}(x_3)\in(x_3,x_4)$, and
  $\widehat{b_{k+1}}^{-1}(x_2,x_3)\subset(x_1,x_3)$. It follows that
  $[a_{k+1},b_{k+1}]=[a,b]$ maps the point $x_2$ into the interval $(x_1,x_3)$,
  hence, for all $\varepsilon>0$, $[a,b]$ maps a point of $[x,x+\varepsilon]$
  in $[x,x+\varepsilon]$, whence $[a,b](x)=x$.
  
  It remains to prove that $\AF_T \neq S^1$, and $\AF_T$ has no isolated point.
  If $\AF_T= S^1$, then $[a,b] = \id$ and the restriction of $\rho$ to $\langle a, b \rangle$ would have abelian
  image; this contradicts the fact that $T$ is bad.
  Finally if $x$ were an isolated point of $\AF_T$, we could take $x_0\in S^1$
  such that $[x_0,x)\cap \AF_T=\emptyset$. Let $x_1$ be the next point of
  $\AF_T$ to the right of $x$. Then $x_0+j_T(x_0)=x$, and $x+j_T(x)=x_1$.
  It follows that for all $k\geq 0$, $\widehat{a_k}^2(x_0)\geq x_1$, hence
  $x_0+j_T(x_0)\geq x_1$, a contradiction.
\end{proof}

%Note that $S^1 \setminus \AF_T$ is a union of half-open
%intervals, so $\AF_T$ is not closed.
Using $j_T$, we now 
prove the following major step towards Proposition \ref{prop:BadToriIntro}.

\begin{proposition}\label{prop:NoTwoBadTori}
  There cannot exist two disjoint bad tori in $\Sigma_g$.
\end{proposition}
\begin{proof}
  By contradiction, let $T=T(a,b)$ and $T'=T(a',b')$ be two disjoint bad tori.
  Up to re-indexing and reversing some of these curves, we may suppose 
  that $(a,b,a',b')$ is the beginning of a standard basis of $\pi_1\Sigma_g$.
  
  \noindent {\em Step 1}: We have $j_T=j_{T'}$.
  
  We proceed by contradiction. Suppose for some $x_0 \in S^1$ we have $j_T(x_0)\neq j_{T'}(x_0)$, without loss
  of generality assume $j_T(x_0) < j_{T'}(x_0)$.
  Let $(a_k,b_k)_{k\geq 0}$ and $(a_k',b_k')_{k\geq 0}$
  be the sequences of generators of $T$ and $T'$ furnished by
  Proposition~\ref{prop:SmallRot}. For $k$ large enough, we have
  $\widehat{a_k}(x_0)<x_0+j_{T'}(x_0)$. Let $m$ be as in
  Lemma~\ref{lem:SmallRot} applied to $(a_k,b_k)$, and put $\alpha=a_k$,
  and $\beta=a_k^{m}b_k$. Then $(\alpha,\beta)$ is a standard generating
  set for $T$, and $\widehat{\alpha}>\widehat{\beta^{-1}}$.
  Since $\rot(b'_\ell) \to 0$, for
  $\ell\geq 0$ large enough we have
  $\rotild(\widehat{b_\ell'})<\rotild(\widehat{\beta^{-1}})$. But
  $\widehat{b_\ell'}(x_0)>x_0+j_{T'}(x_0)$ (indeed, $\widehat{b_\ell'}$
  dominates $\widehat{a_{\ell+1}'}$, by construction of the sequences
  in Proposition~\ref{prop:SmallRot}), hence $\widehat{a_k}$ does not
  dominate $\widehat{b_\ell'}$.  We now prove a sub-lemma to derive a contradiction, this will conclude the proof of Step 1.  
    
 \begin{lemma}\label{lem:OneBad}
  Let $T(a,b)$ be a bad torus, and let $b'$
  be a non separating simple curve outside $T(a,b)$ such that $b'^{-1}a$
  and $bb'$ are simple.
  Suppose that
  $\widehat{a}>\widehat{b^{-1}}$ and
  $\rotild(\widehat{b^{-1}})>\rotild(\widehat{b'})$.
  Then $\widehat{a}$ dominates $\widehat{b'}$.
\end{lemma}

\begin{proof}
  Suppose that $\widehat{a}$ does not dominate $\widehat{b'}$. Then
  $\widehat{b^{-1}}$ does not dominate $\widehat{b'}$ either.
  Observation~\ref{obs:VersRotZero} then asserts that $\rot(b'^{-1}a) = \rot(bb')=0$. Now 
  $i(b'^{-1}a, bb') = \pm 1$, and $b'^{-1}a$ lies in a one-parameter family,
  so as in Observation~\ref{obs:DeformationsConj},
  there is a path-deformation of $\rho$ replacing the action 
  of $bb'$ with $b'^{-1}a \cdot bb'$. Hence,
  $\rot(bb')=0=\rot(b'^{-1}a \cdot bb')=\rot(ab)$.
  This contradicts that $T(a,b)$ is bad.
  \end{proof}

  \noindent {\em Step 2}: We can deform the representation so that
  $j_T\neq j_{T'}$.
  
  As shown in the proof of Proposition~\ref{prop:Jump}, $[a,b] \neq \id$,
  but $\AF_T \subset \Fix([a,b])$.  
  Let $x \in S^1 \smallsetminus \Fix([a,b])$, so then $j_T(x)>0$. Let $y = x+j_T(x)$,
  let $I$ be the connected component of $S^1\smallsetminus \Fix([a,b])$
  containing $x$, and let $c_t$ be a one-parameter family of
  homeomorphisms commuting with $[a,b]$, and with support equal to
  $\overline{I}$.
  
  Then the distance between $c_t(x)$ and
  $c_t(y)$ varies, in a nonconstant way, with $t$: it goes to zero as $t \to \infty$
  if $y\in I$, and simply changes if $y\not\in I$.
  Now, consider a bending deformation of $\rho$ defined by 
  $\rho_t(\gamma)=\rho(\gamma)$ for all curves outside $T$, and
  $\rho_t(\gamma)=c_t\rho(\gamma)c_{-t}$ for $\gamma\in\langle a,b\rangle$.
  This deformation changes the value of $j_T(x)$, without changing the
  value of $j_{T'}(x)$.   In particular, after this path-deformation, Step~1 no longer holds! 
  This gives a contradiction.    
\end{proof}

Supposing again that $T(a,b)$ is a bad torus, it remains to show that any
torus in $\Sigma_g\smallsetminus T(a,b)$ is not only good, but
\emph{very good}.  The next lemma will allow us to easily achieve this goal. 

\begin{lemma}\label{lem:FixBad}
  Let $T=T(a,b)$ be a bad torus, and let $\gamma$ be a non-separating
  simple closed curve outside of $T$, with $\rot(\gamma)=0$.
  Then $\AF_T\subset\Fix(\gamma)$.
\end{lemma}

\begin{proof}
  Let $(a_k,b_k)_{k\geq 0}$ be the sequence given by
  Proposition~\ref{prop:SmallRot}, and orient $\gamma$ so that
  $\gamma^{-1}a_k$ is  also a (non-separating) simple curve.
  Fix $k\geq 0$, and let $\alpha=a_k$, and $\beta=a_k^{m}b_k$, as in
  Lemma~\ref{lem:SmallRot}. Then, by
  Lemma~\ref{lem:OneBad},
  we have $\widehat{a_k}>\widehat{\gamma}$.
  This holds for all $k\geq 0$, hence, for all $x\in\R$ we have
  $\widehat{\gamma}(x)\leq x+j_T(x)$. In particular,
  if $x \in \widetilde{\AF}_T$, we have $\widehat{\gamma}(x) \leq x$.
  
  For the reverse inequality, first observe that 
  the conditions $\widecheck{a}<\widecheck{b^{-1}}$ and
  $\rotild(\widecheck{b^{-1}})<\rotild(\widecheck{\gamma})$ imply the
  domination
  $\widecheck{a}<\widecheck{\gamma}$ (this is exactly the statement of
  Lemma~\ref{lem:OneBad} after reversing the orientation of $\R$), and
   $\widecheck{\gamma}=\widehat{\gamma}$ since $\rot(\gamma)=0$.
   Let $x\in\widetilde{\AF_T}$, and fix $\varepsilon>0$.
  For $k$ large enough, the sequence $(a_k,b_k)$ from
  Proposition~\ref{prop:SmallRot} satisfies
  $\widehat{a_k}(x)<x+\varepsilon$.
  Let $(a',b')=(a_k,b_k)$ for such a large $k$, and 
  define $b''=b'$ and $a''={b'}^ma'$ and then $\alpha=a''$ and
  $\beta={a''}^nb''$, where $m$, and then $n$, are given by
  Lemma~\ref{lem:SmallRot} with these two successive pairs.
  Then, we have
  $\rotild(\widecheck{\alpha})<\rotild(\widecheck{\beta^{-1}})<
  \rotild(\widecheck{\gamma})$, hence, $\widecheck{\alpha}<\widecheck{\gamma}$,
  ie, $\widehat{\alpha^{-1}}$ dominates $\widehat{\gamma}^{-1}$.
  It follows that
  $\widehat{\gamma}(x)\leq x+\varepsilon$.  This shows
  $\widecheck{\gamma}(x)\geq x$, as desired.  
  \end{proof}

\begin{proof}[End of the proof of Proposition \ref{prop:BadToriIntro}]
  Suppose that $T = T(a,b)$ is a bad torus, and let $T'$ be a torus disjoint
  from $T$. By Lemma \ref{lem:OneBad}, $T'$ is good and we may take
  $T' = T(a',b')$ where $\rot(a') = 0$.
  Then we have $\Fix(a')\supset \AF_T$ by Lemma \ref{lem:FixBad}.
  This is also true after replacing $a'$ with a deformation $b'_t a'$, so
  $\Per(b') \supset \AF_T$, or equivalently, $\Fix(b'^{q(b')}) \supset \AF_T$.
  Since this is also true after replacing $b'$ with any deformation $a'_t b'$,
  we conclude $\AF_T \subset P(a',b')$. By
  Lemma~\ref{lem:ProprietesPNU} (1), this means that
  $\langle a', b' \rangle$ has a finite orbit in $S^1$.
\end{proof}

%%%%%%%%%%%%%%%%%%%%%%%%%%%%%%%%%%%%%%%%%%%%%%%%%%%%%%%%%%%%%%%%%%%%%%
%%%%%%%%%%%%%%%%%%%%%%%%%%%%%%%%%%%%%%%%%%%%%%%%%%%%%%%%%%%%%%%%%%%%%%
%%%%%%%%%%%%%%%%%%%%%%%%%%%%%%%%%%%%%%%%%%%%%%%%%%%%%%%%%%%%%%%%%%%%%%

\subsection{Good tori}\label{ssec:P}

In this section, we prove Proposition~\ref{prop:NoTwoJustGoodIntro}: if $\rho$ is path-rigid and
non-geometric, then there cannot exist two disjoint good tori which are
both not very good.
In the course of the proof, we will develop some tools that will be used
again in Section~\ref{sec:Fin} for the proof of Theorem~\ref{theo:RigidGeom}.
%The proof proceeds by showing that any path-rigid, minimal representation that
%fails the hypothesis above on tori necessarily satisfies hypothesis $S_k$. 
%Our main tool is the movement of periodic sets by bending deformations,
%as introduced in Paragraph~\ref{sssec:Twist}.

To motivate the first step, observe that if $\rho$ has two disjoint good tori
$T(a,b)$ and $T(d,e)$ with $\rot(a) = \rot(e) = 0$, and if neither of these tori are very good, then $P(a,b) = P(e,d) = \emptyset$. We can also find $c$ so that $(a,b,c,d,e)$ is a $5$-chain.   This is the set-up of the next Proposition.

%%%%%%%%%%%%%%%%%%

\begin{proposition}\label{prop:DeuxAlorsQuatre}
  Let $\rho$ be path-rigid minimal and let $(a,b,c,d,e)$ be a $5$-chain.
  Suppose that both $P(a,b)$ and $P(e,d)$ are empty.
  Then we have $S_k(b,c)$, for some $k\geq 1$.
\end{proposition}

\begin{proof}[Proof of Proposition~\ref{prop:DeuxAlorsQuatre}]
  After changing orientations of these curves, we may suppose that
  $(a,b,c,d,e)$ is a directed $5$-chain.
  By Theorem~\ref{theo:PerDisjointsSk}, it
  suffices to show that $\Per(b) \cap \Per(c) = \emptyset$. Since
  $P(a,b) = \emptyset$, Lemma~\ref{lem:PasDAcc} says that $\partial N(a,b)$
  is finite. Choose a positive one-parameter family
  $(e_t)_{t\in\R}$, commuting with $\rho(e)$.
  Since $P(e,d) = \emptyset$, we have $\Per(e_td) \subset U(e,d)$ for all $t$,
  so the sets $\Per(e_td)$, for varying $t$, are pairwise disjoint and
  we can choose $t_0$ so that
  $\Per(e_{t_0}d) \cap \partial N(a,b) = \emptyset$.
  Abusing notation, we now replace $d$ with $e_{t_0}d$
  (we will not further use $e$). With this change in notation, we now have
  $\partial N(a,b)\cap P(d,c)=\emptyset$.
  The remaining step will be a useful tool later in Section~\ref{sec:Fin}, so
  we split it off to a separate statement (Lemma~\ref{lem:DeuxAlorsQuatre}),
  proved below.
\end{proof}

\begin{lemma} \label{lem:DeuxAlorsQuatre}
  Let $\rho$ be path-rigid, and let $(a,b,c,d)$ be a $4$-chain.
  Suppose that $P(a,b)=\emptyset$ and $\partial N(a,b)\cap P(d,c)=\emptyset$.
  Then $\Per(b) \cap \Per(c) = \emptyset$.
\end{lemma}

\begin{proof} 
  Orient the curves so that $(a,b^{-1},c,d)$ is a directed $4$-chain.
  Let $a_t$ and $d_t$ be positive one-parameter families commuting with $a$
  and $d$ respectively. By Lemma~\ref{lem:PerCommuns}, it suffices to find
  $t$ and $s$ such that $\Per(a_tb) \cap \Per(d_sc) = \emptyset$.
  Let $F_0=\partial N(a,b)\cap \partial N(d,c)$.  Since $P(a,b) = \emptyset$, Lemma~\ref{lem:PasDAcc} says $\partial N(a,b)$ is finite.  Hence, $F_0$ is
  finite. Let $F_1=\partial N(a,b)\smallsetminus F_0$ and
  $F_2=(P(d,c)\cup \partial N(d,c))\smallsetminus F_0$.
  By construction, the $F_i$ are disjoint closed sets; let
  $\varepsilon>0$ be smaller than the minimum distance between any two of them.
  Fix $t$ large, so that (by Lemma~\ref{lem:ProprietesPNU}), $\Per(a_tb)$ is
  contained in the $\varepsilon$-neighborhood of $F_0\cup F_1$, hence
  disjoint from $F_2$. Since $F_0 \subset N(a,b)$, it is also disjoint
  from $\Per(a_tb)$, i.e.
  $\Per(a_tb) \cap (F_0\cup F_2) = \emptyset$. Now let $\eta>0$ be
  smaller than the distance between $F_0\cup F_2$ and $\Per(a_tb)$.
  By Lemma~\ref{lem:ProprietesPNU} again, for $s$ large enough, the set
  $\Per(d_sc)$ is in the $\eta$-neighborhood
  of $F_0\cup F_2$. Hence, $\Per(a_tb)$ and $\Per(d_sc)$ are disjoint,
  as desired.
\end{proof}

Our next goal is to propagate $S_k(\cdot, \cdot)$ to other curves.  For this, we define two stronger properties.  

\begin{definition}[Strengthenings of $S_k$]
  Say that two curves $a$ and $b$ satisfy $S_k^+(a,b)$ if they satisfy
  $S_k(a,b)$ and if additionally 
  $a(\Per(b))\cap\Per(b)=\emptyset$.  Say that $a$ and $b$ satisfy $S_k^{++}(a,b)$ 
  if they satisfy both $S_k^+(a,b)$ and $S_k^+(b,a)$.
\end{definition}

Property $S_k^+(\cdot, \cdot)$ allows one to move families of periodic
points continuously by twist deformations, as described in the following lemma.

\begin{lemma}%[Periodic points move continuously] 
\label{lemma:PerContPath}
  Let $a$ and $b$ be any curves with $i(a,b)=-1$ satisfying $S_k^+(a,b)$.
  Then there exists a continuous family $a_t$ commuting with $a$ such that
  $\Per(a_t b) \cap \Per(a_s b) = \emptyset$ for all $s \neq t$, and
  $|\Per(a_t b)| = 2k$ for all $t$.
\end{lemma}

Since property $S_k(a,b)$ immediately implies that $\Per(b) \subset U(a,b)$, the nontrivial part of this lemma is controlling the cardinality of 
$\Per(a_t b)$.   This requires a special construction of one-parameter family $a_t$, which is, for once, not a one-parameter group.

\begin{proof}[Proof of Lemma \ref{lemma:PerContPath}]
  Together with Lemma~\ref{lem:Preserv2}, the asumption
  $a\Per(b)\cap\Per(b)=\emptyset$ completely prescribes the cyclic
  order on the set $\bigcup_n a^n(\Per(b))$; it follows that we may choose
  a neighborhood $V$ of $\Per(b)$, consisting of $2k$ open intervals,
  such that $a^n(V)\cap a^m(V)=\emptyset$ for all $n,m\in\Z$.
  We now construct a continuous family of homeomorphisms $a_t$ commuting
  with $a$, supported on $\bigcup_{n \in \Z} a^nV$.  
%  A slight variation on this construction would give a {\em positive} family
%  of homeomorphisms, but this is not required by the lemma. 

  Choose one point in each of the periodic orbits of $b$; let
  $x_1, x_2, \ldots, x_m$ denote these points.
  Parametrize $S^1$ so that, for each $x_i$,
  $b$ agrees with a rigid rotation by $p(b)/q(b)$ on a small neighborhood of
  $b^k(x_i)$ for $k = 0, 1, \ldots, q(b)-2$ and so that $b$ maps a neighborhood
  of $b^{q(b)-1}(x_i)$ to a neighborhood of $x_i = b^{q(b)}(x_i)$ by
  the map $x\mapsto 2x$ or $x\mapsto x/2$, in coordinates, depending on
  whether the orbit of $x_i$ is repelling or attracting.
  
  Let $V_{i,k}$ denote the connected component of $V$ containing $b^k(x_i)$.
  Define $a_t$ to be the identity on $V_{i,k}$ for $k = 0, 1, \ldots, q(b)-2$
  and all $i$. 
  On $V_{i,q(b)-1}$, 
  using the local coordinates in which $b$ is linear, define
  $a_t$ to agree in a neighborhood of 0 with the translation $x \mapsto x+t$, and extend $a_t$ 
  equivariantly (with respect to $a$) over $S^1$.  This all can be done continuously in $t$.   
  After shrinking the $V_{i,k}$ if needed, by construction, each $(a_t b)^{q(b)}$ has a unique fixed point in each $V_{i,k}$, and these vary continuously.   Additionally, for $t$ sufficiently small, no new fixed points will be introduced; this proves the lemma.  \end{proof}
%  Since $U_i$ has closure contained in $V_{i,q-1}$, there exists $\varepsilon>0$
%  so that for all $t\leq\varepsilon$, this map extends to a homeomorphism of
%  $V_{i, q(b)-1}$ fixing the endpoints, and we may take these extensions
%  to vary continuously in $t$.
%  From now on, we restrict to such $t \leq \varepsilon$.
%  Finally, we extend $a_t$ to a family of homeomorphisms of $S^1$ that is
%  equivariant with respect to $a$ on $\bigcup_{n \in \Z} a^nU$, where
%  $U=\bigcup_i U_i$, and agrees with the identity elsewhere.
  
%  The property that $a_t$ agrees with $x \mapsto x+t$ on $U_i$ implies that
%  $(a_t b)^{q(b)}$ has a unique fixed point in each set $b^{-n}U_i$ for
 % $n=0, 1, 2, \ldots, q(b)-1$. Since the union of such sets open $b^{-n}U_i$
%  covers $\Per(b)$, there is some $\delta > 0$ such that
%  $|b^{q(b)}(x) - x| \geq \delta$ for all $x \in S^1 \smallsetminus U$, hence
%  for all $t$ sufficiently small, we will have $\Per(a_t b) \subset U$.
%\end{proof}

The next lemma and proposition allow one to propagate $S_k^{++}$ along chains. 

\begin{lemma}\label{lem:SkPlusSk}
  Let $(a, b, c)$ be a completable $3$-chain.
  Then $S_k^+(a,b)$ implies $S_k(b,c)$.
\end{lemma}

\begin{proposition}\label{prop:SkPPTransitif}
  Let $(a,b,c)$ be a completable $3$-chain. Suppose that $S_k^{++}(a,b)$
  holds. Then $S_k^{++}(b,c)$ holds as well.
\end{proposition}

To prove these two statements, we will need a quick sub-lemma.

\begin{lemma}[Per has empty interior] \label{lemma:int_per_vide}
  Let $a$ and $b$ be any curves with $i(a,b)=\pm1$, and let
  $b_t$ be a positive one-parameter family commuting with $b$.
  Then, for all but countably many $t$, the set $\Per(b_t a)$ has empty
  interior.
\end{lemma}

\begin{proof}
  Let $X = S^1 \smallsetminus P(b, a)$.
  Then for $t \neq s$, we have $\Per(b_t a)\cap\Per(b_s a)\cap X=\emptyset$.
  In particular, the set
  \[ T= \lbrace t:\, \Per(b_t a) \cap X
  \text{ contains a nonempty open set} \rbrace \]
  is countable.
  Also if $U \subset \Per(b_t a)$ is nonempty and open, then
  $U \cap X = U \smallsetminus P(b,a)$ is nonempty and open since $P(b,a)$
  is closed with empty interior, hence $t\in T$.
  It follows that for all $t\not\in T$, $\Per(b_t a)$ has empty interior.
\end{proof}

\begin{proof}[Proof of Lemma \ref{lem:SkPlusSk}]
  Complete $(a, b, c)$ to a $4$-chain $(a,b,c,d)$, and let $(d_t)_{t\in\R}$
  be a positive one-parameter family commuting with $d$.
  By Lemma~\ref{lemma:int_per_vide}, $\Per(d_{t_0}c)$ has empty interior for
  some $t_0\in\R$.
  Now, by Lemma~\ref{lemma:PerContPath}, there exists a one-parameter group
  $(a_s)_{s\in\R}$, an interval $I\subset\R$ and $2k$ maps,
  $\phi_j\colon I\to S^1$,
  each a homeomorphism to its image, such that the $2k$ periodic points of
  $\Per(a_sb)$ are precisely $\phi_1(s),\ldots,\phi_{2k}(s)$, for all
  $s\in I$. The set $\bigcap\phi_j^{-1}(\Per(d_{t_0}c))$ then has empty interior
  in $I$, hence there exists $s_0\in I$ such that
  $\Per(a_{s_0}b)\cap\Per(d_{t_0}c)=\emptyset$,
  and $\Per(b)\cap\Per(c)=\emptyset$ by
  Lemma~\ref{lem:PerCommuns}.
  We conclude by using Theorem~\ref{theo:PerDisjointsSk}.
\end{proof}

\begin{proof}[Proof of Proposition \ref{prop:SkPPTransitif}]
  Complete the $3$-chain into a $5$-chain, $(e,a,b,c,d)$, and apply
  Lemma~\ref{lem:SkPlusSk} to the $3$-chains $(a, b, c)$ and $(e, a, b)$
  to conclude $S_k(b,c)$ and $S_k(a,e)$.
  By Lemma~\ref{lem:PerLabiles}, we may then use a bending deformation of
  $a$ along $e$ to move the periodic set
  of $a$ disjoint from any finite set, so in particular
  $\Per(a) \cap \Per(c) = \emptyset$.
  Now take a positive one-parameter family $a_t$ commuting with $a$.
  Since $\Per(a) \cap \Per(c) = \emptyset$ the points $a_{-t}\Per(c)$ move
  continuously in $t$, so there is some $t$ such that 
  $b\Per(c)\cap a_{-t}\Per(c)=\emptyset$.
  Thus, $a_tb\Per(c)\cap\Per(c) = \emptyset$ hence by
  Lemma~\ref{lem:PerCommuns} $b\Per(c)\cap\Per(c) = \emptyset$.
  Thus, we conclude that $S_k^{+}(b,c)$ holds.
  By Lemma~\ref{lem:SkPlusSk}, this implies that $S_k(c,d)$ holds as well.
  In particular, $\Per(d)$ is finite. We can now apply
  Lemma~\ref{lem:PerLabiles} and use a bending deformation so that
  $\Per(a_t b)\cap\Per(d)=\emptyset$, which implies that
  $\Per(b)\cap\Per(d)=\emptyset$, and repeat the argument above
  (with $d$ and $c$ playing the roles of $a$ and $b$) to conclude
  $S_k^+(c,b)$ holds as well.
\end{proof}

Proposition~ \ref{prop:SkPPTransitif}, Theorem~\ref{thm:SkImpliqueGeom},
and the connectedness of the graph in Lemma~\ref{lem:BallonsConnexe2}
immediately gives the following.

\begin{corollary}\label{cor:TwoCurvesSuffice}
  Let $\rho$ be a path-rigid, minimal representation, and suppose there
  exists $(a,b)$ such that $S_k^{++}(a,b)$ holds. Then $\rho$ is geometric.
\end{corollary}

This consequence is strong enough to imply the main result of the companion
article \cite{AuMoinsG}. We explain this now, as it will be used again in
Section~\ref{sec:Fin}.
\begin{corollary}\label{cor:AuMoinsG}
Let $\rho$ be a path-rigid, minimal representation, and suppose that there is some torus $T(a,b)$ such that the relative Euler number of $T(a,b)$ is $\pm 1$.  Then $\rho$ is semi-conjugate to a Fuchsian representation.
\end{corollary}

\begin{proof}
Since $T(a,b)$ has Euler number 1, it follows form \cite{Matsumoto16} that the restriction of $\rho$ to $\langle a, b \rangle$ is semi-conjugate to a geometric representation in $\PSL$.  (This is not difficult: that $\rotild([\widehat{\rho(a)}, \widehat{\rho(b)}] = \pm1$ easily implies that $\rho(a)$ and $\rho(b)$ are 1-Schottky, hence are semi-conjugate to a geometric representation in $\PSL$. See the beginning of $\S$3 in \cite{Matsumoto16}.)    In particular, property $S_1^{++}(a,b)$ holds, and Corollary \ref{cor:TwoCurvesSuffice} implies that $\rho$ is geometric.
\end{proof}

Given Corollary \ref{cor:TwoCurvesSuffice}, the main goal of this section reduces to the following.
\begin{proposition}\label{prop:PasDeuxBons}
  Let $(a,b,c,d,e)$ be a $5$-chain, and 
  suppose that $P(a,b) = P(e,d) = \emptyset$. Then we have $S_k^{++}(b,c)$.
\end{proposition}

\begin{proof}
  Suppose $P(a,b)=P(e,d)=\emptyset$.  By Proposition~\ref{prop:DeuxAlorsQuatre},
  we have $S_k(b,c)$ and $S_k(c,d)$ for some $k\geq 1$.
  Since $P(e,d)=\emptyset$ and $\Per(b)$ is finite, we have a bending
  deformation $e_td$ such that $\Per(b)\cap\Per(e_td)=\emptyset$, hence
  $\Per(b)\cap\Per(d)=\emptyset$.
  Hence, $\Per(b)\cap d_tc\Per(b)=\emptyset$ for some $t$, so we have
  $\Per(b)\cap c\Per(b)=\emptyset$, ie,
  $S_k^+(c,b)$ holds.
  By Lemma~\ref{lem:SkPlusSk}, this gives $S_k(a,b)$.
  In particular, $\Per(a)$ is finite, and so there exists a bending
  deformation replacing $c$ with $d_t c$ such that
  $\Per(a)\cap\Per(d_t c)=\emptyset$, and hence
  $\Per(a)\cap\Per(c)=\emptyset$.
  Repeating the argument above, we conclude $S_k^+(b,c)$ holds.
\end{proof}

The main result of this section is now a quick corollary.  We restate it here for convenience and to summarize our work.  

\begin{corollary}\label{cor:PasDeuxBons}
  Let $\rho$ be a path-rigid, minimal representation. Suppose $\rho$ admits
  two disjoint good tori that are not very good. Then $\rho$ is geometric.
\end{corollary}

\begin{proof}
  Let $T(a,b)$ and $T(d,e)$ be two disjoint good tori. Since they
  are good, we may suppose $\rot(a)=\rot(e)=0$. Since they are not very
  good, we have $P(a,b)=\emptyset$ and $P(e,d)=\emptyset$. We may
  find a curve $c$ such that $(a,b,c,d,e)$ is a $5$-chain, and then
  Proposition~\ref{prop:PasDeuxBons} and Corollary~\ref{cor:TwoCurvesSuffice}
  imply that $\rho$ is geometric.
\end{proof}

%%%%%%%%%%%%%%%%%%%%%%%%%%%%%%%%%%%%%%%%%%%%%%%%%%%%%%%%%%%%%%%%%%%%%%
%%%%%%%%%%%%%%%%%%%%%%%%%%%%%%%%%%%%%%%%%%%%%%%%%%%%%%%%%%%%%%%%%%%%%%
%%%%%%%%%%%%%%%%%%%%%%%%%%%%%%%%%%%%%%%%%%%%%%%%%%%%%%%%%%%%%%%%%%%%%%
%%%%%%%%%%%%%%%%%%%%%%%%%%%%%%%%%%%%%%%%%%%%%%%%%%%%%%%%%%%%%%%%%%%%%%

\subsection{Finite orbits}\label{ssec:OrbitesFinies}

The goal of this section is the proof of the following proposition.
\begin{proposition}\label{prop:OrbiteFinie}
  Let $\rho\colon\Gamma_g\to\HOS$ be a path-rigid representation, and let
  $\Sigma=\Sigma_{g-1,1}$ be a subsurface containing only very good tori.
  Then $\rho_{|\pi_1\Sigma}$ has a finite orbit.
\end{proposition}

If $T(a,b)$ is very good, then $a$ and $b$ act with a finite orbit, so $\rot(ab) = \rot(a) + \rot(b)$.  Thus, in a subsurface where all tori are very good, rotation number is additive on any pair of curves with intersection number $\pm 1$.  This motivates the following proposition, which gives our first step.  

\begin{proposition}\label{prop:SiBonTresBon}
  Let $\Sigma$ be a one-holed surface of genus $\geq 2$.
  Suppose that $\pi_1\Sigma$ acts on the circle in such a way that
  all nonseparating simple curves have rational rotation number, and
  that for all $\gamma_1$, $\gamma_2$ with $i(\gamma_1,\gamma_2)=\pm 1$,
  we have $\rot(\gamma_1\gamma_2)=\rot(\gamma_1)+\rot(\gamma_2)$.
  
  Then, there exist two curves $\gamma_1$, $\gamma_2$ with
  $i(\gamma_1,\gamma_2)=\pm 1$ and $\rot(\gamma_1)=\rot(\gamma_2)=0$.
\end{proposition}

\begin{proof}
  Let
  $(a_1,\ldots,b_g)$ be a standard generating set of $\pi_1\Sigma$,
 % ie, where the loop $[a_1,b_1]\cdots[a_g,b_g]$ is homotopic to the boundary,
  and consider the non-completable directed $5$-chain
  $(\gamma_1,\gamma_2,\gamma_3,\gamma_4,\gamma_5)=
  (a_1^{-1}b_1a_1,a_1,\delta_1,a_2,b_2^{-1})$, with the notation of
  Section~\ref{ssec:Marquage}.  
%  Our proof proceeds by iteratively applying
%  Dehn twists in this system of curves.
  
%  Recall from Convention~\ref{conv:DeformChaines} that the effect of a
%  Dehn twist along $\gamma_i$ replaces $\gamma_{i-1}$ (if $i\geq 2$)
%  with $\gamma_i^{-1}\gamma_{i-1}$, and replaces $\gamma_{i+1}$ with
%  $\gamma_{i+1}\gamma_i$ (if $i\leq 4$), while leaving the other curves
%  unchanged. 
  Let $r_i = \rot(\gamma_i)$ and let $\tau_i$ denote the map on rotation numbers induced by the Dehn twist along 
  $\gamma_i$.  Then $\tau_i(r_1,r_2,r_3,r_4,r_5) = (r_1',\ldots,r_5')$ 
  where $r_{i-1}'=r_{i-1}-r_i$ and $r_{i+1}'=r_{i+1}+r_i$, and $r_j'=r_j$.  
  As Dehn twists preserve chains, the proof of the proposition is reduced to showing that 
  the operations $\tau_i$ can be iterated to transform any vector in $(\Q/\Z)^5$ to a vector of the form $(0,0,r_3,r_4,r_5)$.
    This is a straightforward exercise (and should be familiar to anyone familiar with the symplectic group $\mathrm{Sp}(2g,\Z)$); we leave the details to the reader.   \end{proof}

Proposition \ref{prop:SiBonTresBon} is useful because it is much easier
to control the dynamics of two curves if their rotation numbers are zero, as in the next Proposition.
%In this case, we do not need a condition as strong as
%$\Per(a)\cap\Per(b)=\emptyset$ in order to control the fixed points as in
%Lemma~\ref{lemma:PerContPath}. More precisely, we have the following statement.
\begin{proposition}\label{prop:Promenade}
  Suppose $\rot(a)=\rot(b)=0$. Then for every $\varepsilon>0$, there exists a one-parameter family
  $(a_t)_{t \in \R}$ commuting with $a$, an interval $J \subset \R$, and a finite collection of homeomorphisms 
  $\phi_i: J \to S^1$ with disjoint images, such that 
   for all $t\in J$, 
  \[ \Fix(a_tb)\cap\left(S^1\smallsetminus V_\varepsilon(P(a,b))\right)
    =\lbrace \phi_1(t),\cdots,\phi_n(t) \rbrace. \]  
\end{proposition}
In other words, for all $t \in J$, the fixed points of $a_tb$ at distance $\geq\varepsilon$
to $P(a,b)$ are finite in number and move continuously in $t$.   Compare with Lemma~\ref{lemma:PerContPath}.
Note that we do not require $a_t$ to be a \emph{positive} family. 

\begin{proof}
Fix a positive one-parameter family $\alpha_t$ commuting with $a$.  We will modify $\alpha_t$ to obtain the desired family $a_t$.   

  When $\rot(a)=\rot(b)=0$, we have $P(a,b)=\Fix(b)\cap\partial\Fix(a)$,
  and the set $U(a,b)$ has a very simple description: 
  $x\in U(a,b)$ if and only if $x$ and $b(x)$ are in the same
  connected component of $S^1\smallsetminus\partial\Fix(a)$.
  Thus,
$U(a,b) = \bigcup_I (I\cap b^{-1}(I))$,
  where $I$ ranges over the connected components of
  $S^1\smallsetminus\partial\Fix(a)$.  As each connected component $I$ is $a$-invariant, we may define $a_t$ separately on each, affecting only $\Fix(a_tb)\cap I$. 
  
  For every connected component $I$ of $S^1\smallsetminus\partial\Fix(a)$,
  let $U(I)$ denote $I\cap b^{-1}(I)$.  By definition, 
  each endpoint of $U(I)$ lies in $\partial N(a,b)\cup P(a,b)$.
  Thus, by Lemma~\ref{lem:ProprietesPNU}, all but finitely many intervals
  $U(I)$ lie in $V_\varepsilon(P(a,b))$. On all the corresponding connected
  components $I$ of $S^1\smallsetminus\partial\Fix(a)$ we set $a_t = \alpha_t$.  
  
 Now we treat the remaining (finitely many) intervals $I$ of $S^1 \smallsetminus \Fix(a)$ such that $U(I)$ is nonempty, considering the configuration of $I$ and $b^{-1}(I)$.   
 As a first case, suppose that $I$ and $b^{-1}(I)$ share an endpoint, i.e. a point in $P(a,b)$.  If this is the right endpoint, define $a_t = \alpha_t$ on $I$.  If the left endpoint is shared, take instead $a_t= \alpha_{-t}$.   If $I = b(I)$, either choice will work.  In each case, for all $s$ sufficiently large, we have
 \begin{equation}  \label{eq1}
 \Fix(a_sb) \cap I \subset V_\varepsilon(P(a,b)).
 \end{equation} 
  
  As a second case, suppose $b$ shifts $I$.  If the shift is to the right, i.e. $I=(x_1,x_3)$ and
  $b(I)=(x_2,x_4)$ with $x_1,x_2,x_3,x_4$ in cyclic order, define $a_t =\alpha_t$ on $I$, and if 
  the shift is to the left, set $a_t =\alpha_{-t}$.   In either case, for all $s$ sufficiently large, we have 
   \begin{equation}  \label{eq2}
    \Fix(a_sb)\cap I=\emptyset.
   \end{equation}

  We are left with the case where either $b\bar{(I)} \subset I$ or $\bar{I} \subset b(I)$.
  Suppose the first holds, as the second can be dealt with by a symmetric argument.   
  Note that (using $\alpha_t$ and $b$) we are in the case
   $n=1$ of Lemma~\ref{lem:EHN}
  of the preceding section. Thus, there exists $s\in\R$ such that
  $\alpha_s b$ has a unique fixed point in $I$. Moreover, $b\bar{(I)} \subset I$ implies that this unique
  fixed point is an attracting point, i.e. we may take local coordinates so that the map $\alpha_sb$
  agrees with $x\mapsto x/2$ at the origin.   After reparametrization of $\alpha_t$ on $I$, we may assume that this time $s$ is sufficiently large to satisfy \eqref{eq1} and \eqref{eq2} above.  
  Working in coordinates, let $(-\delta, \delta)$ be a neighborhood of $0$ contained in a fundamental domain for $a$.  Let $\tau_t$ be a smooth family of bump functions supported on $(-\delta, \delta)$ and agreeing with $x \mapsto x+t$ on an even smaller (fixed) neighborhood of $0$, for all $t < \delta' < \delta$.  Extend $\tau_t$ $a$-equivariantly to a homeomorphism of $I$ 
   Now define $a_t$ on $I$ to agree with $\alpha_t$ for $t<s$, to agree with $\tau_{t-s}\alpha_s$ for $s\leq t \leq s+\delta'$, and arbitrarily (for example, constant in $t$) for $t \geq s+\delta'$.  
 Varying $t$ in $J:= (s, s+\delta')$, the homeomorphism $a_tb$ has a unique fixed point in $I$ that moves continuously with $t$, as desired.  
  Of course, we can choose parameterizations of $a_t$ on each of these (finitely many) intervals so that $J$ does not
  depend on $I$.  This proves the lemma.   
  \end{proof}

Using this tool, we can propagate finite orbits over chains.

\begin{proposition}  \label{prop:MagicProposition}
Let $a, \gamma_1, \gamma_2, \gamma_3, ... \gamma_k$ be a chain.    Suppose that $\Per(a)$ has empty interior, $\rot(\gamma_i) = 0$ for all $i$,  the subgroup $\langle a, \gamma_1 \rangle$ has a finite orbit and $\langle \gamma_i,  \gamma_{i+1} \rangle $ has a global fixed point.   Then 
$\langle a, \gamma_i, ..., \gamma_{k} \rangle$ has a finite orbit.  
\end{proposition} 

\begin{proof}
Inductively, suppose the statement holds for chains of length $k$ and take a chain of length $k+1$ of the form $a, \gamma_1, ..., \gamma_{k}$.    By inductive hypothesis the group generated by the first $k$ elements $\langle a, \gamma_1, ..., \gamma_{k-1} \rangle$ has a finite orbit, i.e. there is a periodic orbit of $a$ contained in $\bigcap_{i=1}^{k-1} \Fix(\gamma_i)$.  

Since $\Per(a)$ has empty interior, for any $n \in \N$,
we can use Proposition~\ref{prop:Promenade} to produce a homeomorphism
$c(n)$ lying in a one-parameter family commuting with $\gamma_k$ such that
$\Fix(c(n)\gamma_{k-1})\cap\Per(a)\subset V_{1/n}(P(\gamma_{k-1},\gamma_k))$.
Indeed, with the notation of that proposition, there exists $t\in J$
such that $\phi_j(t)\not\in\Per(a)$ for all $j$, because
$\bigcap_j\phi_j^{-1}(\Per(a))$ has empty interior in $J$.
Do this for each $n \in \N$; we do not require that the $c(n)$ all belong to a common one-parameter family, all that is important is that they are each obtainable by a bending deformation, hence give a semi-conjugate representation.

The result is a sequence of bending deformations $c(n) \gamma_{k-1}$ of $\gamma_{k-1}$ such
that
$$\Fix(c(n) \gamma_{k-1}) \cap \Per(a) \subset V_{1/n}(\Fix(\gamma_{k-1}) \cap \Fix(\gamma_k)).$$

Since $\langle a, \gamma_1, ..., \gamma_{k-1} \rangle$ has a finite orbit,
and this property is stable under semi-conjugacy, it follows that, for every
$n$, $\bigcap_{i=1}^{k-2}\Fix(\gamma_i) \cap\Fix(c(n) \gamma_{k-1})$ contains
a full orbit of $a$. For each $n$, choose one such full orbit, and denote it by $\mathcal{O}_n$.   After passing to a subsequence, the sets $\mathcal{O}_n$ converge pointwise to a finite subset of $\bigcap_{i=1}^{k-2}\Fix(\gamma_i) \cap \Per(a)$ that is invariant under $a$ (as these are both closed conditions) so the limit is a full orbit.  Moreover, this orbit is contained in every open neighborhood of $\Fix(\gamma_{k-1}) \cap \Fix(\gamma_k)$, so also lies in $\Fix(\gamma_{k-1}) \cap \Fix(\gamma_k)$.  
This gives a periodic orbit of $a$ in $\bigcap_{i=1}^{k}\Fix(\gamma_i)$, as desired.
\end{proof} 

We now prove the main result advertised at the beginning of this section.

\begin{proof}[Proof of proposition \ref{prop:OrbiteFinie}]
  Let $\Sigma_{g, 1}$ be a surface with one boundary component,
  in which all tori are very good.
  Recall that our goal is to show that $\rho$ has a finite orbit.  
  Since all tori are very good, we may use Proposition~\ref{prop:SiBonTresBon}
  to find a standard system of generators 
  $a_1, b_1, \ldots, a_{g-1}, b_{g-1}$ where $\rot(a_i) = \rot(b_i) = 0$
  for all $i = 2, 3, \ldots, g-1$. Since $T(a_1, b_1)$ is good,
  we may also assume that $\rot(b_1) = 0$.
  
  Let $\delta_i = a_{i+1}^{-1}b_{i+1}a_{i+1}b_i^{-1}$ as in
  Section~\ref{ssec:Marquage}, so that 
  $(a_1, \delta_1, a_2, \delta_2, \ldots \delta_{g-2}, a_{g-1}, b_{g-1})$
  forms a chain.
  For each $i$, we can use Lemma~\ref{lemma:int_per_vide} in order to assume
  without loss of generality that $\Per(\delta_i)$ has empty interior,
  and then apply Proposition~\ref{prop:MagicProposition} to the chain
  $(\delta_i,a_i,b_i)$. It follows that $\langle\delta_i,b_i\rangle$
  has a finite orbit, hence
  $$\rot(\delta_{i}) + \rot(b_i) = \rot(a_{i+1}^{-1} b_{i+1} a_{i+1}) =
  \rot(b_{i+1}).$$
  Thus, $\rot(\delta_i) = 0$ for all $i$.  
  
  Lemma \ref{lemma:int_per_vide} implies that, after a deformation, we may assume that $\Per(a_1)$ has empty interior.  Thus, we can 
  apply Proposition~\ref{prop:MagicProposition} to the chain $(a_1, \delta_1, a_2, \delta_2, \ldots \delta_{g-2}, a_{g-1}, b_{g-1})$ 
  to conclude that the subgroup generated by these elements has a finite orbit.  As this subgroup is equal to $\pi_1(\Sigma_{g-1,1})$, this proves the proposition.  
\end{proof}

%%%%%%%%%%%%%%%%%%%%%%%%%%%%%%%%%%%

\subsection{Proof of Theorem \ref{theo:Main}}

Theorem \ref{theo:Main} is now a quick consequence of Proposition \ref{prop:OrbiteFinie} and Corollary \ref{cor:AuMoinsG}.  

\begin{proof}[Proof of Theorem \ref{theo:Main}]
Let $\rho: \pi_1(\Sigma_g) \to \Homeo_+(S^1)$ be a path-rigid representation, and suppose that $\rho$ is not geometric. 
If $\Sigma$ contains a bad torus $T$, then by Proposition \ref{prop:NoTwoJustGoodIntro}, $\Sigma \smallsetminus T$ contains only very good tori.   If $\Sigma$ contains no bad torus, but some torus $T'$ that is not very good, then Proposition \ref{prop:NoTwoJustGoodIntro} implies that $\Sigma \smallsetminus T'$ contains only very good tori.   In either case, there is a genus $g-1$ subsurface $\Sigma_{g-1, 1}$ containing only very good tori, hence by Proposition \ref{prop:OrbiteFinie} the restriction of $\rho$ to $\Sigma_{g-1, 1}$ has a finite orbit.  
In particular, the boundary curve of this subsurface has zero rotation number, and the restriction of $\rho$ to this subsurface has relative Euler number zero.

It follows that the Euler number of the remaining (not very good) torus is either 0 or $\pm1$.   By Corollary \ref{cor:AuMoinsG}, if it is $\pm1$, then $\rho$ is geometric.  Thus, the remaining torus has Euler number 0, and by additivity the Euler number of $\rho$ is zero.
\end{proof}

We conclude by noting that if $\Sigma$ has only very good tori, then the proof of Proposition \ref{prop:OrbiteFinie} actually shows that $\rho$ has a finite orbit (hence automatically Euler number zero).

%%%%%%%%%%%%%%%%%%%%%%%%%%%%%%%%%%%

\section{Proof of Theorem~\ref{theo:RigidGeom} and last comments}\label{sec:Fin}

\subsection{Proof of Theorem~\ref{theo:RigidGeom}}
Here is where we use the stronger hypothesis of rigidity. 
Our proof relies on the following observation, 
inspired by work in
the recent article~\cite{ABR}.  

\begin{lemma}\label{lem:WeakABR}
  Let $\rho$ be a rigid, minimal representation. Let $T=T(a,b)$ be a very
  good torus. Then only finitely many points of $S^1$ have a finite
  orbit under $\langle a,b\rangle$. In particular, if $\rot(a)=0$ then
  $P(a,b)$ is a finite set.
\end{lemma}

This lemma is the \emph{only} place where we use rigidity instead of path-rigidity. 

\begin{proof}
  Let $F(a,b)$ denote the set of points whose orbit under $\langle a,b\rangle$
  is finite.  To simplify the exposition of the proof, fix a metric on $S^1$
  so that $a$ and $b$ act on
  $F(a,b)$ by rigid rotations. 
  Given any $\varepsilon>0$, let $J_1, J_2, ... $ denote the (finitely many)
  connected components of $S^1 \smallsetminus F(a,b)$ consisting of intervals
  of length greater than $\varepsilon$ (by our choice of metric, this is a
  $\langle a, b\rangle$-invariant set).
  If $F(a,b)$ is finite, and $\varepsilon$ small enough, then
  $\bigcup_i\overline{J_i} = S^1$. Otherwise (even in the case where
  $\bigcup_i \overline{J_i} = \emptyset$), we may divide
  $S^1 \smallsetminus \bigcup_i \overline{J_i}$ into finitely many disjoint
  open intervals $I_1, I_2, ...$ each of length at most $\varepsilon$ and
  with endpoints in $F(a,b)$, such that these intervals are permuted by
  $\langle a, b\rangle$, and such that
  $S^1 = (\bigcup_i \overline{J_i}) \cup (\bigcup_i \overline{I_i})$.
  
  Since $T$ is very good, we can suppose without loss of generality that
  $\rot(a)=0$. We claim that there exist $a', b' \in \HOS$, agreeing with
  $a$ and $b$ on $S^1 \smallsetminus \bigcup_i I_i$, such that
  $[a', b'] = [a,b]$ holds globally,  and such that
  $\Per(b') \cap \bigcup I_i = \emptyset$.
  
 Let $c = [a,b]$.  As $\bigcup_i I_i$ is $a,b$-invariant, constructing $a'$
  and $b'$ amounts to solving the equation
  $b'c = a'^{-1} b a'$ on $\bigcup_i J_i$. That this can be solved is shown
  in \cite[Lemma~2.7]{EHN}; as their notation and context is slightly
  different, we explain the strategy. Take coordinates identifying each
  $J_i$ with $\R$. If $b'$ is defined on some $J_i$ (with image in $J_j$)
  to increase sufficiently quickly (as a homomorphism $\R \to \R$), then
  $b' c$ will also be strictly increasing, hence conjugate to $b'$.
  One then defines $a'$ to be this conjugacy.
  
  Let $\rho'$ be the representation obtained from $\rho$ by replacing $(a,b)$
  by $(a',b')$. As $\varepsilon>0$ is arbitrary, this $\rho'$ can be taken
  arbitrarily close to $\rho$ in $\Hom(\Gamma_g,\HOS)$.
  Rigidity implies that, for small enough $\varepsilon$, $\rho'$ is
  semi-conjugate to $\rho$. Minimality implies that 
  there is a \emph{continuous} semi-conjugacy $h: S^1 \to S^1$ such that
  $h \circ \rho' = \rho \circ h$. Let
  \[F':=  \{ x \in S^1 \mid x \text{ has finite orbit under } \langle \rho'(a), \rho'(b) \rangle \}. \]
  By construction of $\rho'$, this set is finite.
  However, $h(F') = F(a,b)$. It follows that $F(a,b)$ was finite as well.
\end{proof}

To apply this to the proof of Theorem~\ref{theo:RigidGeom}, let $\rho$ be a rigid,
minimal representation, and assume for contradiction that $\rho$ is
non-geometric. Then Lemma~\ref{lem:WeakABR},
Proposition~\ref{prop:Jump} and Lemma~\ref{lem:FixBad} imply that $\rho$ has no bad tori. In order to derive a contradiction, we will show that all good tori are actually 
very good. We pursue this with an argument in the spirit of Proposition~\ref{prop:DeuxAlorsQuatre}.  

\begin{lemma}\label{lem:SiPerFini1}
  Suppose $P(a,b)=\emptyset$. Then $\partial N(a,b) \subset \partial\Per(a)\cup b^{-1}(\partial\Per(a))$.
\end{lemma}
\begin{proof}
  Assume $P(a,b) = \emptyset$ and let $x \in \partial N(a,b)$.  Since 
  $P(a,b)=\emptyset$, the set $N(a,b)$ is closed, hence
  $x\in N(a,b) \cap \overline{U(a,b)}$.  

  Suppose that
  $x\not\in(\partial\Per(a)\cup b^{-1}(\partial\Per(a))$.
  Then, there exists two intervals, $I,J$, neighborhoods of $x$, with
  $I\subset S^1\smallsetminus\partial\Per(a)$ and
  $J\subset S^1\smallsetminus b^{-1}(\partial\Per(a))$. As $x\in\overline{U(a,b)}$,
  there exists $u\in U(a,b)\cap I\cap J$.
  Let $a_t$ be a positive one-parameter family commuting with $a$.
  Since $b(J)$ contains $b(x)$ and $b(u)$ and $b(J)\cap \partial\Per(a) = \emptyset$, there exists $t_0\in\R$ such that $a_{t_0}b(x)=b(u)$.
  Similarly, there exists $t_1\in\R$ such that $a_{t_1}(u)=x$.
  Thus, $\Delta_{a,b}(x,t_1+T(u),T(u),\ldots, T(u), T(u)+t_0)=0$, and it now follows easily that
  $x\in U(a,b)$. This proves the lemma. 
\end{proof}

\begin{lemma}\label{lem:SiPerFini2}
  Suppose $\rot(a)=0$ and suppose $\langle a,b\rangle$ has no finite orbit.  
  Choose a positive one-parameter group $b_t$ commuting with $b$.
  Then for all $x\in S^1$, there exist at most two values of $t$ such that
  $x\in\partial N(b_t a,b)$.
\end{lemma}

\begin{proof}
  Since $\langle a,b\rangle$ has no finite orbit, $P(a,b)=\emptyset$ and
  hence $P(b_ta, b) = \emptyset$ for all $t$.
  Let $x\in S^1$; we will apply Lemma~\ref{lem:SiPerFini1} to the pairs
  $(b_ta, b)$.
  If $x\in\Per(b)$, then $x \notin N(b_ta,b)$, and in particular
  $x\not\in\partial N(b_ta,b)$ for all $t\in\R$.
  Thus, suppose $x\not\in\Per(b)$.
  
  By Lemma~\ref{lem:SiPerFini1}, if $x\in\partial N(b_t a,b)$, then
  $x\in\partial\Per(b_t a)\cup b^{-1}(\partial\Per(b_t a))$.
  Note that $x$ cannot be in $P(b,a)$, as $x\not\in\Per(b)$. Hence,
  if there exists some $t\in\R$ such that $x\in\Per(b_ta)$,
  then $x\in U(b,a)$, and
  this $t$ is unique. Similarly, if there exists some $t\in\R$ such that
  $b(x)\in\Per(b_ta)$, then $b(x)\in U(b,a)$, and this $t$ is unique.
  This concludes the proof.
\end{proof}

Using these tools, we will now show that $\rho$ (always assumed rigid and minimal) satisfies hypothesis $S_k$.
%We divide the first part of this proof into two lemmas.  

\begin{lemma}\label{lem:Fin1}
  Let $(a,b,c,d)$ be a $4$-chain, and suppose $\rot(a)=\rot(d)=0$ holds.
  Suppose that $T(a,b)$ is good but not very good. Then we have $S_k(b,c)$.
\end{lemma}
\begin{proof}
  By Lemma~\ref{lem:WeakABR}, the set $P(d,c)$ is finite, and using Lemma~\ref{lem:SiPerFini2}, we can
  first deform $a$, to some $b_ta$, so that $\partial N(a,b)$ does not intersect
  $P(d,c)$. Then by Lemma~\ref{lem:DeuxAlorsQuatre}, we have $\Per(b)\cap\Per(c)=\emptyset$.
\end{proof}

\begin{lemma}\label{lem:Fin2}
  Let $(a,b,c,d)$ be a $4$-chain, and suppose $S_k(a,b)$ and $\rot(d)=0$ hold.
  Then we have $S_k(b,c)$.
\end{lemma}
\begin{proof}
  By Lemma~\ref{lem:WeakABR}, the set $P(d,c)$ is finite.
  By Lemma~\ref{lem:PerLabiles} in the torus $T(a,b)$, the set $\Per(b)$
  is disjoint from $P(d,c)$.
  
  Hence, $\Per(b)\subset U(d,c)\cup N(d,c)$, and $\Per(b)$ is finite.
  Thus, for all but finitely many $t$, we have $\Per(b)\cap\Per(d_tc)=\emptyset$.
  Hence $\Per(b)\cap\Per(c)=\emptyset$ by Lemma \ref{lem:PerCommuns}.
\end{proof}

Now we can complete the proof of the Theorem.  

\begin{proof}[Proof of Theorem~\ref{theo:RigidGeom}]
  Let $\rho$ be a rigid, minimal representation. As remarked above,
  $\rho$ has no bad torus.
  If all tori are very good, then as in the proof of Theorem~\ref{theo:Main}, we know
  that $\rho$ admits a finite orbit, a contradiction.
  
  Thus, $\rho$ admits a good torus, $T(a,b)$, which
  is not very good. We may suppose $\rot(a)=0$.
  As all tori are good, we may choose a curve $d$ outside $T(a,b)$ with $\rot(d)=0$,
  and we may form a $4$-chain $(a,b,c,d)$. By Lemma~\ref{lem:Fin1},
  we have $S_k(b,c)$ for some $k$.
  
  Now rename $(b,c)$ into $(a,b)$, and forget about the other curves, remembering only that we have two curves $a, b$ with $S_k(a,b)$.  
  Since all tori are good, we may choose a curve $d$
  outside $T(a,b)$ such that $\rot(d)=0$, and such that
  there exists a standard generating system beginning with $(a,b,d,\gamma)$
  Define $u=\gamma a^{-1}b^{-1}a$ and $v=\gamma a^{-1}$.
  Then $(u,a,b,v)$, $(d,u,a,b)$ and $(a,b,v,d)$
  are $4$-chains (we encourage the reader to refer to Figure~\ref{fig:BaseStandard}
  and draw these curves $u$ and $v$ for him/herself).
  Apply Lemma~\ref{lem:Fin2} to the $4$-chain $(a,b,v,d)$. This
  proves that $S_k(b,v)$ holds. The same lemma applied
  to the $4$ chain $(d,u,a,b)$ implies $S_k(u,a)$.
  Hence, the $4$-chain $(u,a,b,v)$ satisfies $S_k(u,a)$, $S_k(a,b)$ and
  $S_k(b,v)$. We can deform $a$ along $u$, thanks to Lemma~\ref{lem:PerLabiles},
  in such a way that $\Per(a)\cap\Per(v)=\emptyset$, hence we have
  $S_k^+(b,a)$, and we can deform $b$ along $v$, in such a way that
  $\Per(b)\cap\Per(u)=\emptyset$, hence we have $S_k^+(a,b)$. Finally,
  this proves $S_k^{++}(a,b)$, and thus $\rho$ is geometric by
  Corollary~\ref{cor:TwoCurvesSuffice}.
\end{proof}

\subsection{Comments and further questions}
We conclude this paper by discussing some natural questions and directions for further work.

%%%%%%%%%%%%%%%%%%%%%%%%%%%  Path-rigidity  %%%%%%%%%%%%%%%%%%%%

\subsubsection{Path-rigidity}
Given Theorem \ref{theo:Main}, we expect that path-rigiditity should suffice to imply that a representation is geometric.
The most obvious route to this result would be through an improvement of
Lemma~\ref{lem:WeakABR}, as it is the only place where we use the stronger
hypothesis of rigidity.

\begin{question}\label{q:WeakABR} Does Lemma \ref{lem:WeakABR} hold when
  ``rigid'' is replaced by ``path-rigid''? 
\end{question}

This question also arises naturally out of the work of Alonso--Brum--Rivas in \cite{ABR}.
Their main result is the following.  

\begin{theorem}[Alonso--Brum--Rivas \cite{ABR}]
  Let $\rho \in \Hom(\Gamma_g, \HOS)$ or in $\Hom(\Gamma_g, \Homeo^+(\R))$.
  In any neighborhood $U$ of $\rho$, there exists a representation $\rho'$
  without global fixed points.
\end{theorem} 

Since it is unknown whether these representation spaces are locally connected, their result does not imply that there is a {\em path-deformation} of $\rho$ without global fixed points.
Thus, the obvious problem arising out of their work is to upgrade this result to path-deformations.  
A first step in this direction would be to attempt to
reprove \cite[Lemma~3.9,~3.10]{ABR}.
These lemmas show that, in any neighborhood of $\rho$, there exists a
representation $\rho'$ whose fixed points are isolated and either attracting
or repelling points.  Can $\rho'$ be attained by deforming along a path?
If so, can this be generalized to finite orbits, rather than fixed points,
for actions on $S^1$?

%%%%%%%%%%%%%%%%%%%%%%  Commutateur   %%%%%%%%%%%%%%%%%%%

\subsubsection{The commutator equation}

More general than Question~\ref{q:WeakABR} above,
the following basic problem appears to be essential in understanding the topology of $\Hom(\Gamma_g, \HOS)$.  

\begin{problem} \label{prob:nu_h}
For fixed $h \in \HOS$, describe the topology of the set 
$$\nu_h := \{f, g \in \HOS \times \HOS \mid [f,g]=h\}.$$
\end{problem}

As it stands, remarkably little is known about this space.  If $\rot(h) \in \Q \smallsetminus \{0\}$, then it is known that $\nu_h$ is not connected; however, we do not know the number of connected components, nor do we know in any circumstances whether $\nu_h$ is locally connected or not.

This problem is strongly related to Question~\ref{q:CC} on classifying connected components of $\Hom(\Gamma_g, \HOS)$ that we raised in the introduction.  For instance, Goldman's classification of connected components of $\Hom(\Gamma_g, \PSL)$ given in~\cite{Goldman88} is built upon a complete understanding of this space for $\nu_h \cap \PSL \times \PSL$.   This is of course a much easier problem, as $\PSL$ is a finite dimensional Lie group, and the commutator map is smooth.  
The result of the first author in~\cite{KatieInvent} (that Euler number does not classify connected components of $\Hom(\Gamma_g, \HOS)$, unlike the $\PSL$ case) may also serve as warning that the topology of $\nu_h$ should be more complicated than its intersection with $\PSL \times \PSL$.

Throughout this paper, we navigated within $\nu_h$ by making bending deformations. 
This raises a few obvious questions, such as the following.

\begin{question} \label{q:bending}
  Let $h\in\HOS$, and let $(f,g)$ and $(f',g')$ be in the same path-component of
  $\nu_h$. Identifying $f, g$ with the image of generators of a one-holed torus, is there a path
  from $(f,g)$ to $(f',g')$ consisting of a sequence of bending
  deformations? 
  More generally, given $\rho$ and $\rho'$ in the same path-component of $\Hom(\Gamma_g, \HOS)$, is there a path from $\rho$ to $\rho'$ using bending deformations in simple closed curves on $\Sigma_g$? 
\end{question}

This question is reminiscent of Thurston's earthquake theorem for Teichm\"uller space.  It also calls to mind work of Goldman--Xia \cite{GoldmanXia}, who use the analogous (positive) result for bending deformations in connected components of classical character varieties in order to studying the action of the mapping class group on these varieties.
As well as justifying our use of bending deformations alone, a positive answer to Question \ref{q:bending} would give another analogy between classical character varieties and $\chi(\Gamma_g, \HOS)$.

%%%%%%%%%%%%%%%%%%  Bad tori  %%%%%%%%%%%%%%%%%%%%%%%%%

\subsubsection{Bad tori}
In Section \ref{sec:Section5}, we  
used a long series of lemmas to prove that a
path-rigid representation cannot contain two disjoint bad tori.
However, we do not know any example of a path-rigid representation with even
a single bad torus.
Besides being an interesting question in itself, the question of existence bad tori could provide 
a means of showing  path-rigid representations are geometric: if one showed that path-rigid representations of $\Gamma_g$ have no bad tori, an enhanced version of Lemma~\ref{lem:DeuxAlorsQuatre} would complete the proof.
%rephrased last sentence here

However, we were
somewhat surprised to be unable to tackle the following even more basic 
question.

\begin{question}\label{q:BadTorus}
  Let $T(a,b)$ be a one-holed torus.
  Does there exist a representation
  $\rho\colon\pi_1(T) \to \HOS$ such that the rotation number of every
  nonseparating simple closed curve is 
  rational, but nonzero?   
\end{question}
This is obviously related to understanding mapping class group
actions on character varieties, as we are asking for a
nonseparating simple closed curve.

By contrast, relaxing the condition that curves be simple gives a problem already solved by
a classical (though not widely known) result of Antonov.

\begin{theorem}[Antonov \cite{Antonov}]
Let $\rho: \langle a, b \rangle \to \HOS$ be a minimal action.  
Either $\rho$ has abelian image and is conjugate to an action by rotations, or (up to taking a quotient of $S^1$ by a finite order rotation commuting with $\rho$), the probability that the rotation number of the image of a random word of length $N$ in $\{ a, b, a^{-1}, b^{-1}\}$ is zero tends to 1 as $N$ tends to $\infty$.  
\end{theorem}

In the case where $\rho$ commutes with a finite order rotation, say of order $n$, the rotation numbers of random words equidistribute in $\{0, \frac{1}{n}, \ldots, \frac{n-1}{n} \}$.  Thus, for any action, almost all words have rational rotation number.

%%%%%%%%%%%%%%%%%%%%%%%%%%%%%%%%%%   Local rigidity  %%%%%%%%%%%%%%%%%

\subsubsection{Local versus global rigidity} 

Thus far, we have discussed rigidity and path-rigidity of representations; rigidity being the natural notion to study from our interest in character spaces, and path deformations being easier to work with in practice.   
However, from a dynamical perspective, it is also interesting to study {\em local rigidity} or {\em stability} of actions.   

\begin{definition}[3.1 in \cite{KatieSurvey}, see also \cite{ABR}]
A representation $\rho$ is \emph{locally rigid} if it has a neighborhood
{\em in the representation space} $\Hom(\pi_1\Sigma_g,\HOS)$
containing only representations semi-conjugate to $\rho$.
\end{definition} 

In many circumstances, this condition is much easier to satisfy than rigidity or path-rigidity. 
For example,
a savage element $g\in\HOS$ (as in Definition~\ref{def:Intervalles} above),
thought of as a representation of $\Z$,
is easily seen to be locally rigid, but it is semi-conjugate to 
the identity. 
We do not know if this phenomenon generalizes to representations
of~$\Gamma_g$.

\begin{question}
Is there a representation $\rho \in \Hom(\Gamma_g; \HOS)$ that is locally rigid, but not rigid?
\end{question}

Again, a natural first step to this question could be to study the local topology of the sets $\nu_h$ defined above.

\bibliographystyle{plain}

\bibliography{Biblio}

\begin{thebibliography}{10}

\bibitem{ABR}
Juan Alonso, Joaqu{\'\i}n Brum, and Crist\'obal Rivas.
\newblock Orderings and flexibility of some subgroups of {$Homeo_+(\Bbb R)$}.
\newblock {\em J. Lond. Math. Soc. (2)}, 95(3):919--941, 2017.

\bibitem{Antonov}
V.~A. Antonov.
\newblock Modeling of processes of cyclic evolution type. {S}ynchronization by
  a random signal.
\newblock {\em Vestnik Leningrad. Univ. Mat. Mekh. Astronom.}, (vyp. 2):67--76,
  1984.

\bibitem{BIW}
Marc Burger, Alessandra Iozzi, and Anna Wienhard.
\newblock Higher {T}eichm{\"u}ller spaces: from {${\rm SL}(2,\Bbb R)$} to other
  {L}ie groups.
\newblock In {\em Handbook of {T}eichm\"uller theory. {V}ol. {IV}}, volume~19
  of {\em IRMA Lect. Math. Theor. Phys.}, pages 539--618. Eur. Math. Soc.,
  Z\"urich, 2014.

\bibitem{CalegariWalker}
Danny Calegari and Alden Walker.
\newblock Ziggurats and rotation numbers.
\newblock {\em J. Mod. Dynamics}, 5(4):711--746, 2011.

\bibitem{EHN}
David Eisenbud, Ulrich Hirsch, and Walter Neumann.
\newblock Transverse foliations of {S}eifert bundles and self-homeomorphisms of
  the circle.
\newblock {\em Comment. Math. Helv.}, 56(4):638--660, 1981.

\bibitem{Ghys87}
{\'E}tienne Ghys.
\newblock Groupes d'hom\'eomorphismes du cercle et cohomologie born\'ee.
\newblock In {\em The Lefschetz centennial conference, Part III (Mexico City,
  1984)}, volume~{\bf 58} of {\em Contemp. Math.}, pages 81--106. Amer. Math.
  Soc., 1987.

\bibitem{Ghys01}
{\'E}tienne Ghys.
\newblock Groups acting on the circle.
\newblock {\em Enseign. Math. (2)}, 47(3-4):329--407, 2001.

\bibitem{Goldman84}
William~M. Goldman.
\newblock The symplectic nature of fundamental groups of surfaces.
\newblock {\em Adv. in Math.}, 54(2):200--225, 1984.

\bibitem{Goldman88}
William~M. Goldman.
\newblock Topological components of spaces of representations.
\newblock {\em Invent. Math.}, 93(3):557--607., 1988.

\bibitem{GoldmanXia}
William~M. Goldman and Eugene~Z. Xia.
\newblock Ergodicity of mapping class group actions on {${\rm
  SU}(2)$}-character varieties.
\newblock In {\em Geometry, rigidity, and group actions}, Chicago Lectures in
  Math., pages 591--608. Univ. Chicago Press, Chicago, IL, 2011.

\bibitem{Hatcher91}
Allen Hatcher.
\newblock On triangulations of surfaces.
\newblock {\em Topology Appl.}, 40(2):189--194, 1991.

\bibitem{KatokHasselblatt}
Anatole Katok and Boris Hasselblatt.
\newblock {\em Introduction to the modern theory of dynamical systems}.
\newblock Cambridge University Press, 1997.

\bibitem{Luna1}
Domingo Luna.
\newblock Sur certaines op\'erations diff\'erentiables des groupes de {L}ie.
\newblock {\em Amer. J. Math.}, 97:172--181, 1975.

\bibitem{Luna2}
Domingo Luna.
\newblock Fonctions diff\'erentiables invariantes sous l'op\'eration d'un
  groupe r\'eductif.
\newblock {\em Ann. Inst. Fourier (Grenoble)}, 26(1):ix, 33--49, 1976.

\bibitem{KatieInvent}
Kathryn Mann.
\newblock Spaces of surface group representations.
\newblock {\em Invent. Math.}, 201(2):669--710, 2015.

\bibitem{KatieSurvey}
Kathryn Mann.
\newblock Rigidity and flexibility of group actions on {$S^1$}.
\newblock In {\em Handbook of group actions}. 2017.

\bibitem{AuMoinsG}
Kathryn Mann and Maxime Wolff.
\newblock A characterization of {F}uchsian actions by topological rigidity.
\newblock ar{X}iv:1711.05665.

\bibitem{Mod2}
Julien March{\'e} and Maxime Wolff.
\newblock The modular action on {PSL}(2,{R})-characters in genus 2.
\newblock {\em Duke Math. J.}, 165(2):325--359, 2016.

\bibitem{Matsumoto86}
Shigenori Matsumoto.
\newblock Numerical invariants for semiconjugacy of homeomorphisms of the
  circle.
\newblock {\em Proc. AMS}, 98:163--168, 1986.

\bibitem{Matsumoto87}
Shigenori Matsumoto.
\newblock Some remarks on foliated {$S\sp 1$} bundles.
\newblock {\em Invent. Math.}, {\bf 90}:343--358, 1987.

\bibitem{Matsumoto16}
Shigenori Matsumoto.
\newblock Basic partitions and combinations of group actions on the circle: {A}
  new approach to a theorem of {K}athryn {M}ann.
\newblock {\em Enseign. Math.}, 62(1/2):15--47, 2016.

\bibitem{Milnor}
J.W. Milnor.
\newblock On the existence of a connection with curvature zero.
\newblock {\em Comment. Math. Helv.}, 32:215--223, 1958.

\bibitem{Navas}
Andres Navas.
\newblock {\em Groups of Circle Diffeomorphisms}.
\newblock Chicago Lectures in Mathematics. University of Chicago Press, 2011.

\bibitem{Wood}
J.W. Wood.
\newblock Bundles with totally disconnected structure group.
\newblock {\em Comment. Math. Helv.}, 51:183--199, 1971.

\end{thebibliography}

\end{document}